\documentclass[oneside,11pt]{amsart}
\usepackage{amsmath}
\usepackage{amssymb}
\usepackage{graphicx}
\usepackage{caption}
\usepackage{subcaption}
\usepackage{pinlabel}
\usepackage{tikz}

\newtheorem{thm}{Theorem}[section]
\newtheorem{prop}[thm]{Proposition}
\newtheorem{lem}[thm]{Lemma}
\newtheorem{cor}[thm]{Corollary}
\newtheorem{conj}[thm]{Conjecture}

\theoremstyle{definition}
\newtheorem{defn}[thm]{Definition}
\newtheorem{rem}[thm]{Remark}
\newtheorem{conv}[thm]{Convention}

\newtheorem{ques}[thm]{Question}

\newcommand{\abs}[1]{\lvert{#1}\rvert}

\renewcommand{\bar}[1]{\overline{#1}}

\newcommand{\set}[2]{\{\,{#1} \mid {#2} \,\}}
\newcommand{\bigset}[2]{ \bigl\{ \, {#1} \bigm| {#2} \, \bigr\} }
\renewcommand{\emptyset}{\varnothing}

\newcommand{\field}[1]{\mathbb{#1}}

\newcommand{\PP}{\field{P}}
\newcommand{\OO}{\field{O}}

\newcommand{\NN}{\field{N}}
\newcommand{\RR}{\field{R}}

\newcommand{\inclusion}{\hookrightarrow}

\renewcommand{\implies}{\Rightarrow}

\DeclareMathOperator{\CAT}{CAT}

\DeclareMathOperator{\Div}{Div}

\DeclareMathOperator{\diam}{diam}

\begin{document}
\font\myfont=cmr12 at 8.5pt

\title[On strongly quasiconvex subgroups]{On strongly quasiconvex subgroups}

\author{Hung Cong Tran}
\address{Department of Mathematics\\
 The University of Georgia\\
1023 D. W. Brooks Drive\\
Athens, GA 30605\\
USA}
\email{hung.tran@uga.edu}

\date{\today}

\begin{abstract}
We develop a theory of \emph{strongly quasiconvex subgroups} of an arbitrary finitely generated group. Strong quasiconvexity generalizes quasiconvexity in hyperbolic groups and is preserved under quasi-isometry. We show that strongly quasiconvex subgroups are also more reflexive of the ambient groups geometry than the stable subgroups defined by Durham-Taylor, while still having many analogous properties to those of quasiconvex subgroups of hyperbolic groups. We characterize strongly quasiconvex subgroups in terms of the lower relative divergence of ambient groups with respect to them.

We also study strong quasiconvexity and stability in relatively hyperbolic groups, right-angled Coxeter groups, and right-angled Artin groups. We give complete descriptions of strong quasiconvexity and stability in relatively hyperbolic groups and we characterize strongly quasiconvex special subgroups and stable special subgroups of two dimensional right-angled Coxeter groups. In the case of right-angled Artin groups, we prove that two notions of strong quasiconvexity and stability are equivalent when the right-angled Artin group is one-ended and the subgroups have infinite index. We also characterize non-trivial strongly quasiconvex subgroups of infinite index (i.e. non-trivial stable subgroups) in right-angled Artin groups by quadratic lower relative divergence, expanding the work of Koberda-Mangahas-Taylor on characterizing purely loxodromic subgroups of right-angled Artin groups.
\end{abstract}

\subjclass[2000]{%
20F67, 
20F65} 
\maketitle

\section{Introduction}
One method to understand the structure of a group $G$ is to investigate subgroups of $G$. To utilize this method in the study of the geometry of the word metric on finitely generated groups, one needs to investigate subgroups which reflect the geometry of $G$ and are invariant under choices of finite generating set for $G$. Quasiconvex subgroups of hyperbolic groups are a primary example of such subgroups. A quasiconvex subgroup of a hyperbolic group is also hyperbolic. Further, the invariance of quasiconvexity under quasi-isometries between hyperbolic spaces ensures that the quasiconvexity of a subgroup is independent of choice of finite generating set for the ambient group.


Unfortunately, quasiconvexity is not as useful for arbitrary finitely generated groups. In non-hyperbolic spaces, quasiconvexity is not preserved under quasi-isometry. Therefore the quasiconvex subgroups of a non-hyperbolic group $G$ will depend upon the choice of finite generating set for $G$. In \cite{MR3426695}, Durham-Taylor introduce a strong notion of quasiconvexity in finitely generated groups, called \emph{stability}, which is preserved under quasi-isometry, and agrees with quasiconvexity when the ambient group is hyperbolic. Moreover, stable subgroups of mapping class groups are precisely the \emph{convex cocompact subgroups} defined by Farb-Mosher in \cite{MR1914566} and the such subgroups of mapping class groups is a primary motivation for the concept of stable subgroups of arbitrary finitely generated groups (see \cite{MR3426695}). Stable subgroups have many similar properties to quasiconvex subgroups of hyperbolic groups and the study of stable subgroups has received much attention in recent years (see \cite{KMT}, \cite{ADT}, \cite{AMST}, \cite{ABD}, \cite{MR3664526}).


However, a stable subgroup of a finitely generated group is always hyperbolic regardless of the geometry of the ambient group (see \cite{MR3426695}). Thus, the geometry of a stable subgroup does not completely reflect that of the ambient group. Therefore, we introduce another concept of quasiconvexity, called \emph{strong quasiconvexity}, which is strong enough to be preserved under quasi-isometry and relaxed enough to capture the geometry of ambient groups.

\begin{defn}
Let $G$ be a finite generated group and $H$ a subgroup of $G$. We say $H$ is \emph{strongly quasiconvex} in $G$ if for every $K \geq 1,C \geq 0$ there is some $M = M(K,C)$ such that every $(K,C)$--quasi–geodesic in $G$ with endpoints on $H$ is contained in the $M$--neighborhood of $H$.
\end{defn}

Outside hyperbolic setting, there are many strongly quasiconvex subgroups that are not stable. For example, non-hyperbolic peripheral subgroups of a relatively hyperbolic group and non-hyperbolic hyperbolically embedded subgroups of a finitely generated group are non-stable strongly quasiconvex subgroups (see \cite{MR2153979} and \cite{MR3519976}). In Section~\ref{storacgs}, we also provide many examples of non-stable strongly quasiconvex subgroups in non-relatively hyperbolic right-angled Coxeter groups. 

However, there is a strong connection between strong quasiconvexity and stability. More precisely, we prove that a subgroup is stable if and only if it is strongly quasiconvex and hyperbolic (see Theorem \ref{th3}). As a result, our study of strongly quasiconvex subgroup yields new results for stable subgroups of relatively hyperbolic groups, right-angled Coxeter groups, and right-angled Artin groups (See Corollary~\ref{intro_rel_hyp_cor}, Corollary~\ref{b1}, and Theorem~\ref{intro_raags} below). 

In this paper, we prove that strongly quasiconvex subgroups of any finitely generated group have many of the same properties as quasiconvex subgroups in hyperbolic groups. We also characterize a strongly quasiconvex subgroup via the completely super linear lower relative divergence of the ambient group with respect to the subgroup. Finally, we study strongly quasiconvex subgroups of relatively hyperbolic groups, right-angled Artin groups, and right-angled Coxeter groups.

The concept of strongly quasiconvex subgroups was also introduced independently in \cite{Genevois2017} by Genevois under the name \emph{Morse subgroups}. In that paper, Genevois characterizes strongly quasiconvex subgroups in cubulable groups and then he uses strongly quasiconvex subgroups to give a characterization of hyperbolically embedded subgroups in the same group collection. He also independently studies strongly quasiconvex subgroups in right-angled Artin groups and right-angled Coxeter groups.

\subsection{Some properties of strongly quasiconvex subgroups}

A quasiconvex subgroup of a hyperbolic group is always finitely generated, undistorted, and has finite index in its commensurator (see \cite{MR1170365} and \cite{MR1426902}). Further, the intersection between two quasiconvex subgroups of a hyperbolic group is also a quasiconvex subgroup (see \cite{MR1170365}) and any collection of quasiconvex subgroups of a hyperbolic group has finite height, finite width, and bounded packing (see \cite{MR1389776} and \cite{MR2497315}). In this paper, we prove that strongly quasiconvex subgroups of arbitrary finitely generated groups have analogous properties. 

\begin{thm}
\label{vuiqua}
Let $G$ be an arbitrary finitely generated group. Then:
\begin{enumerate}
\item If $H$ is a strongly quasiconvex subgroup of $G$, then $H$ is finitely generated, undistorted subgroup and $H$ has finite index in its commensurator. 
\item If $H_1$ and $H_2$ are arbitrary strongly quasiconvex subgroups of $G$, then $H_1\cap H_2$ is strongly quasiconvex in $G$, $H_1$, and $H_2$.
\item If $\mathcal{H} = \{H_1,\cdots,H_\ell\}$ is a finite collection of strongly quasiconvex subgroups of $G$, then $\mathcal{H}$ has finite height, finite width, and bounded packing. (We refer the reader to Definitions \ref{hw} and \ref{bp} for the concepts of finite height, finite width, and bounded packing of a finite collection of subgroups.) 
\end{enumerate}
\end{thm}

In \cite{AMST}, Antol\'in-Mj-Sisto-Taylor prove that if $\mathcal{H}$ is a finite collection of stable subgroups of a finitely generated group, then $\mathcal{H}$ has finite height, finite width, and bounded packing. The above theorem strengthens their work to strongly quasiconvex subgroups. 
Combining Theorem \ref{vuiqua} with the work of Sageev \cite{MR1347406,MR1438181} and Hruska-Wise \cite{MR2497315}, we have the following as an immediate corollary:

\begin{cor}
Suppose $H$ is a strongly quasiconvex codimension 1 subgroup of a finitely generated group $G$. Then the corresponding $\CAT(0)$ cube complex is finite dimensional.
\end{cor}



In the hyperbolic setting, the inclusion map $H\inclusion G$ of a quasiconvex subgroup $H$ into $G$ induces a topological embedding of the Gromov boundary of $H$ into the Gromov boundary of $G$. The image of this embedding is also the limit set, $\Lambda H$, of $H$ in the Gromov boundary of $G$. Moreover, if $H_1$ and $H_2$ are two quasiconvex subgroups of a hyperbolic group $G$, then the limit set $\Lambda (H_1\cap H_2)$ is the intersection $\Lambda H_1\cap \Lambda H_2$ (see \cite{MR1253544}). 

In order to prove analogies of the above properties for strongly quasiconvex subgroups, we need a quasi-isometry invariant boundary for any finitely generated group. In \cite{MC1}, Cordes introduced the \emph{Morse boundary} of proper geodesic metric spaces which generalizes the contracting boundary of CAT(0) spaces (see \cite{MR3339446}). Critically, the Morse boundary is a quasi-isometry invariant and it agrees with Gromov boundary on hyperbolic spaces (see \cite{MC1}). In this paper, we prove strongly quasiconvex subgroups interact with the Morse boundary similarly to how quasiconvex subgroups of hyperbolic groups interact with the Gromov boundary.

\begin{thm}
\label{intro_Morse_boundary}
Let $G$ be a finitely generated group. Then
\begin{enumerate}
\item If $H$ is a strongly quasiconvex subgroup of $G$, then the inclusion $i:H\inclusion G$ induces a topological embedding $\hat{i}\!:\partial_M H\to \partial_M G$ of the Morse boundary of $H$ into the Morse boundary of $G$ such that $\hat{i}(\partial_M H)=\Lambda H$, where $\Lambda H$ is the limit set of $H$ in the Morse boundary $\partial_M G$ of $G$. 
\item If $H_1$ and $H_2$ are strongly quasiconvex subgroups of $G$, then $H_1\cap H_2$ is strongly quasiconvex in $G$ and $\Lambda (H_1\cap H_2)=\Lambda H_1\cap \Lambda H_2$.
\end{enumerate}
\end{thm}

\subsection{Strong quasiconvexity and lower relative divergence}

\emph{Lower relative divergence} was introduced by the author in \cite{MR3361149} to study geometric connection between a geodesic space and its subspaces. Lower relative divergence generalizes the lower divergence of quasi-geodesic studied in \cite{MR3339446} and \cite{ACGH} to any subset of a geodesic space. Roughly speaking, relative divergence measures the distance distortion of the complement of the $r$--neighborhood of a subspace in the whole space when $r$ increases. Since lower relative divergence is preserved under quasi-isometry (see Proposition 4.9 in \cite{MR3361149}), we can define the lower relative divergence of a pair of groups $(G,H)$, where $G$ is finitely generated and $H\leq G$. We refer the reader to Section \ref{prelim} for precise definitions.


In \cite{MR3361149} the lower relative divergence of a hyperbolic group with respect to a quasiconvex subgroup is at least exponential. For strongly quasiconvex subgroups of arbitrary finitely generated groups, the lower relative divergence must always be completely super linear.

\begin{thm}
\label{intro2}
Let $G$ be a finitely generated group and $H$ infinite subgroup of $G$. Then $H$ is strongly quasiconvex in $G$ if and only if the lower relative divergence of $G$ with respect to $H$ is completely super linear.
\end{thm}

Theorem \ref{intro2} is actually a corollary of the following more general statement about the lower relative divergence of geodesic spaces with respect to Morse subsets.

\begin{thm}
\label{intro1}
Let $X$ be a geodesic space and $A$ a subset of $X$ of infinite diameter. Then $A$ is Morse in $X$ if and only if the lower relative divergence of $X$ with respect to $A$ is completely super linear.
\end{thm}


In \cite{Tran1}, the author proves that the lower relative divergence of a non-hyperbolic group with respect to a stable subgroup can be any polynomials. In this paper, we prove that the same result holds for non-stable strongly quasiconvex subgroups. 

\begin{thm}
\label{intro3}
For each $d\geq 2$ there is a non-hyperbolic group $G_d$ together with a non-stable strongly quasiconvex subgroup whose lower relative divergence is a polynomial of degree $d$.
\end{thm}

The above theorem shows that the lower relative divergence with respect to strongly quasiconvex subgroups are potentially diverse. Since lower relative divergence is a quasi-isometry invariant, Theorem \ref{intro3} suggests that classifying strongly quasiconvex subgroups using lower relative divergence would produce a rich source of new quasi-isomerty invariants.

Theorem \ref{intro1} is a generalization of the work of Charney-Sultan \cite{MR3339446} and Arzhantseva-Cashen-Gruber-Hume \cite{ACGH} where they characterize Morse quasi-geodesics as those which have completely super linear lower divergence. In fact, Charney-Sultan showed in the setting of $\CAT(0)$ spaces, that Morse quasi-geodesic are actually characterized by quadratic lower divergence. Thus Theorem \ref{intro1} suggests the following interesting question

\begin{ques}
Is the lower relative divergence of a $\CAT(0)$ group with respect to an infinite strongly quasiconvex subgroup at least quadratic?
\end{ques}

\subsection{Strong quasiconvexity and stability in relatively hyperbolic groups}

In this paper, we fully describe strongly quasiconvex subgroups and stable subgroups of relatively hyperbolic groups. We remark that strongly quasiconvex subgroups and stable subgroups of relatively hyperbolic groups are distinct from relatively quasiconvex subgroups defined by Dahmani \cite{MR2026551} and Osin \cite{Osin06}. We first come up with the description of strongly quasiconvex subgroups of relatively hyperbolic groups. 

\begin{thm}
Let $(G,\PP)$ be a finitely generated relatively hyperbolic group and $H$ finitely generated undistorted subgroup of $G$. Then the following are equivalent:
\begin{enumerate}
\item The subgroup $H$ is strongly quasiconvex in $G$.
\item The subgroup $H\cap P^g$ is strongly quasiconvex in $P^g$ for each conjugate $P^g$ of peripheral subgroup in $\PP$.
\item The subgroup $H\cap P^g$ is strongly quasiconvex in $G$ for each conjugate $P^g$ of peripheral subgroup in $\PP$. 
\end{enumerate}
\end{thm}

In \cite{ADT}, Aougab-Durham-Taylor characterize stability in relatively hyperbolic groups whose peripheral subgroups are one-ended and have linear divergence. In this paper, we give a complete characterization of stable subgroups of arbitrary relatively hyperbolic groups.

\begin{cor}
\label{intro_rel_hyp_cor}
Let $(G,\PP)$ be a finitely generated relatively hyperbolic group and $H$ finitely generated undistorted subgroup of $G$. Then the following are equivalent:
\begin{enumerate}
\item The subgroup $H$ is stable in $G$.
\item The subgroup $H\cap P^g$ is stable in $P^g$ for each conjugate $P^g$ of peripheral subgroup in $\PP$.
\item The subgroup $H\cap P^g$ is stable in $G$ for each conjugate $P^g$ of peripheral subgroup in $\PP$. 
\end{enumerate}
\end{cor}
 
\subsection{Strong quasiconvexity, stability and the Morse boundary of right-angled Coxeter groups}

In this paper, we establish a characterization of strongly quasiconvex special subgroups of two-dimensional right-angled Coxeter groups. Independent work of Genevois has recently partially expanded this characterization to special subgroups of right-angled Coxeter groups of arbitrary dimension (see Proposition 4.9 in \cite{Genevois2017}).

\begin{thm}
\label{h1}
Let $\Gamma$ be a simplicial, triangle free graph with vertex set $S$ and $K$ subgroup of $G_\Gamma$ generated by some subset $S_1$ of $S$. Then the following conditions are equivalent:
\begin{enumerate}
\item The subgroup $K$ is strongly quasiconvex in $G_\Gamma$.
\item If $S_1$ contains two non-adjacent vertices of an induced 4--cycle $\sigma$, then $S_1$ contains all vertices of $\sigma$.
\item Either $\abs{K}<\infty$ or the lower relative divergence of $G_\Gamma$ with respect to $K$ is at least quadratic. 
\end{enumerate}
\end{thm}

An immediate corollary of the above theorem is a characterization of stable special subgroups of two-dimensional right-angled Coxeter groups.

\begin{cor}
\label{b1}
Let $\Gamma$ be a simplicial, triangle free graph with vertex set $S$ and $K$ subgroup of $G_\Gamma$ generated by some subset $S_1$ of $S$. Then the following conditions are equivalent:
\begin{enumerate}
\item The subgroup $K$ is stable in $G_\Gamma$.
\item The set $S_1$ does not contain a pair of non-adjacent vertices of an induced 4--cycle in $\Gamma$.
\end{enumerate}
\end{cor}

Behrstock produced the first example of $\mathcal{CFS}$ right-angled Coxeter group whose Morse boundary was not totally disconnected \cite{B}. Using Corollary \ref{b1} and Theorem \ref{intro_Morse_boundary}, we prove sufficient condition for the Morse boundary of a two-dimensional right-angled Coxeter group to be not totally disconnected. This provides an alternative proof that the Morse boundary of the example in \cite{B} is not totally disconnected


\begin{cor}
\label{b2}
If the simplicial, triangle free graph $\Gamma$ contains an induced loops $\sigma$ of length greater than $4$ such that the vertex set of $\sigma$ does not contain a pair of non-adjacent vertices of an induced 4--cycle in $\Gamma$, then the Morse boundary of right-angled Coxeter group $G_\Gamma$ is not totally disconnected.
\end{cor}

Corollary \ref{b2} inspires the following conjecture on the connectedness of the Morse boundary of right-angled Coxeter groups.


\begin{conj}
The Morse boundary of a right-angled Coxeter group $G_\Gamma$ is not totally disconnected if and only if the defining graph $\Gamma$ contains an induced loops $\sigma$ of length greater than $4$ such that the vertex set of $\sigma$ does not contain a pair of non-adjacent vertices of an induced 4--cycle in $\Gamma$.
\end{conj}

\subsection{Strong quasiconvexity, stability, and lower relative divergence in right-angled Artin groups}

As discussed above, strong quasiconvexity and stability are equivalent in the hyperbolic setting, but different in general. It is natural to ask about the existence of a non-hyperbolic ambient group in which the two notions of strongly quasiconvex subgroup of infinite index and stable subgroup are equivalent:

\begin{ques}
Is there any non-hyperbolic ambient group setting in which the two notions of strongly quasiconvex subgroup of infinite index and stable subgroup are equivalent? In other words, is there any non-hyperbolic ambient group setting in which all strongly quasiconvex subgroups of infinite index are hyperbolic?
\end{ques}

The following theorem answers the above question:

\begin{thm}
\label{intro_raags}
Let $\Gamma$ be a simplicial, finite, connected graph such that $\Gamma$ does not decompose as a nontrivial join. Let $H$ be a non-trivial, infinite index subgroup of the right-angled Artin group $A_\Gamma$. Then the following are equivalent:
\begin{enumerate}
\item $H$ is strongly quasiconvex.
\item $H$ is stable.
\item The lower relative divergence of $A_\Gamma$ with respect to $H$ is quadratic.
\item The lower relative divergence of $A_\Gamma$ with respect to $H$ is completely super linear.
\end{enumerate}
\end{thm}


In \cite{KMT}, Koberda-Mangahas-Taylor give several characterizations of purely loxodromic subgroups in right-angled Artin groups. One of these is the equivalence between purely loxodromic subgroups and stable subgroups. The above theorem builds on the work of Koberda-Mangahas-Taylor in \cite{KMT} to provide several new characterizations of purely loxodromic subgroups in one-ended right-angled Artin groups.

\begin{cor}
\label{coooool}
Let $\Gamma$ be a simplicial, finite, connected graph such that $\Gamma$ does not decompose as a nontrivial join. Let $H$ be a non-trivial, finitely generated, infinite index subgroup of the right-angled Artin group $A_\Gamma$. Then the following are equivalent:
\begin{enumerate}
\item $H$ is purely loxodromic.
\item $H$ is strongly quasiconvex.
\item The lower relative divergence of $A_\Gamma$ with respect to $H$ is quadratic.
\item The lower relative divergence of $A_\Gamma$ with respect to $H$ is completely super linear.
\end{enumerate}
\end{cor}

Genevois provides an alternative proof for the equivalence among strongly quasiconvex, stable, and purely loxodromic subgroups of one-ended right-angled Artin groups in Theorem B.1 of \cite{Genevois2017}. However, Corollary~\ref{coooool} also provides a classification of the possible lower relative divergences for any infinite index subgroup of a one-ended right-angled Artin group. In \cite{Kim2017}, Kim has recently proved a similar characterizations for convex cocompact subgroups of mapping class groups.

 

\subsection{Some questions}

Hierarchically hyperbolic groups were recently introduced by Behrstock-Hagen-Sisto \cite{BHS} to provide a uniform framework in which to study many important families of groups, including mapping class groups, right-angled Coxeter groups, most 3-manifold groups, and right-angled Artin groups. In \cite{ABD}, Abbott-Behrstock-Durham prove a characterization of stability in hierarchically hyperbolic groups. We expect a similar characterization of strong quasiconvexity in hierarchically hyperbolic groups is possible.


\begin{ques}
Characterize strong quasiconvexity in hierarchically hyperbolic groups.
\end{ques}

We also hope there is a connection between strong quasiconvexity and stability in some hierarchically hyperbolic groups which is analogous to the one demonstrated in Theorem \ref{intro_raags} for right-angled Artin groups.

\begin{ques}
What are conditions on a hierarchically hyperbolic group so that all strongly quasiconvex subgroups of infinite index are stable? 
\end{ques}




Since quasiconvex subgroups of hyperbolic groups are hyperbolic, it is natural to ask about the existence of hierarchically hyperbolic structures on a strongly quasiconvex subgroup of a hierarchically hyperbolic group.

\begin{ques}
Is a strongly quasiconvex subgroup of a hierarchically hyperbolic group hierarchically hyperbolic? 
\end{ques}

Moving on from hierarchically hyperbolic groups, we wonder about further connections between strongly quasiconvex subgroups and hyperbolically embedded subgroups. In particular, Dahmani-Guirardel-Osin and Sisto prove that a hyperbolically embedded subgroup is almost malnormal and strongly quasiconvex (see \cite{DGO} and \cite{MR3519976}) and Genevois shows that the converse holds in the case of cubulable groups \cite{Genevois2017}. We ask about other conditions of a group where one can obtain a converse statement.

\begin{ques}
Under what other conditions of a finitely generated group are all almost malnormal, strongly quasiconvex subgroups hyperbolically embedded?
\end{ques}

In \cite{MR3430352}, Osin characterizes acylindrically hyperbolic groups using hyperbolically embedded subgroups. Therefore, the answer to the above question may give some connection between two notions of strongly quasiconvex subgroups and acylindrically hyperbolic groups.

\subsection*{Acknowledgments}
I want to thank Christopher Hruska, Ruth Charney, Jason Behrstock, Thomas Koberda, Fran\c{c}ois Dahmani, Hoang Thanh Nguyen, Matthew Haulmark, Anthony Genevois, Jacob Russell, and Adam Saltz for their very helpful conversations. I also thank the referee for a careful reading and advice that improved the exposition of the paper.

\section{Preliminaries}
\label{prelim}
In this section, we review some well-known concepts in geometric group theory: geodesic spaces, geodesics, quasigeodesics, quasi-isometry, quasi-isometric embedding, and Morse subsets. We also discuss some recently developed concepts: lower relative divergence, geodesic divergence, and geodesic lower divergence.

\begin{defn}
Let $(X,d_X)$ and $(Y,d_Y)$ be two metric spaces and the map $\Phi$ from $X$ to $Y$ a \emph{$(K,L)$--quasi-isometric embedding} if for all $x_1, x_2$ in $X$ the following inequality holds:\[({1}/{K})\,d_X(x_1,x_2)-L\leq d_Y\bigl(\Phi(x_1),\Phi(x_2)\bigr)\leq K\,d_X(x_1,x_2)+L.\]
If, in addition, $N_L\bigl(\Phi(X)\bigr)=Y$, then $\Phi$ is called a \emph{$(K,L)$--quasi-isometry}. Two spaces $X$ and $Y$ are quasi-isometric if there is a $(K,L)$--quasi-isometry form $X$ to $Y$.

The special case of a quasi-isometric embedding where the domain is a connected interval in $\RR$ (possibly all of $\RR$ is called a \emph{$(K,L)$--quasi-geodesic}. A \emph{geodesic} is a $(1,0)$--quasi-geodesic. The metric space $X$ is a \emph{geodesic space} if any pair of points in $X$ can be joined by a geodesic segment.
\end{defn}


\begin{defn}
Let $X$ be a geodesic space and $A$ a subset of $X$. Let $r$ be any positive number.
\begin{enumerate}
\item $N_r(A)=\bigset{x \in X}{d_X(x, A)<r}$
\item $\partial N_r(A)=\bigset{x \in X}{d_X(x, A)=r}$ 
\item $C_r(A)=X-N_r(A)$.
\item Let $d_{r,A}$ be the induced length metric on the complement of the $r$--neighborhood of $A$ in $X$. If the subspace $A$ is clear from context, we can use the notation $d_r$ instead of using $d_{r,A}$. 
\end{enumerate}
\end{defn}

\begin{defn}
A subset $A$ of a geodesic metric space $X$ is \emph{Morse} if for every $K \geq 1,C \geq 0$ there is some $M = M(K,C)$ such that every $(K,C)$--quasi–geodesic with endpoints on $A$ is contained in the $M$--neighborhood of $A$. We call the function $M$ a \emph{Morse gauge}.
\end{defn}

Before we define the concepts of lower relative divergence, geodesic divergence, and geodesic lower divergence, we need to build the notions of domination and equivalence.

\begin{defn}
Let $\mathcal{M}$ be the collection of all functions from $[0,\infty)$ to $[0,\infty]$. Let $f$ and $g$ be arbitrary elements of $\mathcal{M}$. \emph{The function $f$ is dominated by the function $g$}, denoted \emph{$f\preceq g$}, if there are positive constants $A$, $B$, $C$ and $D$ such that $f(x)\leq Ag(Bx)+Cx$ for all $x>D$. Two functions $f$ and $g$ are \emph{equivalent}, denoted \emph{$f\sim g$}, if $f\preceq g$ and $g\preceq f$. \emph{The function $f$ is strictly dominated by the function $g$}, denoted \emph{$f\prec g$}, if $f$ is dominated by $g$ and they are not equivalent. 

We say a function $f$ in $\mathcal{M}$ is \emph{completely super linear} if for every choice of $C>0$ the collection of $x \in [0,\infty)$ such that $f(x)\leq Cx$ is bounded.
\end{defn}

\begin{rem}
The relations $\preceq$ and $\prec$ are transitive. The relation $\sim$ is an equivalence relation on the set $\mathcal{M}$. If $f\preceq g$ in $\mathcal{M}$ and $f$ is completely super linear, then $g$ is also completely super linear. Therefore, if $f\sim g$ in $\mathcal{M}$ and one of them is completely super linear, then the other is also completely super linear. 

It is clear that two polynomial functions of degree 0 and 1 are always equivalent. Now we let $f$ and $g$ be two polynomial functions of degree at least 1 in the family $\mathcal{M}$. We observe that $f$ is dominated by $g$ if and only if the degree of $f$ is less than or equal to the degree of $g$. Also two functions $f$ and $g$ are equivalent if and only if they have the same degree. All exponential functions of the form $a^{bx+c}$, where $a>1, b>0$ are equivalent. Therefore, a function $f$ in $\mathcal{M}$ is \emph{linear, quadratic or exponential...} if $f$ is respectively equivalent to any polynomial with degree one, two or any function of the form $a^{bx+c}$, where $a>1, b>0$.
\end{rem}

\begin{defn}
Let $\{\delta^n_{\rho}\}$ and $\{\delta'^n_{\rho}\}$ be two families of functions of $\mathcal{M}$, indexed over $\rho \in (0,1]$ and positive integers $n\geq 2$. \emph{The family $\{\delta^n_{\rho}\}$ is dominated by the family $\{\delta'^n_{\rho}\}$}, denoted \emph{$\{\delta^n_{\rho}\}\preceq \{\delta'^n_{\rho}\}$}, if there exists constant $L\in (0,1]$ and a positive integer $M$ such that $\delta^n_{L\rho}\preceq \delta'^{Mn}_{\rho}$. Two families $\{\delta^n_{\rho}\}$ and $\{\delta'^n_{\rho}\}$ are \emph{equivalent}, denoted \emph{$\{\delta^n_{\rho}\}\sim \{\delta'^n_{\rho}\}$}, if $\{\delta^n_{\rho}\}\preceq \{\delta'^n_{\rho}\}$ and $\{\delta'^n_{\rho}\}\preceq \{\delta^n_{\rho}\}$.

We say a family $\{\delta^n_{\rho}\}$ of functions in $\mathcal{M}$ is \emph{completely super linear} if there is some $n_0\geq 3$ such that $\delta^{n_0}_{\rho}$ are completely super linear for all $\rho \in (0,1]$.
\end{defn}

\begin{rem}
A family $\{\delta^n_{\rho}\}$ is dominated by (or dominates) a function $f$ in $\mathcal{M}$ if $\{\delta^n_{\rho}\}$ is dominated by (or dominates) the family $\{\delta'^n_{\rho}\}$ where $\delta'^n_{\rho}=f$ for all $\rho$ and $n$. The equivalence between a family $\{\delta^n_{\rho}\}$ and a function $f$ in $\mathcal{M}$ can be defined similarly. Thus, a family $\{\delta^n_{\rho}\}$ is linear, quadratic, exponential, etc. if $\{\delta^n_{\rho}\}$ is equivalent to the function $f$ where $f$ is linear, quadratic, exponential, etc..

If $\{\delta^n_{\rho}\}\preceq \{\delta'^n_{\rho}\}$ and $\{\delta^n_{\rho}\}$ is completely super linear, then $\{\delta'^n_{\rho}\}$ is also completely super linear. Therefore, if $\{\delta^n_{\rho}\}\sim \{\delta'^n_{\rho}\}$ and one of them is completely super linear, then the other is also completely super linear. 
\end{rem}

We now review the concept of lower relative divergence which was originally introduced by the author in \cite{MR3361149}.

\begin{defn}[Lower relative divergence in spaces]
Let $X$ be a geodesic space and $A$ a subset of $X$ of infinite diameter. For each $\rho \in (0,1]$ and positive integer $n\geq 2$, we define function $\sigma^n_{\rho}\!:[0, \infty)\to [0, \infty]$ as follows: 

For each positive $r$, if there is no pair of $x_1, x_2 \in \partial N_r(A)$ such that $d_r(x_1, x_2)<\infty$ and $d(x_1,x_2)\geq nr$, we define $\sigma^n_{\rho}(r)=\infty$. Otherwise, we define $\sigma^n_{\rho}(r)=\inf d_{\rho r}(x_1,x_2)$ where the infimum is taken over all $x_1, x_2 \in \partial N_r(A)$ such that $d_r(x_1, x_2)<\infty$ and $d(x_1,x_2)\geq nr$. 


The family of functions $\{\sigma^n_{\rho}\}$ is \emph{the lower relative divergence} of $X$ with respect $A$, denoted $div(X,A)$.
\end{defn}

\begin{conv}
Let $X$ be a geodesic space and $A$ a subset of $X$ of infinite diameter. Let $\{\sigma^n_{\rho}\}$ be the lower relative divergence of $X$ with respect to $A$. Assume that $N_{r_0}(A)=X$ for some $r_0$. Therefore, $\partial N_r(A)=\emptyset$ for each $r>r_0$. This implies that for each $\rho \in (0,1]$ and positive integer $n\geq 2$ $\sigma^n_{\rho}(r)=\infty$ for $r>r_0$.
\end{conv}

\begin{defn}[Lower relative divergence in groups] 
Let $G$ be a finitely generated group and $H$ its subgroup. We define \emph{the lower relative divergence} of $G$ with respect to $H$, denoted \emph{$div(G,H)$}, to be the lower relative divergence of the Cayley graph $\Gamma(G,S)$ with respect to $H$ for some (any) finite generating set $S$. 

\end{defn}

\begin{rem}
We first remark that the subgroup $H$ in the above definition is not required to be finitely generated. Moreover, the lower relative divergence is a pair quasi-isometry invariant (see Proposition 4.9 in \cite{MR3361149}). This implies that the lower relative divergence on a finitely generated group does not depend on the choice of finite generating sets of the ambient group. 
\end{rem}

We now define geodesic divergence and geodesic lower divergence. 

\begin{defn}[Geodesic divergence]
The \emph{divergence} of two geodesic rays $\alpha$ and $\beta$ with the same initial point $x_0$ in a geodesic space $X$, denoted $\Div_{\alpha,\beta}$, is a function $g: (0,\infty)\to(0,\infty]$ defined as follows. For each positive $r$, if there is no path outside the open ball with radius $r$ about $x_0$ connecting $\alpha(r)$ and $\beta(r)$, we define $g(r)=\infty$. Otherwise, we define $g(r)$ is the infimum on the lengths of all paths outside the open ball with radius $r$ about $x_0$ connecting $\alpha(r)$ and $\beta(r)$.

The \emph{divergence of a bi-infinite geodesic} $\gamma$, denoted $\Div_{\gamma}$, is the divergence of the two geodesic rays obtained from $\gamma$ with the initial point $\gamma(0)$.
\end{defn}

\begin{defn}[Geodesic lower divergence]
Let $\alpha$ be a bi-infinite geodesic. For any $r>0$ and $t\in \RR$, if there is no path from $\alpha(t-r)$ to $\alpha(t+r)$ that lies outside the open ball of radius $r$ about $\alpha(t)$, we define $\rho_{\alpha}(r,t)=\infty$. Otherwise, we let $\rho_{\alpha}(r,t)$ denote the infimum of the lengths of all paths from $\alpha(t-r)$ to $\alpha(t+r)$ that lies outside the open ball of radius $r$ about $\alpha(t)$. Define the \emph{lower divergence} of $\alpha$ to be the growth rate of the following function: \[ldiv_{\alpha}(r)=\inf_{t\in \RR}\rho_{\alpha}(r,t)\] 
\end{defn}

\begin{rem}
\label{gddv}
The geodesic divergence is a classical notion related to curvature and the geodesic lower divergence was originally introduced by Charney-Sultan in \cite{MR3339446}. In general, these concepts measure different aspects of a bi-infinite geodesic in a geodesic space. More precisely, let $\alpha$ be a bi-infinite geodesic in a geodesic space $X$. It is not hard to see that $ldiv_{\alpha}\leq Div_{\alpha}$ and two functions are not the same in general. However, $ldiv_{\alpha}\sim Div_{\alpha}$ if $\alpha$ is a periodic geodesic (i.e. there is an isometry $g$ of $X$ such that $g\alpha=\alpha$). 

Also, the lower divergence $ldiv_{\alpha}$ of $\alpha$ in $X$ is equivalent to the lower relative divergence $div(X,\alpha)$ of $X$ with respect to $\alpha$. The proof of this fact is similar to the proof of Proposition 6.6 in \cite{MR3361149} and we leave it to the reader. 
\end{rem}

\section{Characterizing Morse subsets via lower relative divergence}
In most parts of this section, we characterize Morse subsets with infinite diameter in terms of the lower divergence related to them. This result will be used to characterize strongly quasiconvex subgroups defined in Section \ref{haha}. We also study the behavior of geodesic rays in a finite neighborhood of some Morse subsets and this result will be applied to study the height of strongly quasiconvex subgroups later.

\begin{prop}[Completely super linear lower relative divergence implies Morse]
\label{p1}
Let $X$ be a geodesic space and $A$ a subset of $X$ of infinite diameter. If the lower relative divergence $\{\sigma^n_{\rho}\}$ of $X$ with respect to $A$ is completely super linear, then $A$ is Morse in $X$.
\end{prop}
\begin{proof}
Let $n_0\geq 3$ and $\rho_0 \in (0,1]$ such that $\sigma^{n_0}_{\rho_0}$ is a completely super linear function. Let $\gamma\!:[a,b]\to X$ be an arbitrary $(K,L)$--quasi-geodesic in $X$ connecting two points in $A$. By Lemma 1.11 of \cite{MR1744486} III.H, we can find a continuous quasi-geodesic $c\!:[a,b]\to X$ such that:
\begin{enumerate}
\item $c(a)=\gamma(a)$ and $c(b)=\gamma(b)$.
\item $\ell(c|_{[t,t']})\leq K_1 d\bigl(c(t),c(t')\bigr)+L_1$, where $K_1\geq 1$, $L_1\geq 0$ only depend on $K$ and $L$.
\item The Hausdorff distance between $\gamma$ and $c$ is less than $C$, where $C$ only depends on $K$ and $L$.
\end{enumerate}

Therefore, it is sufficient to prove that $c$ lies in some $M$--neighborhood of $A$, where $M$ only depends on $K_1$, $L_1$, $n_0$ and the function $\sigma^{n_0}_{\rho_0}$. Let $t_0 \in [a,b]$ such that $c(t_0)$ is the farthest point from $A$ and let $m=d\bigl(c(t_0),A\bigr)$. If $m\leq 2L_1$, then $c$ obviously lies in some $M$--neighborhood of $A$, where $M$ only depends on $L_1$. Therefore, we can assume that $m>2L_1$. We claim that if $m>2L_1$, then $\sigma^{n_0}_{\rho_0}\bigl(\frac{m}{2K_1n_0}\bigr)\leq (4K_1+2)m$. Let $m_1=\frac{m}{2K_1n_0}$, $t_1\in [a,t_0]$ and $t_2\in [t_0,b]$ such that $c(t_1)$, $c(t_2)$ both lie in $\partial N_{m_1}(A)$ and $c[t_1,t_2]$ lies outside $N_{m_1}(A)$. We consider two cases. 

\textbf{Case 1:} $\ell(c|_{[t_1,t_2]})\leq 4K_1m$

Since $d\bigl(c(t_0), A)=m$ and $d\bigl(c(t_1), A)=\frac{m}{2K_1n_0}$, then the length $\ell(c|_{[t_1,t_0]})$ is bounded below $\bigl(1-\frac{1}{2K_1n_0}\bigr)m$. Similarly, the length $\ell(c|_{[t_0,t_2]})$ is bounded below $\bigl(1-\frac{1}{2K_1n_0}\bigr)m$. Therefore, the length $\ell(c|_{[t_1,t_2]})$ is bounded below by $\bigl(2-\frac{1}{K_1n_0}\bigr)m$. This implies that

\begin{align*}
d\bigl(c(t_1),c(t_2)\bigr)&\geq \frac{1}{K_1}\ell(c|_{[t_1,t_2]})-\frac{L_1}{K_1}\\&\geq \frac{1}{K_1}\bigl(2-\frac{1}{K_1n_0}\bigr)m-\frac{L_1}{K_1}\\&\geq \frac{m}{K_1}+\bigl(1-\frac{1}{K_1n_0}\bigr)\bigl(\frac{m}{K_1}\bigr)-\frac{L_1}{K_1}\\&\geq \frac{m}{K_1}+\frac{m}{2K_1}-\frac{L_1}{K_1}\geq \frac{m}{K_1}\geq n_0\bigl(\frac{m}{2K_1n_0}\bigr)
\end{align*}

Also, the path $c([t_1,t_2])$ lies outside $\bigl(\frac{m}{2K_1n_0}\bigr)$--neighborhood of $A$. Therefore,
\[\sigma^{n_0}_{\rho_0}\bigl(\frac{m}{2K_1n_0}\bigr)\leq \ell(c|_{[t_1,t_2]})\leq 4K_1m.\] 

\textbf{Case 2:} $\ell(c|_{[t_1,t_2]})> 4K_1m$

We can choose $t_3$ and $t_4$ in $[t_1,t_2]$ ($t_3<t_4$) such that $\ell(c|_{[t_3,t_4]})= 4K_1m$. Since $m_1\leq d\bigl(c(t_3),A\bigr)\leq m$, we can choose a point $u$ in $\partial N_{m_1}(A)$ such that the geodesic $\alpha_1$ connecting $u$, $c(t_3)$ lies outside $N_{m_1}(A)$ and its length is bounded above by $m$. Similarly, we can choose a point $v$ in $\partial N_{m_1}(A)$ such that the geodesic $\alpha_2$ connecting $v$, $c(t_4)$ lies outside $N_{m_1}(A)$ and its length is bounded above by $m$. Let $\beta=\alpha_1\bigcup c[t_3,t_4]\bigcup \alpha_2$. Then $\beta$ is a path outside $N_{m_1}(A)$ that connects two points $u$ and $v$ on $\partial N_{m_1}(A)$. Moreover, the length of $\beta$ is bounded above by $(4K_1+2)m$ by the construction. Also,

\begin{align*}
d(u,v)&\geq d\bigl(c(t_3), c(t_4)\bigr)-2m\geq \frac{1}{K_1}\ell(c|_{[t_3,t_4]})-\frac{L_1}{K_1}-2m\\&\geq 4m-\frac{L_1}{K_1}-2m\geq 2m-\frac{L_1}{K_1}\geq m\geq n_0\bigl(\frac{m}{2K_1n_0}\bigr).
\end{align*}

Therefore,
\[\sigma^{n_0}_{\rho_0}\bigl(\frac{m}{2K_1n_0}\bigr)\leq \ell(\beta)\leq (4K_1+2)m.\] 

Since $\sigma^{n_0}_{\rho_0}$ is a completely super linear function, then there is an upper bound on $m$ depending on $K_1$ and $n_0$. Therefore, $c$ lies in some $M$--neighborhood of $A$, where $M$ only depends on $K_1$, $L_1$, $n_0$ and the function $\sigma^{n_0}_{\rho_0}$.
\end{proof} 

The following two lemmas will be used to compute the lower relative divergence of geodesic spaces with respect to their Morse subsets. Most techniques for the proof of these lemmas were studied in \cite{ACGH}. 

\begin{lem}
\label{l1}
Suppose $\gamma$ is a concatenation of two $(L,0)$--quasi-geodesics $\gamma_1$ and $\gamma_2$, where $\gamma_1$, $\gamma_2$, and $\gamma$ are all parametrized by arc length. Let $C>1$ and $r>0$ such that the distance between two endpoints $a$, $c$ of $\gamma$ is at least $r$ and the length of $\gamma$ is at most $Cr$. For each $\rho \in (0,1]$ and $L'>L+C+C/\rho+1$, there exist an $(L',0)$--quasi-geodesic $\alpha$ parametrized by arc length connecting two points $a$, $c$ such that the image of $\alpha$ lies $\rho r$--neighborhood the image of $\gamma$ and the length of $\alpha$ is bounded above by the length of $\gamma$. 
\end{lem}
\begin{proof}
For each path $\beta$ parametrized by arc length we will abuse notation to identify $\beta$ to its image. Moreover, for each pair of points $u$ and $v$ on $\beta$ we denote $d_\beta(u,v)$ to be the length of the subpath of $\beta$ connecting $u$ and $v$. If $\gamma$ is an $(L',0)$--quasi-geodesic, then $\alpha=\gamma$ is a desired quasi-geodesic. We now assume that $\gamma$ is not an $(L',0)$--quasi-geodesic. 

Since $d(a,c)\geq r\geq d_\gamma(a,c)/C>d_\gamma(a,c)/L'$, then there is a maximal subsegment $[x,y]_\gamma$ of $\gamma$ such that $d(x,y)=d_{\gamma}(x,y)/L'$ by the continuity of $\gamma$. It is obvious that $x$ and $y$ can not both lie in the same path $\gamma_1$ or $\gamma_2$. Therefore, we can assume $x\in \gamma_1$ and $y\in \gamma_2$. Let $[x,y]$ be a geodesic connecting $x$, $y$ and $\alpha=[a,x]_{\gamma_1}\cup [x,y] \cup [y,c]_{\gamma_2}$. We parametrized $\alpha$ by arc length and we will prove that $\alpha$ is an $(L',0)$--quasi-geodesic.

Let $u$ and $v$ be an arbitrary two points in $\alpha$. If $u$ and $v$ both lie in the same one of these segments $[a,x]_{\gamma_1}$, $[x,y]$, or $[y,c]_{\gamma_2}$, then $d_{\alpha}(u,v)\leq L' d(u,v)$. We now consider the remaining three cases.

\textbf{Case 1:} One of these points $u$, $v$ lies in $[a,x]_{\gamma_1}$ and the other lies in $[x,y]$. 

We can assume that $u$ lies in $[a,x]_{\gamma_1}$ and $v$ lies in $[x,y]$. We observe that $d(u,y)>d_{\gamma}(u,y)/L'$ by the maximality of $[x,y]_\gamma$ and the continuity of $\gamma$. Therefore,

\begin{align*}
d(u,v)&\geq d(u,y)-d(v,y)\geq \frac{d_{\gamma}(u,y)}{L'}-d(v,y)=\\&=\frac{d_{\gamma}(u,x)+d_{\gamma}(x,y)}{L'}-d(v,y)=\frac{d_{\gamma}(u,x)}{L'}+d(x,y)-d(v,y)=\\&=\frac{d_{\gamma}(u,x)}{L'}+d(x,v)\geq \frac{d_{\gamma}(u,x)+d(x,v)}{L'}=\frac{d_{\alpha}(u,v)}{L'}
\end{align*}

or $d_{\alpha}(u,v)\leq L'd(u,v)$.

\textbf{Case 2:} One of these points $u$, $v$ lies in $[y,c]_{\gamma_2}$ and the other lies in $[x,y]$. Using symmetric argument as in Case 1, we can prove $d_{\alpha}(u,v)\leq L'd(u,v)$.

\textbf{Case 3:} One of these points $u$, $v$ lies in $[a,x]_{\gamma_1}$ and the other lies in $[y,c]_{\gamma_2}$. We can assume that $u$ lies in $[a,x]_{\gamma_1}$ and $v$ lies in $[y,c]_{\gamma_2}$. We observe again that $d(u,v)>d_{\gamma}(u,v)/L'$ by the maximality of $[x,y]_\gamma$ and the continuity of $\gamma$. Also, it is not hard to see $d_{\alpha}(u,v)\leq d_{\gamma}(u,v)$. This implies that $d(u,v)>d_{\alpha}(u,v)/L'$ or $d_{\alpha}(u,v)\leq L'd(u,v)$.

 Therefore, $\alpha$ is an $(L',0)$-quasi-geodesic parametrized by arc length connecting $a$ and $c$. It is not hard to see the length of $\alpha$ is less than or equal to the length of $\gamma$. In order to see that the path $\alpha$ lies completely inside $\rho r$--neighborhood of $\gamma$, we observe that \[d(x,y)=\frac{d_{\gamma}(x,y)}{L'}< \frac{Cr}{C/\rho}=\rho r.\] 
\end{proof}

\begin{lem}
\label{l2}
For each $C>1$ and $\rho \in (0,1]$ there is a constant $L=L(C,\rho)\geq 1$ such that the following holds. Let $r$ be an arbitrary positive number and $\gamma$ a continuous path with the length less than $Cr$. Assume the distance between two endpoints $x$ and $y$ is at least $r$. Then there is an $(L,0)$--quasi-geodesic $\alpha$ connecting two points $x$, $y$ such that the image of $\alpha$ lies in the $\rho r$--neighborhood of $\gamma$ and the length of $\alpha$ is less than or equal to the length of $\gamma$.
\end{lem}

\begin{proof}
Let $x_0=x$ and $x_1$ be the last point on $\gamma$ such that $d(x_0,x_1)=\rho r/4$. Similarly, we define $x_i=y$ if $d(x_{i-1},y)<\rho r/2$ or to be the last point on $\gamma$ such that $d(x_{i-1},x_i)=\rho r/4$. By construction, we observe that $d(x_i,x_j)\geq\rho r/4$ if $i\neq j$ and $y=x_n$ for some $n\leq 4C/\rho$. Let $\alpha_1$ be a concatenation of geodesics $[x_0,x_1][x_1,x_2]\cdots[x_{n-1},x_n]$. Then $\alpha_1$ connects two points $x$, $y$, the length of $\alpha_1$ is less than or equal to the length of $\gamma$, and $\alpha_1$ lies completely inside $(\rho r/2)$--neighborhood of $\gamma$.

Using Lemma \ref{l1} for each integer $1\leq i\leq \lfloor n/2\rfloor$, there are $(L_2,0)$--quasi-geodesics (where $L_2$ only depends on $C$ and $\rho$) from $x_{2(i-1)}$ to $x_{2i}$ such that the concatenation $\alpha_2$ of these with subpath $[x_{2\lfloor n/2\rfloor}, x_n]_{\alpha_1}$ of $\alpha_1$ satisfying the following conditions:
\begin{enumerate}
\item $\alpha_2$ lies completely inside $(\rho r/4)$--neighborhood of $\alpha_1$.
\item The length of $\alpha_2$ is less than or equal to the length of $\alpha_1$.
\end{enumerate} 

Repeating this process at most $d$ times ($d\leq n$) until we get an $(L_d,0)$--quasi-geodesic $\alpha_d$ (where $L_d$ only depends on $C$ and $\rho$) connecting two points $x$ and $y$ satisfying the following conditions:
\begin{enumerate}
\item $\alpha_d$ lies completely inside $(\rho r/2^d)$--neighborhood of $\alpha_{d-1}$.
\item The length of $\alpha_d$ is less than or equal to the length of $\alpha_{d-1}$.
\end{enumerate}

Therefore, the length of the $(L_d,0)$--quasi-geodesic $\alpha_d$ is less than the length of $\gamma$. Also,
\[\frac{\rho r}{2}+\frac{\rho r}{4}+\cdots+\frac{\rho r}{2^d}<\rho r.\]
This implies that $\alpha_d$ lies in the $\rho r$--neighborhood of $\gamma$. Therefore, $L=L_d$ is the desired number only depending on $C$ and $\rho$.
\end{proof}

\begin{prop}[Morse implies completely super linear lower relative divergence]
\label{p2}
Let $X$ be a geodesic space and $A$ a subset of $X$ of infinite diameter. If $A$ is Morse in $X$, then the lower relative divergence $\{\sigma^n_{\rho}\}$ of $X$ with respect to $A$ is completely super linear.
\end{prop}

\begin{proof}
We will prove the stronger result that $\sigma^n_{\rho}$ is completely super linear for all $n\geq 5$ and $\rho \in (0,1]$. Let $\mu$ be the Morse gauge of $A$ and assume for the contradiction that $\sigma^{n_0}_{\rho_0}$ is not completely super linear for some $n_0\geq 5$ and $\rho_0 \in (0,1]$. Then there is $C>0$ for which there is an unbounded sequence of numbers $r_m$ and paths $\gamma_m$ satisfying
\begin{enumerate}
\item Each path $\gamma_m$ lies outsides the $(\rho_0r_m)$--neighborhood of $A$ and their endpoints $x_m$, $y_m$ lie in $\partial N_{r_m}(A)$.
\item $d(x_m,y_m)\geq n_0r_m$ and $\ell(\gamma_m)\leq Cr_m$.
\end{enumerate} 

For each $m$, let $\gamma'_m$ be a geodesic of length less than $2r_m$ that connects $x_m$ and some point $u_m$ in $A$. Similarly, let $\gamma''_m$ be a geodesic of length less than $2r_m$ that connects $y_m$ and some point $v_m$ in $A$. Let $\bar{\gamma}_m=\gamma'_m\cup\gamma_m\cup\gamma''_m$. Then $\bar{\gamma}_m$ is a path of length at most $(C+4)r_m$ that connects two points $u_m$, $v_m$ on $A$ with $d(u_m,v_m)\geq (n_0-4)r_m\geq r_m$. 

By Lemma \ref{l2}, for each $m$ there is an $(L,0)$--quasi-geodesic $\alpha_m$ ($L$ does not depend on $m$) connecting two points $u_m$, $v_m$ satisfying
\begin{enumerate}
\item $\ell(\alpha_m)\leq \ell(\bar{\gamma}_m)\leq (C+4)r_m$.
\item $\alpha_m$ lies in the $(\rho_0 r_m/2)$--neighborhood of $\bar{\gamma}_m$.
\end{enumerate}

By triangle inequality, each $(\rho_0 r_m/2)$--neighborhood of $\gamma'_m$ is contained in the open ball $B(x_m,2r_m+\rho_0 r_m/2)$ and $(\rho _0r_m/2)$--neighborhood of $\gamma''_m$ is contained in the open ball $B(y_m,2r_m+\rho_0 r_m/2)$. Since $d(x_m,y_m)\geq n_0r_m\geq 5 r_m$, the intersection between two open balls $B(x_m,2r_m+\rho_0 r_m/2)$ and $B(y_m,2r_m+\rho_0 r_m/2)$ is empty. Therefore, two $(\rho_0 r_m/2)$--neighborhoods of $\gamma'_m$ and $\gamma''_m$ have an empty intersection. Since $\alpha_m$ is connected and $\alpha_m$ intersects with both $(\rho_0 r_m/2)$--neighborhoods of $\gamma'_m$ and $\gamma''_m$, $\alpha_m$ can not lie completely inside the $(\rho_0 r_m/2)$--neighborhood of $\gamma'_m\cup\gamma''_m$. This implies that there is a point $s_m$ on $\alpha_m$ that lies in the $(\rho_0 r_m/2)$--neighborhood of $\gamma_m$. Since each path $\gamma_m$ lies outside the $(\rho_0 r_m)$--neighborhood of $A$, the point $s_m$ must lie outside the $(\rho_0 r_m/2)$--neighborhood of $A$. 

Recall that each $\alpha_m$ is an $(L,0)$--quasi--geodesic with endpoints on $A$ and $A$ is Morse in $X$ with Morse gauge $\mu$. Each path $\gamma_m$ must lie completely inside the $\mu(L,0)$--neighborhood of $A$. In particular, each point $s_m$ must lie inside the $\mu(L,0)$--neighborhood of $A$. Therefore, $\mu(L,0)\geq \rho_0r_m/2$ for all $m$ which contradicts to the choice of sequence $r_m$. Therefore, $\sigma^n_{\rho}$ is completely super linear for all $n\geq 5$ and $\rho \in (0,1]$.

\end{proof}

The following theorem is deduced from Propositions \ref{p1} and \ref{p2}.

\begin{thm}[Characterization of Morse subsets]
\label{th1}
Let $X$ be a geodesic space and $A$ a subset of $X$ of infinite diameter. Then $A$ is Morse in $X$ if and only if the lower relative divergence of $X$ with respect to $A$ is completely super linear.
\end{thm}

We now study the behavior of geodesic rays in a finite neighborhood of some Morse subset. We first start with a technical lemma that basically amounts to ``quasigeodesic approximation'' of some certain concatenations of geodesics. 

\begin{lem}
\label{cvcvcv1}
Let $\gamma=\gamma_1\gamma_2\gamma_3$ be the concatenation of three geodesics $\gamma_1$, $\gamma_2$, and $\gamma_3$. If there is a number $r>0$ such that $\ell(\gamma_1)<r$, $\ell(\gamma_2)=20r$, and $\ell(\gamma_3)<r$, then there is a $(5,0)$--quasigeodesic $\alpha$ with the same endpoints of $\gamma$ and $\alpha\cap\gamma_2\neq \emptyset$.
\end{lem}

\begin{proof}
For each path $\beta$ parametrized by arc length we will abuse notation to identify $\beta$ to its image. Moreover, for each pair of points $u$ and $v$ on $\beta$ we denote $d_\beta(u,v)$ to be the length of the subpath of $\beta$ connecting $u$ and $v$. Let $a$ and $b$ be the endpoints of $\gamma_1$, let $b$ and $c$ be the endpoints of $\gamma_2$, and let $c$ and $d$ be the endpoints of $\gamma_3$. We claim that there are points $x_1\in\gamma_1$ and $y_1\in\gamma_2$ such that $d(x_1,y_1)\leq 5r$ and the concatenation $\alpha_1=[a,x_1]_{\gamma_1}\cup [x_1,y_1] \cup [y_1,c]_{\gamma_2}$ is a $(5,0)$--quasigeodesic, where $[x_1,y_1]$ is a geodesic connecting $x_1$, $y_1$.

If the concatenation $\beta=\gamma_1\gamma_2$ is a $(5,0)$--quasigeodesic, then we can choose $x_1=y_1=b$ and $\alpha_1=\beta$. We now assume that $\beta$ is not a $(5,0)$--quasigeodesic. Since $d(a,c)\geq d(b,c)-d(a,b)\geq 19r$ and $d_\beta(a,c)\leq d(a,b)+d(b,c)\leq 21r$, $d(a,c)>d_\beta(a,c)/5$. Therefore, there is a maximal subsegment $[x_1,y_1]_\beta$ of $\beta$ such that $d(x_1,y_1)=d_{\beta}(x_1,y_1)/5$ by the continuity of $\beta$. It is obvious that $x_1$ and $y_1$ can not both lie in the same path $\gamma_1$ or $\gamma_2$. Therefore, we can assume $x_1\in \gamma_1$ and $y_1\in \gamma_2$. Obviously, $d(x_1,y_1)=d_{\beta}(x_1,y_1)/5\leq \ell(\beta)/5\leq 5r$. Moreover, $\alpha_1$ is a $(5,0)$--quasigeodesic by using the same argument as in Lemma \ref{l1}.

Similarly, there are points $x_2\in\gamma_3$ and $y_2\in\gamma_2$ such that $d(x_2,y_2)\leq 5r$ and the concatenation $\alpha_2=[b,y_2]_{\gamma_2}\cup [y_2,x_2] \cup [x_2,d]_{\gamma_3}$ is a $(5,0)$--quasigeodesic, where $[y_2,x_2]$ is a geodesic connecting $y_2$, $x_2$. 
Let \[\alpha=[a,x_1]_{\gamma_1}\cup [x_1,y_1] \cup [y_1,y_2]_{\gamma_2}\cup [y_2,x_2]\cup [x_2,d]_{\gamma_3}.\] Then $\alpha\subset \alpha_1\cup\alpha_2$ and $\alpha\cap \gamma_2\neq \emptyset$. We now prove that $\alpha$ is a $(5,0)$--quasigeodesic. Let $u$ and $v$ be arbitrary two points in $\alpha$. We consider several cases.

If two points $u$ and $v$ both lies in $\alpha_1$, then $d_{\alpha_1}(u,v)\leq 5d(u,v)$ because $\alpha_1$ is a $(5,0)$--quasigeodesic. Also, $d_\alpha(u,v)=d_{\alpha_1}(u,v)$ by the construction of $\alpha$ and $\alpha_1$. This implies that $d_{\alpha}(u,v)\leq 5d(u,v)$. We also obtain a similar inequality if two points $u$ and $v$ both lies in $\alpha_2$. In the remaining case, $u$ and $v$ does not lie in the same path $\alpha_1$ or $\alpha_2$. Therefore, one of two points $u$ and $v$ must lie in $[a,x_1]_{\gamma_1}\cup [x_1,y_1]$ (say $u$) and the remaining point must lie in $[y_2,x_2]\cup [y_2,d]_{\gamma_3}$ (say $v$). Since the length of $[a,x_1]_{\gamma_1}\cup [x_1,y_1]$ is less than $6r$, $d(a,u)\leq 6r$. Similarly, $d(d,v)\leq 6r$. Therefore,
\begin{align*}
d(u,v) &\geq d(a,d)-d(a,u)-d(d,v)\\&\geq d(b,c)-d(a,b)-d(c,d)-12r\\&\geq 20r-r-r-12r\geq 6r.
\end{align*}
Also, $d_\alpha(u,v)\leq \ell(\alpha)\leq \ell(\gamma)\leq 22r$. Therefore, $d_\alpha(u,v)\leq 5 d(u,v)$. Therefore, $\alpha$ is a $(5,0)$--quasigeodesic.

\end{proof}

\begin{prop}
\label{vqvqvq1}
Let $X$ be a geodesic space and $A$ a Morse subset of $X$ with the Morse gauge $\mu$. Then there is a constant $D>0$ such that the following hold. If $\gamma:[0,\infty)\to X$ is a geodesic ray that lies in some finite neighborhood of $A$, then there is $C>0$ such that $\gamma_{|[C,\infty)}$ lies in the $D$--neighborhood of $A$.
\end{prop}

\begin{proof}
Assume that $\gamma$ lies in the $r$--neighborhood of $A$ for some $r>0$. For each $i\geq 1$ let $\gamma_i=\gamma_{|[20(i-1)r,20ir]}$. Then each $\gamma_i$ is a geodesic of length $20r$. Since $\gamma$ lies in the $r$--neighborhood of $A$, for each $i\geq 0$ there is a geodesic $\beta_i$ of length less than $r$ that connects $\gamma(20ir)$ and $A$. By Lemma \ref{cvcvcv1} for each $i\geq 1$ there is a $(5,0)$--quasigeodesic $\alpha_i$ with the same endpoints with the concatenation $\beta_{i-1}\gamma_i\beta_i$ such that $\alpha_i\cap\gamma_i\neq \emptyset$. 

Since each $\alpha_i$ is a $(5,0)$--quasigeodesic with endpoints in $A$, $\alpha_i$ lies entirely inside the $\mu(5,0)$--neighborhood of $A$. We recall that $\alpha_i\cap\gamma_i\neq \emptyset$. Therefore, there is $t_i\in [20(i-1)r,20ir]$ such that $\gamma(t_i)$ lies in the $\mu(5,0)$--neighborhood of $A$. Therefore, for each $i\geq 1$ each geodesic $\gamma_{|[t_i,t_{i+1}]}$ has endpoints in the $\mu(5,0)$--neighborhood of $A$. This implies that $\gamma_{|[t_i,t_{i+1}]}$ lies in some $D$--neighborhood of $A$, where $D$ only depends on $\mu$. Set $C=t_1$ then $\gamma_{|[C,\infty)}$ lies in the $D$--neighborhood of $A$ obviously.
\end{proof}

\section{Strong quasiconvexity, stability, and lower relative divergence}
\label{haha}

In this section, we propose the concept of strongly quasiconvex subgroups. We show the connection between strongly quasiconvex subgroups and stable subgroups and we characterize these subgroups using lower relative divergence. We also study some basic properties of strongly quasiconvex subgroups that are analogous to properties of quasiconvex subgroups in hyperbolic groups. 
 
\begin{defn}
Let $\Phi\!:\!A\to X$ be a quasi-isometric embedding between geodesic metric spaces. We say $A$ is \emph{strongly quasiconvex} in $X$ if the image $\Phi(A)$ is Morse in $X$. We say $A$ is \emph{stable} in $X$ if for any $K\geq 1$, $L \geq 0$ there is an $R=R(K,L)\geq 0$ so that if $\alpha$ and $\beta$ are two $(K,L)$--quasi-geodesics with the same endpoints in $\Phi(A)$, then the Hausdorff distance between $\alpha$ and $\beta$ is less than $R$. 
\end{defn}

Note that when we say $A$ is strongly quasiconvex (stable) in $X$ we mean that $A$ is strongly quasiconvex (stable) in $X$ with respect to a particular quasi-isometric embedding $\Phi\!:\!A\to X$. Such a quasi-isometric embedding will always be clear from context, for example an undistorted subgroup $H$ of a finitely generated group $G$. The following proposition provides some basic facts about strong quasiconvexity.

\begin{prop}
\label{pp1}
Suppose that $A$, $B$, $X$ are geodesic metric spaces and $B \overset{f}{\rightarrow} A \overset{g}{\rightarrow} X$ are quasi-isometric embeddings. 
\begin{enumerate}
\item If $B$ is strongly quasiconvex in $X$ via $g\circ f$, then $B$ is strongly quasiconvex in $A$ via $f$.
\item If $B$ is strongly quasiconvex in $A$ via $f$ and $A$ is strongly quasiconvex in $X$ via $g$, then $B$ is strongly quasiconvex in $X$ via $g\circ f$. 
\end{enumerate}
\end{prop}

\begin{proof}
We first prove Statement (1) of the above proposition. More precisely, we are going to prove that $f(B)$ is Morse in $A$. For each $K\geq 1$ and $L\geq 0$ let $\gamma$ be an arbitrary $(K,L)$--quasi-geodesic in $A$ that connects two points in $f(B)$. Then $g(\gamma)$ is $(K',L')$--quasigeodesic in $X$ that connects two points in $g\bigl(f(B)\bigr)$, where $K'$, $L'$ depend only on $K$, $L$, and the quasi-isometry embedding $g$. Since $B$ is strongly quasiconvex in $X$ via $g\circ f$, then $g(\gamma)$ lies in some $M$--neighborhood of $g\bigl(f(B)\bigr)$, where $M$ depend only on $K'$ and $L'$. Again, $g$ is a quasi-isometry embedding. Therefore, $\gamma$ lies in some $M'$--neighborhood of $f(B)$, where $M'$ depends only on $K$, $L$, $f$ and the quasi-isometry embedding $g$. This implies that $B$ is strongly quasiconvex in $A$ via $f$.

We now prove Statement (2) of the above proposition. More precisely, we are going to prove that $(g\circ f)(B)$ is Morse in $X$. For each $K\geq 1$ and $L\geq 0$ let $\alpha$ be an arbitrary $(K,L)$--quasi-geodesic in $X$ that connects two points in $g\bigl(f(B)\bigr)$. Since $g\bigl(f(B)\bigr)\subset g(A)$ and $g(A)$ is Morse subset in $X$, then $\alpha$ lies in some $M$--neighborhood of $g(A)$, where $M$ depends only on $K$ and $L$. Also, $g\!:A\to X$ is a quasi-isometry embedding. Then there is a $(K_1,L_1)$--quasigeodesic $\beta$ in $A$ that connects two points in $f(B)$ such that the Hausdorff distance between $g(\beta)$ and $\alpha$ is bounded above by $C$, where $K_1$, $L_1$, $C$ depend only on $K$, $L$, $M$, and the quasi-isometry embedding $g$. Also, $f(B)$ is a Morse subset of $A$. Then $\beta$ lies in some $M_1$--neighborhood $f(B)$, where $M_1$ depends only on $K_1$ and $L_1$. Again, $g\!:A\to X$ is a quasi-isometry embedding. Then $g(\beta)$ lies in some $M_2$--neighborhood of $g\bigl(f(B)\bigr)$, where $M_2$ depends only on $M_1$ and the map $g$. Therefore, $\alpha$ lies in the $(M_2+C)$--neighborhood of $g\bigl(f(B)\bigr)$. Therefore, $B$ is strongly quasiconvex in $X$ via $g\circ f$.
\end{proof}

The following proposition gives the exact relationship between strongly quasiconvex and stable subspaces.

\begin{prop}
\label{p3}
Let $\Phi\!:\!A\to X$ be a quasi-isometric embedding between geodesic metric spaces. Then $A$ is stable in $X$ if and only if $A$ is strongly quasiconvex and hyperbolic.
\end{prop}

\begin{proof}
One direction of the above proposition is deduced from Remark 3.1 and Lemma 3.3 in \cite{MR3426695}. We now prove that if $A$ is strongly quasiconvex and hyperbolic, then $A$ is stable in $X$. For any $K\geq 1$, $L \geq 0$ let $\alpha$ and $\beta$ are two $(K,L)$--quasi-geodesics with the same endpoints in $\Phi(A)$. Since $\Phi(A)$ is Morse in $X$, there is a constant $C\geq 0$ not depending on $\alpha$ and $\beta$ such that $\alpha$ and $\beta$ both lies in the $C$--neighborhood of $\Phi(A)$. This is an easy exercise that there are constants $K_1\geq 1$, $L_1\geq 0$, $D\geq 0$ not depending on $\alpha$ and $\beta$ and two $(K_1,L_1)$--quasi-geodesics $\alpha_1$ and $\beta_1$ with the same endpoints in $A$ such that the Hausdorff distance between $\Phi(\alpha_1)$, $\alpha$ and the Hausdorff distance between $\Phi(\beta_1)$, $\beta$ are both bounded above by $D$. Since $A$ is a hyperbolic space and $\Phi$ is a quasi-isometric embedding, the Hausdorff distance between $\Phi(\alpha_1)$ and $\Phi(\beta_1)$ is bounded above by some constant $D_1$ which does not depend on $\alpha$ and $\beta$. Therefore, the Hausdorff distance between $\alpha$ and $\beta$ is bounded above by $D_1+2D$. Therefore, $A$ is stable in $X$.
\end{proof}

We now define the concepts of strongly quasiconvex subgroups and stable subgroups.

\begin{defn}
Let $G$ be a finite generated group and $S$ an arbitrary finite generating set of $G$. Let $H$ be a finite generated subgroup of $G$ and $T$ an arbitrary finite generating set of $H$. The subgroup $H$ is \emph{undistorted} in $G$ if the natural inclusion $i\!:H\to G$ induces a quasi-isometric embedding from the Cayley graph $\Gamma(H,T)$ into the Cayley graph $\Gamma(G,S)$. We say $H$ is \emph{stable} in $G$ if $\Gamma(H,T)$ is stable in $\Gamma(G,S)$. 

We remark that stable subgroups were proved to be independent of the choice of finite generating sets (see Section 3 in \cite{MR3426695}).
\end{defn}

\begin{defn}
Let $G$ be a finite generated group and $H$ a subgroup of $G$. We say $H$ is \emph{quasiconvex} in $G$ with respect to some finite generating set $S$ of $G$ if there exists some $C>0$ such that every geodesic in the Cayley graph $\Gamma(G,S)$ that connects a pair of points in $H$ lies inside the $C$--neighborhood of $H$. We say $H$ is \emph{strongly quasiconvex} in $G$ if $H$ is a Morse subset in the Cayley graph $\Gamma(G,S)$ for some (any) finite generating set $S$.
\end{defn}

\begin{rem}
If $H$ is a quasiconvex subgroup of a group $G$ with respect to some finite generating set $S$, then $H$ is also finitely generated and undistorted in $G$ (see Lemma 3.5 of \cite{MR1744486} III.$\Gamma$). However, we emphasize that the concept of quasiconvex subgroups depend on the choice of finite generating set of the ambient group.

It is clear that if $H$ is a Morse subset in the Cayley graph $\Gamma(G,S)$ with some finite generating set $S$, then $H$ is also a quasiconvex subgroup of $G$ with respect to $S$. In particular, $H$ is finitely generated and undistorted in $G$. Therefore, the strong quasiconvexity of a subgroup does not depend on the choice of finite generating sets by Statement (1) in Proposition \ref{pp1}. Moreover, if a finitely generated group $G$ acts properly and cocompactly in some space, then $H$ is a strongly quasiconvex (stable) subgroup of $G$ if and only if $H$ is strongly quasiconvex (stable) in $X$ via some (any) orbit map restricted on $H$. 
\end{rem}


The following theorem is a direct consequence of Theorem \ref{th1}

\begin{thm}[Characterizing strongly quasiconvex subgroups]
\label{th2}
Let $G$ be a finitely generated group and $H$ infinite subgroup of $G$. Then $H$ is strongly quasiconvex in $G$ if and only if the lower relative divergence of $G$ with respect to $H$ is completely super linear.
\end{thm}

We also obtain several characterizations of stable subgroups by the following theorem and the proof can be deduced from Proposition \ref{p3} and Theorem \ref{th2}. 

\begin{thm}[Characterizing stable subgroups]
\label{th3}
Let $G$ be a finitely generated group and $H$ infinite subgroup of $G$. Then the following are equivalent.
\begin{enumerate}
\item $H$ is stable in $G$.
\item $H$ is hyperbolic and strongly quasiconvex in $G$
\item $H$ is hyperbolic and the lower relative divergence of $G$ with respect to $H$ is completely super linear.
\end{enumerate}
\end{thm}

We now study some basic results about strongly quasiconvex subgroups related to subgroup inclusion and subgroup intersection.

\begin{lem} (Proposition 9.4, \cite{Hruska10}).
\label{ll1}
Let $G$ be a group with a finite generating set $S$. Suppose $xH$ and $yK$ are arbitrary left cosets of subgroups of $G$. For each constant $L$ there is a constant $L'=L'(G, S, xH, yK, L)$ so that in the metric space $(G, d_S)$ we have
\[N_L(xH) \cap N_L(yK) \subset N_{L'}(xHx^{-1} \cap yKy^{-1}).\]
\end{lem}

\begin{prop}
\label{pp2}
Let $G$ be a finitely generated group, $A$ a subgroup of $G$, and $B$ a subgroup of $A$. Then
\begin{enumerate}
\item If $A$ is finitely generated, undistorted in $G$ and $B$ is strongly quasiconvex in $G$, then $B$ is strongly quasiconvex in $A$.
\item If $B$ is strongly quasiconvex in $A$ and $A$ is strongly quasiconvex in $G$, then $B$ is strongly quasiconvex in $G$
\end{enumerate}
\end{prop}

\begin{proof}
We note that if subgroups $A$ is undistorted in $G$, then $B$ is undistorted in $A$ if and only if $B$ is undistorted in $G$. Therefore, the above proposition is a direct result of Proposition \ref{pp1}.
\end{proof}

We remark that Statement (1) in Proposition \ref{pp2} can be strengthened by the following proposition.

\begin{prop}
\label{pp3}
Let $G$ be a finitely generated group and $A$ undistorted subgroup of $G$. If $H$ is a strongly quasiconvex subgroup of $G$, then $H_1=H\cap A$ is a strongly quasiconvex subgroup of $A$. In particular, $H_1$ is finitely generated and undistorted in $A$.
\end{prop}

\begin{proof}
We fix finite generating sets $S$ and $T$ for $G$ and $A$ respectively. Let $f\!:\Gamma(A,T)\to\Gamma(G,S)$ be a quasi-isometry embedding which is an extension of the inclusion $A\inclusion G$. We will prove that $H_1$ is a Morse subset of $\Gamma(A,T)$. For each $K\geq 1$ and $L\geq 0$ let $\alpha$ be an arbitrary $(K,L)$--quasi-geodesic in $\Gamma(A,T)$ that connects two points in $H_1$. Since $f\!:\Gamma(A,T)\to\Gamma(G,S)$ is a quasi-isometry embedding which is an extension of the inclusion $A\inclusion G$, then $f(\alpha)$ is a $(K',L')$--quasigeodesic in $\Gamma(G,S)$ that connects two points in $H_1$ and $f(\alpha)$ lies in some $D$--neighborhood of $A$ in $\Gamma(G,S)$, where $K'$, $L'$, $D$ depend only on $K$, $L$ and the map $f$. Since $H$ is a Morse subset in $\Gamma(G,S)$, $f(\alpha)$ also lies in some $D_1$--neighborhood of $H$ in $\Gamma(G,S)$, where $D_1$ depends only on $K'$ and $L'$. By Lemma \ref{ll1}, there is $M=M(G,S,A,H,D,D_1)$ such that $f(\alpha)$ lies in the $M$--neighborhood of $H_1=H\cap A$ in $\Gamma(G,S)$. Again, $f\!:\Gamma(A,T)\to\Gamma(G,S)$ is a quasi-isometry embedding which is an extension of the inclusion $A\inclusion G$. Then, $\alpha$ lies in some $M'$--neighborhood of $H_1$ in $\Gamma(A,T)$, where $M'$ depend only on $M$ and the map $f$.
\end{proof}

The following proposition is a direct result of Propositions \ref{pp2} and \ref{pp3}.

\begin{prop}
\label{prp1}
Let $G$ be a finitely generated group and $H_1$, $H_2$ strongly quasiconvex subgroups of $G$. Then $H_1\cap H_2$ is strongly quasiconvex in $H_1$, $H_2$, and $G$.
\end{prop}

Quasiconvex subgroups in hyperbolic groups have the well-known property of finite height and therefore, they have finite index in their commensurators. We now show that strongly quasiconvex subgroups in general also have these properties.

\begin{defn}
\label{hw}
Let $G$ be a group and $H$ a subgroup.
\begin{enumerate}
\item Conjugates $g_1Hg_1^{-1}, \cdots g_kHg_k^{-1} $ are \emph{essentially distinct} if the cosets $g_1H,\cdots,g_kH$ are distinct.
\item $H$ has height at most $n$ in $G$ if the intersection of any $(n+1)$ essentially distinct conjugates is finite. The least $n$ for which this is satisfied is called the height of $H$ in $G$.
\item The \emph{width} of $H$ is the maximal cardinality of the set $\{g_iH:\abs{g_iH{g_i}^{-1} \cap g_jH{g_j}^{-1}}=\infty\}$, where $\{g_iH\}$
ranges over all collections of distinct cosets.
\end{enumerate}
Similarly, given a finite collection $\mathcal{H} =\{H_1,\cdots,H_l\}$.
\begin{enumerate}
\item Conjugates $g_1H_{\sigma(1)}g_1^{-1}\cdots g_kH_{\sigma(k)}g_k^{-1}$ are \emph{essentially distinct} if the cosets $g_1H_{\sigma(1)},\cdots,g_kH_{\sigma(k)}$ are distinct. 
\item The finite collection $\mathcal{H}$ of subgroups of $G$ has height at most $n$ if the intersection of any $(n+1)$ essentially distinct conjugates is finite. The least $n$ for which this is satisfied is called the height of $\mathcal{H}$ in $G$.
\item The \emph{width} of $\mathcal{H}$ is the maximal cardinality of the set $$\{g_{\sigma(i)}H_{\sigma(i)}:\abs{g_{\sigma(i)}H_{\sigma(i)}g^{-1}_{\sigma(i)}\cap g_{\sigma(j)}H_{\sigma(j)}g^{-1}_{\sigma(j)}}=\infty\},$$ where $\{g_{\sigma(i)}H_{\sigma(i)}\}$ 
ranges over all collections of distinct cosets.
\end{enumerate}
\end{defn}

\begin{defn}
Let $G$ be a group and $H$ a subgroup. The \emph{commensurator} of $H$ in $G$, denoted $Comm_G(H)$ is defined as
\[Comm_G(H) = \set{g \in G}{[H: H \cap gHg^{-1}] <\infty, [gHg^{-1}: H \cap gHg^{-1}]<\infty}.\]
\end{defn}

\begin{thm}
\label{lacloi1}
Let $\mathcal{H} = \{H_1,\cdots,H_\ell\}$ be a finite collection of strongly quasiconvex subgroups of a finitely generated group $G$. Then $\mathcal{H}$ has finite height.
\end{thm}

\begin{proof}
By Proposition \ref{vqvqvq1} there is a constant $D>0$ such that the following hold. If $\gamma:[0,\infty)\to \Gamma(G,S)$ is a geodesic ray in the Cayley graph $\Gamma(G,S)$ that lies in some finite neighborhood of $H_i$, then there is $C>0$ such that $\gamma_{|[C,\infty)}$ lies in the $D$--neighborhood of $H_i$. Let $g_1H_{\sigma(1)}g_1^{-1}\cdots g_kH_{\sigma(k)}g_k^{-1}$ be essential distinct conjugates with infinite intersection $\bigcap g_iH_{\sigma(i)}g_i^{-1}$. 

Since the Cayley graph $\Gamma(G,S)$ is proper and $\bigcap g_iH_{\sigma(i)}g_i^{-1}$ is an infinite strongly quasiconvex subgroups, there is a geodesic ray $\gamma:[0,\infty)\to \Gamma(G,S)$ in the Cayley graph $\Gamma(G,S)$ that lies in some finite neighborhood of $\bigcap g_iH_{\sigma(i)}g_i^{-1}$. Therefore, for each $i$ the ray $\gamma$ lies in some finite neighborhood of $g_iH_{\sigma(i)}$. This implies that the ray $g_i^{-1}\gamma$ lies in some finite neighborhood of $H_{\sigma(i)}$. Thus, there is some $C_i>0$ such that $g_i^{-1}\gamma_{|[C_i,\infty)}$ lies in the $D$--neighborhood of $H_{\sigma(i)}$. In other word, $\gamma_{|[C_i,\infty)}$ lies in the $D$--neighborhood of $g_iH_{\sigma(i)}$. Therefore, all left cosets $g_iH_{\sigma(i)}$ intersect some ball $B(a,D)$. But there is a uniform bound $N$ on the number of cosets of subgroups in $\mathcal{H}$ intersecting any metric ball of radius $D$. Thus, the collection $\mathcal{H}$ has finite height. 
\end{proof}

\begin{cor}
If $H$ is an infinite strongly quasiconvex subgroup of a finitely generated group $G$, then $H$ has finite index in its commensurator $Comm_G(H)$.
\end{cor}

We now prove that a finite collection of strongly quasiconvex subgroups has bounded packing and finite width. We first recall the concept of bounded packing subgroups in \cite{MR2497315}.

\begin{defn}
\label{bp}
Let $G$ be a finitely generated group and $\Gamma$ a Cayley graph with respect to a finite generating set. A subgroup $H$ has \emph{bounded packing} in $G$ if, for all $D \geq 0$, there exist $N \in \NN$ such that for any collection of $N$ distinct cosets $gH$ in $G$, at least two are separated by a distance of at least $D$.

Similarly, a finite collection $\mathcal{H}= \{H_1,\cdots,H_{\ell}\}$ of subgroups of $G$ has \emph{bounded packing} if, for all $D \geq 0$, there exist $N \in \NN$ such that for any collection of $N$ distinct cosets $gH_i$ with $H_i \in \mathcal{H}$, at least two are separated by a distance of at least $D$.
\end{defn}

\begin{lem}
\label{dacbiet}
Let $(X,d)$ be a geodesic space. Let $A, B$, and $C$ be Morse subsets of $X$ with a Morse gauge $\mu$. For each $D>0$ there is a number $D'$ depending on $D$ and $\mu$ such that if $A, B$, and $C$ are pairwise $D$-close, then the intersection $N_{D'}(A)\cap N_{D'}(B) \cap N_{D'}(C)$ is not empty.
\end{lem}

\begin{proof}
Let \begin{align*}
     x\in N_D(A)\cap N_D(B)\\y\in N_D(B)\cap N_D(C)\\z\in N_D(C)\cap N_D(A). 
    \end{align*}
Let $u$ be a point in a geodesic $[y,z]$ connecting $y$ and $z$ such that $d(x,u)=d\bigl(x,[y,z]\bigr)$. Let $\gamma_1$ be the concatenation of two geodesics $[x,u]$ and $[u,y]$. Similarly, let $\gamma_2$ be the concatenation of two geodesics $[x,u]$ and $[u,z]$. Then $\gamma_1$ and $\gamma_2$ are both $(3,0)$--quasigeodesic. The proof of this claim is quite elementary and the reader can also see the proof of this claim in the proof of Lemma 2.2 in \cite{MC1}. Since $\gamma_1$, $\gamma_2$, and $[y,z]$ are all $(3,0)$--quasigeodesic and $A, B$, and $C$ are all $\mu$--Morse subsets of $X$, then there is a number $D'$ depending on $D$ and $\mu$ such that 
\begin{align*}
     \gamma_1 \subset N_{D'}(B)\\ \gamma_2\subset N_{D'}(A)\\ [y,z]\subset N_{D'}(C) 
    \end{align*}		
In particular, $u$ belongs to the intersection $N_{D'}(A)\cap N_{D'}(B) \cap N_{D'}(C)$. Therefore, this intersection is not empty.		
\end{proof}

\begin{lem}[Lemma 4.2 in \cite{MR2497315}]
\label{4.2HW}
Suppose $H\leq G$ has height $0 <n<\infty$ in $G$. Choose $g \in G$ so that $gH\neq H$, and let $K=H\cap gHg^{-1}$. Then $K$ has the height less than $n$ in $H$.
\end{lem}

\begin{thm}
\label{buocdem}
Let $H$ be a strongly quasiconvex subgroup of a finitely generated group $G$. Then $H$ has bounded packing in $G$.
\end{thm}

The proof of the above theorem follows the same line of argument as Theorem 4.8 in \cite{MR2497315}. At some point in the proof of Theorem 4.8 in \cite{MR2497315}, Hruska-Wise need to use the thinness of a triangle in a hyperbolic space. However, due to the lack of hyperbolicity in group $G$ and subgroup $H$ in the above theorem, we need to use Lemma \ref{dacbiet} instead. 

\begin{proof}
Fix finite generating sets $S$ and $T$ for $G$ and $H$ respectively. By Theorem \ref{lacloi1}, we know that the height of $H$ in $G$ is finite. We will prove the theorem by induction on height. If the height of $H$ in $G$ is zero, then $H$ is a finite group. Therefore, $H$ has bounded packing by Corollary 2.6 in \cite{MR2497315}. We now assume by induction that the theorem holds for every finitely generated group $G_0$ and strongly quasiconvex subgroup $H_0$ with the height of $H_0$ in $G_0$ less than the height of $H$ in $G$. 

Let $\mathcal{H}$ be a set of left cosets $gH$ which are pairwise $D$-close. We are going to prove that the cardinality of $\mathcal{H}$ is bounded by a number depending on $D$. We can translating $\mathcal{H}$ so we may assume that $H \in \mathcal{H}$. Observe that if $d(gH,H)<D$ then $gH=hxH$ for some $h\in H$ and $x\in B(1,D)$. It follows that the left cosets $gH$ intersecting $N_D(H)$ lie in at most $\abs{B(1,D)}$ distinct $H$--orbits. Thus it suffices to bound the number of elements of $\mathcal{H}$ in the orbit $H(gH)$ for each fixed $g\notin H$.

If we let $K=H\cap gHg^{-1}$, then Lemma \ref{4.2HW} shows that the height of $K$ in $H$ is less than the height of $H$ in $G$. Since $K$ is a strongly quasiconvex subgroup $H$, the inductive hypothesis applied to $K \leq H$ gives for each $D'$ a number $M'= M'(D')$ so that any
collection of $M'$ distinct cosets $hK$ in $H$ contains a pair separated by a $d_T$--distance at least $D'$. Furthermore, the proof of Lemma 4.2 in \cite{MR2497315} shows that there is a well-defined map $hgH \to hK$ taking left cosets of $H$ in the orbit of $gH$ to left cosets of $K$. A similar argument shows that this map is bijective.

In order to complete the proof, we will show that $D$--closeness of distinct cosets $h_1gH$ and $h_2gH$ in $(G,d_S)$ implies $D'$--closeness of the corresponding cosets $h_1K$ and $h_2K$ in $(H,d_T)$ for some $D'$ depending on $D$. The proof of this step in Theorem 4.8 in \cite{MR2497315} requires the thinness of a triangle. However, we need to use Lemma \ref{dacbiet} for our situation. In fact, there is $D_1$ depending on $D$ such that $N_{D_1}(H)\cap N_{D_1}(h_1gH) \cap N_{D_1}(h_2gH)$ contains an element $u$. By Lemma \ref{ll1}, there is number $D_2$ depending on $D_1$ but is independent of the choice of $h_i\in H$ such that $$N_{D_1}(H)\cap N_{D_1}(h_igH)\subset N_{D_2}(h_iK).$$ Therefore, $u\in N_{D_2}(h_1K)\cap N_{D_2}(h_2K)$. In other word, $d_S(h_1K,h_2K)<2D_2$. Since $H$ is an undistorted subgroup of $G$, then $d_T(h_1K,h_2K)<D'$ for some $D'$ depending on $D$, as desired.
\end{proof}

\begin{thm}
Let $\mathcal{H}=\{H_1,\cdots,H_{\ell}\}$ be a finite collection of strongly quasiconvex subgroups of a finitely generated group $G$. Then $\mathcal{H}$ has bounded packing in $G$. Further, $\mathcal{H}$ has finite width.
\end{thm}


\begin{proof}
First, let $\mathcal{M}$ be a collection of coset of subgroups in $\mathcal{H}$ which are pairwise $D$--close. By Theorem \ref{buocdem}, for each $i=1,2,\cdots,\ell$ there is a $N_i\geq 0$ such that the number of coset of $H_i$ in $\mathcal{M}$ is at most $N_i$. Therefore, the number of elements of $\mathcal{M}$ is at most $\Sigma N_i$. Thus, $\mathcal{H}$ has bounded packing.

Arguing as in the proof of Theorem \ref{lacloi1}, we see that any two conjugates of subgroups in $\mathcal{H}$ with infinite intersection have cosets uniformly close together. Therefore, $\mathcal{H}$ has finite width.
\end{proof}

\section{Morse boundaries and strong quasiconvexity}
\label{bvsq}

We recall basic properties of a quasiconvex subgroup in a hyperbolic group: the subgroup is also a hyperbolic group and the subgroup inclusion induces a topological embedding of the Gromov boundary of the subgroup into the Gromov boundary of the ambient group. Moreover, the image of this topological embedding is identical to the limit set of the subgroup inside the ambient group. Also, the intersection of two quasiconvex subgroups is again a quasiconvex subgroup and the limit set of the intersection of two quasiconvex subgroups is equal to the intersection of limit sets of the two subgroups. In this section, we prove some analogous properties for strongly quasiconvex subgroups in general. 

We first review the concept of Morse boundary in \cite{MC1}. 

\begin{defn}
Let $\mathcal{M}$ be the set of all Morse gauges. We put a partial ordering on $\mathcal{M}$ so that for two Morse gauges $N, N' \in \mathcal{M}$, we say $N \leq N'$ if and only if $N(K,L) \leq N'(K, L)$ for all $K, L$.
\end{defn}

\begin{defn}
Let $X$ be a proper geodesic space. The \emph{Morse boundary} of $X$ with basepoint $p$, denoted $\partial_M {X_p}$, is defined to be the set of all equivalence classes of Morse geodesic rays in $X$ with initial point $p$, where two rays $\alpha,\alpha'\!:[0,\infty)\rightarrow X$ are equivalent if there exists a constant $K$ such that $d_X\bigl(\alpha(t), \alpha'(t)\bigr) < K$ for all $t > 0$. We denote the equivalence class of a ray $\alpha$ in $\partial_M {X_p}$ by $[\alpha]$.

On $\partial_M {X_p}$, we could build a topology as follows:

Consider the subset of the Morse boundary 
\[\partial_M^N {X_p}=\set{x}{\text{The class $x$ contains at least one $N$--Morse geodesic rays $\alpha$ with $\alpha(0)=p$}}.\]

We define convergence in $\partial_M^N {X_p}$ by: $x_n \rightarrow x$ as $n \rightarrow \infty$ if and only if there exists $N$--Morse geodesic rays $\alpha_n$ with $\alpha_n(0) = p$ and $[\alpha_n] = x_n$ such that every subsequence of ${\alpha_n}$ contains a subsequence that converges uniformly on compact sets to a geodesic ray $\alpha$ with $[\alpha]=x$. The closed subsets $F$ in $\partial_M^N {X_p}$ are those satisfying the condition
\[\bigl[\{x_n\} \subset F \text{ and } x_n \rightarrow x\bigr] \implies x\in F.\]
We equip the the Morse boundary $\partial_M {X_p}$ with the direct limit topology $$\partial_M {X_p}=\lim_{\overrightarrow{\mathcal{M}}} \partial_M^N {X_p}.$$

Let $A$ be a subset of $X$ with basepoint $p$. The \emph{limit set} $\Lambda A$ of $A$ in $\partial_M {X_p}$ is the set of all points $c$ in $\partial_M {X_p}$ represented by rays based at $p$ that lie in some finite neighborhood of $A$. 
\end{defn}

\begin{rem}
The direct limit topology on $\partial_M {X_p}$ is independent of the basepoint $p$ (see Proposition 3.5 in \cite{MC1}). Therefore, we can assume the
basepoint is fixed, suppress it from the notation and write $\partial_M {X}$. Moreover, the Morse boundary is a quasi-isometry invariant (see Proposition 3.7 \cite{MC1}). Therefore, we define the \emph{Morse boundary} of a finitely generated group $G$, denoted $\partial_M G$, as the Morse boundary of its Cayley graph. We also define the \emph{limit set} of a subgroup $H$ of $G$ in $\partial_M G$ accordingly.

We define \emph{an action} of $G$ on $\partial_M G$ as follows. For each element $g$ in $G$ and $[\alpha]$ in $\partial_M G$, $g[\alpha]=[\beta]$, where $\alpha$ and $\beta$ are two rays at the basepoint in some Cayley graph of $G$ such that the Hausdorff distance between $g\alpha$ and $\beta$ is finite.
\end{rem}

\begin{defn}
Let $X$ and $Y$ be proper geodesic metric spaces and $p \in X$, $p' \in Y$.
We say that $f\!:\partial_M {X_p} \rightarrow \partial_M {Y_{p'}}$ is \emph{Morse preserving} if given $N$ in $\mathcal{M}$ there exists
an $N'$ in $\mathcal{M}$ such that $f$ injectively maps $\partial_M^N {X_p}$ to $\partial_M^{N'} {Y_{p'}}$.
\end{defn}

The following proposition shows that if a quasi-isometric embedding defines a strongly quasiconvex subspace then it induces a Morse preserving map. This proposition is a key lemma for the proof of the fact that inclusion map of a strongly quasiconvex subgroup into a finitely generated group induces a topological embedding on Morse boundaries.
 
\begin{prop}
\label{ppp1}
Let $X$, $A$ be two proper geodesic spaces and $A$ is strongly quasiconvex in $X$ via the map $f$. Then for each Morse gauge $N$ there is another Morse gauge $N'$ such that for every $N$--Morse geodesic ray $\gamma\!:[0,\infty)\to A$ there is an $N'$--Morse geodesic ray with basepoint $f\bigl(\gamma(0)\bigr)$ in $X$ bounded Hausdorff distance from $f(\gamma)$ (.i.e. $f$ induces a Morse preserving map). 
\end{prop}

\begin{proof}
Since $f$ is a quasi-isometry embedding, $f(\gamma)$ is a $(K,L)$--quasi-geodesic, where $K$, $L$ only depends on the quasi-isometry embedding constants of $f$. Also $f(\gamma)$ is Morse with the Morse gauge $N_1$ depending only on the Morse gauge $N$ and the map $f$ (see the proof of Statement (2) of Proposition \ref{pp1}). Let $(y_n)$ be a sequence of points in $f(\gamma)$ such that $d(f\bigl(\gamma(0),y_n\bigr)$ approaches infinity as $n$ approaches infinity. For each $n$ let $\alpha_n$ be a geodesic segment connecting the basepoint $f\bigl(\gamma(0)\bigr)$ and $y_n$. Since $f(\gamma)$ is a $(K,L)$--Morse quasi-geodesic with Morse gauge $N_1$, each geodesic segment $\alpha_n$ lies in some $D_1$--neighborhood of $f(\gamma)$, where $D_1$ only depends on the Morse gauge $N_1$. Since $X$ is a proper geodesic space, some subsequence of $(\alpha_n)$ converges to a ray $\alpha$ based at $f\bigl(\gamma(0)\bigr)$. Moreover, the Hausdorff distance between $\alpha$ and $f(\gamma)$ is bounded above by some constant $D$ depending only on $K$, $L$, and the Morse gauge $N_1$. This implies that $D$ only depends on the Morse gauge $N$ and the map $f$. Since the Hausdorff distance between $f(\gamma)$ and $\alpha$ is $D$ and $f(\gamma)$ is $N_1$--Morse, the ray $\alpha$ is $N'$--Morse where $N'$ only depends on $D$ and the Morse gauge $N_1$. Therefore, $N'$ only depends on the Morse gauge $N$ and the map $f$.
\end{proof}

\begin{thm}
\label{thth1}
Let $X$, $A$ be two proper geodesic spaces and $A$ is strongly quasiconvex in $X$ via the map $f$. Then $f$ induces a topological embedding $\partial_M f\!:\partial_M A\to \partial_M X$ such that $\partial_M f(\partial_M A)=\Lambda f(A)$. 
\end{thm}

\begin{proof}
The fact that $f$ induces a topological embedding $\partial_M f\!:\partial_M A\to \partial_M X$ is a direct result of Proposition \ref{ppp1} via Proposition 4.2 in \cite{MC1}. However, we need to check that $\partial_M f(\partial_M A)=\Lambda f(A)$.

We first recall the construction of the map $\partial_M f$ from Proposition 4.2 in \cite{MC1}. Fix a point $p$ in $A$ and let $q=f(p)$. By Proposition \ref{ppp1}, for each Morse gauge $N$ there is another Morse gauge $N'$ such that for every $N$--Morse geodesic ray $\alpha$ based at $p$ there is an $N'$--Morse geodesic ray $\beta$ with basepoint $q$ in $X$ bounded Hausdorff distance from $f(\alpha)$. We define $\partial_M f\bigl([\alpha]\bigr)=[\beta]$, then $\partial_M f\!:\partial_M A\to \partial_M X$ is a topological embedding (see the proof of Proposition 4.2 in \cite{MC1}). Obviously, $\partial_M f(\partial_M A)\subset\Lambda f(A)$ by the construction. We now prove the opposite inclusion.

Let $c$ be an arbitrary element in $\Lambda f(A)$. There is a Morse geodesic ray $\beta_1$ based at $q$ in $X$ such that $c=[\beta_1]$ and $\beta_1$ lies in some finite neighborhood of $f(A)$. Since $f$ is a quasi-isometric embedding, there is a quasi-geodesic $\alpha_1$ based at $p$ in $A$ such that the Hausdorff distance between $f(\alpha_1)$ and $\beta_1$ is finite. Therefore, $f(\alpha_1)$ is also a Morse quasi-geodesic. This implies that $\alpha_1$ is also a Morse quasi-geodesic in $A$ by Statement (1) of Proposition \ref{pp1}. Also, $A$ is a proper geodesic space. Then by a similar argument as in the proof of Proposition \ref{ppp1} there is a Morse geodesic ray $\alpha$ based at $p$ in $X$ such that the Hausdorff distance between $\alpha$ and $\alpha_1$ is finite. Thus, the Hausdorff distance between $f(\alpha)$ and $f(\alpha_1)$ is also finite. This implies that the Hausdorff distance between $f(\alpha)$ and $\beta_1$ is finite. By the construction of $\partial_M f$ we have $\partial_M f\bigl([\alpha]\bigr)=[\beta]$ for some geodesic ray $\beta$ based at $q$ such that the Hausdorff distance between $f(\alpha)$ and $\beta$ is finite. Therefore, Hausdorff distance between $\beta$ and $\beta_1$ is finite. This implies that $c=[\beta_1]=[\beta]= \partial_M f\bigl([\alpha]\bigr)\in \partial_M f(\partial_M A)$. Therefore, $\partial_M f(\partial_M A)=\Lambda f(A)$.

\end{proof}

We now state the main theorem of this section.

\begin{thm}
Let $G$ be a finitely generated group. Then
\begin{enumerate}
\item If $H$ is a finitely generated strongly quasiconvex subgroup of $G$, then the inclusion $i:H\inclusion G$ induces a topological embedding $\hat{i}\!:\partial_M H\to \partial_M G$ such that $\hat{i}(\partial_M H)=\Lambda H$. 
\item If $H_1$ and $H_2$ are finitely generated strongly quasiconvex subgroups of $G$, then $H_1\cap H_2$ is strongly quasiconvex in $G$ and $\Lambda (H_1\cap H_2)=\Lambda H_1\cap \Lambda H_2$.
\end{enumerate}
\end{thm}

\begin{proof}
Statement (1) is a direct result of Theorem \ref{thth1}. Therefore, we only need to prove Statement (2). The fact that $H_1\cap H_2$ is strongly quasiconvex in $G$ is a result of Proposition \ref{prp1}. Since $H_1\cap H_2\subset H_1$, $\Lambda (H_1\cap H_2)\subset \Lambda H_1$. Similarly, $\Lambda (H_1\cap H_2)\subset \Lambda H_2$. Therefore, $\Lambda (H_1\cap H_2)\subset\Lambda H_1\cap \Lambda H_2$. We now let $c$ be an arbitrary element in $\Lambda H_1\cap \Lambda H_2$. Then there is a Morse geodesic ray $\gamma$ based at $e$ in some Cayley graph $\Gamma(G,S)$ such that $[\gamma]=c$ and $\gamma$ lies in some finite neighborhoods of $H_1$ and $H_2$. Therefore, $\gamma$ also lies in a finite neighborhood of $H_1\cap H_2$ by Lemma \ref{ll1}. This implies that $c$ is also an element in $\Lambda (H_1\cap H_2)$.
\end{proof}

\section{Strong quasiconvexity and stability in relatively hyperbolic groups}
\label{strhg}

In this section, we investigate strongly quasiconvex subgroups and stable subgroups in relatively hyperbolic groups. We first recall the concepts of coned off Cayley graphs and relative hyperbolic groups.

\begin{defn}
Given a finitely generated group $G$ with Cayley graph $\Gamma(G,S)$ equipped with the path metric and a finite collection $\PP$ of subgroups of G, one can construct the \emph{coned off Cayley graph} $\hat{\Gamma}(G,S,\PP)$ as follows: For each left coset $gP$ where $P\in \PP$, add a vertex $v_{gP}$, called a \emph{peripheral vertex}, to the Cayley graph $\Gamma(G,S)$ and for each element $x$ of $gP$, add an edge $e(x,gP)$ of length 1/2 from $x$ to the vertex $v_{gP}$. This results in a metric space that may not be proper (i.e. closed balls need not be compact).
\end{defn}

\begin{defn} [Relatively hyperbolic group]
\label{rel}
A finitely generated group $G$ is \emph{hyperbolic relative to a finite collection $\PP$ of subgroups of $G$} if the coned off Cayley graph is $\delta$--hyperbolic and \emph{fine} (i.e. for each positive number $n$, each edge of the coned off Cayley graph is contained in only finitely many circuits of length $n$).

Each group $P\in \PP$ is a \emph{peripheral subgroup} and its left cosets are \emph{peripheral left cosets} and we denote the collection of all peripheral left cosets by $\Pi$.

\end{defn}

We now review some known results on relatively hyperbolic subgroups including undistorted subgroups and peripheral subgroups in relatively hyperbolic groups, relations between quasigeodesics in Cayley graphs and coned off Cayley graphs.

\begin{thm}\cite[Theorem 1.5 and Theorem 9.1]{Hruska10}
\label{thhk}
Let $(G,\PP)$ be a finitely generated relatively hyperbolic group and let $H$ be a finitely generated undistorted subgroup of $G$. Let $S$ be some (any) finite generating set for $G$. Then
\begin{enumerate}
 \item There is a constant $A=A(S)$ such that for each geodesic $c$ in the coned off Cayley graph $\hat{\Gamma}(G,S,\PP)$ connecting two points of $H$, every $G$--vertex of $c$ lies within the $A$--neighborhood of $H$ with respect to the metric $d_S$.
\item All subgroups of $H$ of the form $H\cap P^{g}$ where $g\in G$, $P\in\PP$ and $\abs{H\cap P^{g}}=\infty$ lie in only finitely many conjugacy classes in $H$. Furthermore, if $\OO$ is a set of representatives of these conjugacy classes then $(H,\OO)$ is relatively
hyperbolic.
\end{enumerate}
\end{thm}

\begin{thm}\cite[Corollary 1.14]{MR2153979}
\label{tmkr}
A relatively hyperbolic group $(G,\PP)$ is hyperbolic if each peripheral subgroup in $\PP$ is hyperbolic.
\end{thm}

\begin{lem}\cite[Lemma 8.8]{Hruska10}
\label{lh}
Let $(G,\PP)$ be a finitely generated relatively hyperbolic group with a finite generating set $S$. For each $K\geq 1$ and $L\geq 0$ there is $A = A(K,L)>0$ such that the following holds. Let $\alpha$ be a $(K,L)$--quasigeodesic in $\Gamma(G,S)$ and $c$ a geodesic in $\hat{\Gamma}(G,S,\PP)$ with the same endpoints in $G$. Then each $G$--vertex of $c$ lies in the $A$--neighborhood of some vertex of $\alpha$ with respect to the metric $d_S$.
\end{lem}

\begin{lem} \cite[Lemma 4.15]{MR2153979}
\label{lemma2}
Let $(G,\PP)$ be a finitely generated relatively hyperbolic group. Then each conjugate $P^g$ of peripheral subgroup in $\PP$ is strongly quasiconvex in $G$. Moreover, each $D$--neighborhood of peripheral left coset is $M$--Morse in Cayley graph of $G$ where the Morse gauge $M$ only depends on $D$ and the choice of finite generating set of $G$.
\end{lem}

The following theorem provides characterizations of strongly quasiconvex subgroups in relatively hyperbolic groups. This is the main theorem of this section.

\begin{thm}
\label{thfstqg}
Let $(G,\PP)$ be a finitely generated relatively hyperbolic group and $H$ finitely generated undistorted subgroup of $G$. Then the following are equivalent:
\begin{enumerate}
\item The subgroup $H$ is strongly quasiconvex in $G$.
\item The subgroup $H\cap P^g$ is strongly quasiconvex in $P^g$ for each conjugate $P^g$ of peripheral subgroup in $\PP$.
\item The subgroup $H\cap P^g$ is strongly quasiconvex in $G$ for each conjugate $P^g$ of peripheral subgroup in $\PP$. 
\end{enumerate}
\end{thm}

\begin{proof}
The implications ``$(1)\implies(2)\implies (3)$'' are direct results from Propositions \ref{pp2}, \ref{prp1}, and Lemma \ref{lemma2}. We now prove the implication ``$(3) \implies (1)$''. Let $S$ be a finite generating set of $G$ and let $A_0=A_0(S)$ be the constant in Statement (1) of Theorem \ref{thhk}.

Let $K\geq 1$ and $L\geq 0$ arbitrary. Let $A=A(K,L)$ be constant in Lemma \ref{lh}. By Lemma \ref{ll1} there is a constant $A_1>0$ such that for each $t\in B(e,A_0)\cap G$ and $P\in\PP$ \[N_A(tP)\cap N_{A+A_0} (H)\subset N_{A_1}(tPt^{-1}\cap H).\] Let $\mu$ be the maximum of all Morse gauge functions of all set of the form $N_{A_1}(tPt^{-1}\cap H)$ where $\abs{t}_S<A_0$ and $P\in\PP$.

Let \[D=K(2A+1)+L+\mu(K,L)+A_0+A+A_1+1.\]

Let $\alpha\!:[a,b]\to\Gamma(G,S)$ be an arbitrary $(K,L)$--quasigeodesic in $\Gamma(G,S)$ that connects two points $h_1$ and $h_2$ in $H$. By Lemma 1.11 of \cite{MR1744486} III.H, we can assume that $\alpha$ is continuous and \[\ell(\alpha_{|[t,t']})\leq Kd_S\bigl(\alpha(t),\alpha(t')\bigr)+L.\] 
We will show that $\alpha$ lies in the $D$--neighborhood of $H$. Let $c$ be a geodesic in $\hat{\Gamma}(G,S,\PP)$ that connects two points $h_1$ and $h_2$ and let $h_1=s_0,s_1,\cdots,s_n=h_2$ be all $G$--vertices of $c$. By the choice of $A_0$, all $G$--vertices $s_i$ lie in the $A_0$--neighborhood of $H$. Also by the choice of $A$, there is $t_i\in[a,b]$ such that $d_S\bigl(s_i,\alpha(t_i)\bigr)<A$. We consider $t_0=s_0$ and $t_n=s_n$. Therefore, each point $\alpha(t_i)$ lies in the $(A_0+A)$--neighborhood of $H$. It is sufficient to show that each $\alpha(I_i)$ lies in the $D$--neighborhood of $H$ where $I_i$ is a closed subinterval of $[a,b]$ with endpoints $t_{i-1}$ and $t_i$. We consider two cases.

\textbf{Case 1:} Two $G$-vertices $s_{i-1}$ and $s_i$ are adjacent in $\Gamma(G,S)$. Then the distance between two points $\alpha(t_{i-1})$ and $\alpha(t_i)$ with respect to the metric $d_S$ is bounded above by $2A+1$. Therefore, \[\ell(\alpha_{|I_i})\leq Kd_S\bigl(\alpha(t_{i-1}),\alpha(t_i)\bigr)+L\leq K(2A+1)+L.\]
Also, $\alpha(t_i)$ lies in the $(A+A_0)$--neighborhood of $H$. Therefore, $\alpha(I_i)$ lies in the $D$--neighborhood of $H$ by the choice of $D$.

\textbf{Case 2:} Two $G$-vertices $s_{i-1}$ and $s_i$ are not adjacent in $\Gamma(G,S)$. Therefore, $s_{i-1}$ and $s_i$ both lie in the same peripheral left coset $gP$. Since $d_S(gP,H)\leq d_S(s_i,H)<A_0$, $gP=htP$ for some $h\in H$ and $\abs{t}_S<A_0$. Observe that the endpoints of $\alpha(I_i)$ both lies in $N_A(htP)\cap N_{A+A_0}(H)$. Therefore, the endpoints of $h^{-1}\alpha(I_i)$ both lies in $N_A(tP)\cap N_{A+A_0}(H)\subset N_{A_1}(tPt^{-1}\cap H)$. Since $h^{-1}\alpha$ is also a $(K,L)$--quasigeodesic, $h^{-1}\alpha(I_i)$ lies entirely in the $\mu(K,L)$--neighborhood of the set $N_{A_1}(tPt^{-1}\cap H)$. Therefore, $h^{-1}\alpha(I_i)$ lies in $(\mu(K,L)+A_1)$--neighborhood of the set $H$. Translating by $h$, we see that $\alpha(I_i)$ also lies in the $\bigl(\mu(K,L)+A_1\bigr)$--neighborhood of the set $H$. By the choice of $D$, $\alpha(I_i)$ lies in $D$--neighborhood of the set $H$. 

Therefore, subgroup $H$ is strongly quasiconvex in $G$. 
\end{proof}

Two characterizations of stable subgroups in relatively hyperbolic groups follow immediately.

\begin{cor}
Let $(G,\PP)$ be a finitely generated relatively hyperbolic group and $H$ finitely generated undistorted subgroup of $G$. Then the following are equivalent:
\begin{enumerate}
\item The subgroup $H$ is stable in $G$.
\item The subgroup $H\cap P^g$ is stable in $P^g$ for each conjugate $P^g$ of peripheral subgroup in $\PP$.
\item The subgroup $H\cap P^g$ is stable in $G$ for each conjugate $P^g$ of peripheral subgroup in $\PP$. 
\end{enumerate}
\end{cor}

\begin{proof}
We can see that (1) implies (2) and (1) implies (3) easily. In fact, if the subgroup $H$ is stable in $G$, then $H$ is strongly quasiconvex in $G$ and $H$ is hyperbolic (see Theorem \ref{th3}). Therefore, $H\cap P^g$ is strongly quasiconvex in $H$, $P^g$, and $G$ by Proposition \ref{prp1} and Lemma \ref{lemma2}. In particular, $H\cap P^g$ is hyperbolic because $H$ is a hyperbolic group. Therefore, $H\cap P^g$ is stable in $P^g$ and $G$ by Theorem \ref{th3}. The equivalence between (2) and (3) can also be seen easily from Theorems \ref{th3} and \ref{thfstqg}. 

We now prove that (3) implies (1). By the hypothesis, $H\cap P^g$ is strongly quasiconvex in $G$ for each conjugate $P^g$ of peripheral subgroup in $\PP$. Therefore, $H$ is strongly quasiconvex in $G$ by Theorem \ref{thfstqg}. Also $H\cap P^g$ is hyperbolic for each conjugate $P^g$ of peripheral subgroup in $\PP$. Therefore, the subgroup $H$ is also hyperbolic by Statement (2) of Theorem \ref{thhk} and Theorem \ref{tmkr}. Therefore, the subgroup $H$ is stable in $G$ by Theorem \ref{th3}. 

\end{proof}

\section{Strong quasiconvexity, stability in right-angled Coxeter groups and related lower relative divergence}
\label{storacgs}

In this section, we investigate strongly quasiconvex subgroups and stable subgroups in right-angled Coxeter groups. More precisely, we characterize strongly quasiconvex (stable) special subgroups in two dimensional right-angled Coxeter groups. Also for each integer $d\geq 2$ we construct a right-angled Coxeter group together with a non-stable strongly quasiconvex subgroup whose lower relative divergence is exactly polynomial of degree $d$. We remark that these are also the first examples of groups together with non hyperbolic subgroups whose lower relative divergence are polynomial of degree $d$ ($d\geq 2$).

\subsection{Some backgrounds in right-angled Coxeter groups}

This subsection provides background on right-angled Coxeter groups.

\begin{defn}
Given a finite, simplicial graph $\Gamma$, the associated \emph{right-angled Coxeter group} $G_\Gamma$ has generating set $S$ the vertices of $\Gamma$, and relations $s^2 = 1$ for all $s$ in $S$ and $st = ts$ whenever $s$ and $t$ are adjacent vertices.

Let $S_1$ be a subset of $S$. The subgroup of $G_\Gamma$ generated by $S_1$ is a right-angled Coxeter group $G_{\Gamma_1}$, where $\Gamma_1$ is the induced subgraph of $\Gamma$ with vertex set $S_1$ (i.e. $\Gamma_1$ is the union of all edges of $\Gamma$ with both endpoints in $S_1$). The subgroup $G_{\Gamma_1}$ is called a \emph{special subgroup} of $G_\Gamma$. Any of its conjugates is called a \emph{parabolic subgroup} of $G_\Gamma$.

\end{defn}

\begin{defn}
Given a finite, simplicial graph $\Gamma$, the associated \emph{Davis complex} $\Sigma_\Gamma$ is a cube complex constructed as follows. For every $k$--clique, $T \subset \Gamma$, the special subgroup $G_T$ is isomorphic to the direct product of $k$ copies of $Z_2$. Hence, the Cayley graph of $G_T$ is isomorphic to the 1--skeleton of a $k$--cube. The Davis complex $\Sigma_\Gamma$ has 1--skeleton the Cayley graph of $G_\Gamma$, where edges are given unit length. Additionally, for each $k$--clique, $T \subset \Gamma$, and coset $gG_T$, we glue a unit $k$--cube to $gG_T \subset\Sigma_\Gamma$. The Davis complex $\Sigma_\Gamma$ is a $\CAT(0)$ cube complex and the group $G_\Gamma$ acts properly and cocompactly on the Davis complex $\Sigma_\Gamma$ (see \cite{MR2360474}).
\end{defn}

\begin{rem}[see \cite{MR2360474}]
\label{rr}
Let $\Gamma$ be a finite simplicial graph with vertex set $S$ and $\Gamma_1$ the induced subgraph of $\Gamma$ generated by some subset $T$ of $S$. Then the Davis complex $\Sigma_{\Gamma_1}$ embeds isometrically in the Davis complex $\Sigma_{\Gamma}$. Moreover, the Cayley graph $\Sigma^{(1)}_{\Gamma_1}$ of the group $G_{\Gamma_1}$ embeds isometrically in the Cayley graph $\Sigma^{(1)}_{\Gamma}$ of the group $G_{\Gamma}$. 

We also remark that a Davis complex $\Sigma_{\Gamma}$ is a two dimensional cube complex if and only if the defining graph $\Gamma$ is triangle free and contains at least one edge.
\end{rem}


\begin{defn}
Let $\Sigma$ be a $\CAT(0)$ cube complex. In $\Sigma$, we consider the equivalence relation on the set of midcubes of cubes generated by the rule that two midcubes are related if they share a face. A \emph{hyperplane}, $H$, is the union of the midcubes in a single equivalence class. We define the \emph{support} of a hyperplane $H$, denoted $N(H)$, to be the union of cubes which contain midcubes of $H$.
\end{defn}

\begin{rem}
\label{rm11}
Let $\Gamma$ be a finite simplicial graph with vertex set $S$ and $\Sigma_\Gamma$ the associated Davis complex. Each hyperplane in $\Sigma_\Gamma$ separates $\Sigma_\Gamma$ into two convex sets. It follows that the distance between a pair of vertices with respect to the metric $d_S$ equals the number of hyperplanes in $\Sigma_\Gamma$ separating those vertices.

For a generator $v$, let $e_v$ denote the edge from the basepoint 1 to the vertex $v$. Any edge in $\Sigma_\Gamma$ determines a unique hyperplane, namely the hyperplane containing the midpoint of that edge. Denote by $H_v$ the hyperplane containing the midpoint of $e_v$.

For a cube in $\Sigma_\Gamma$, all of the parallel edges are labeled by the same generator $v$. It follows that all of the edges crossing a hyperplane $H$ have the same label $v$, and we call this a hyperplane of \emph{type $v$}. Obviously, if two hyperplanes with the types $v_1$ and $v_2$ cross, then $v_1$ and $v_2$ commute. Since $G_\Gamma$ acts transitively on edges labeled $v$, a hyperplane is of type $v$ if and only if it is a translate of the standard hyperplane $H_v$. Obviously, the star subgroup $G_{St(v)}$ is the stabilizer of the hyperplane $H_v$ and $G_{St(v)}$ can also be considered as vertices of the support $N(H_v)$ of $H_v$. Therefore, the subgroup $gG_{St(v)}g^{-1}$ is the stabilizer of the hyperplane $gH_v$ and $gG_{St(v)}$ can also be considered as vertices of the support $N(gH_v)$ of $gH_v$.
\end{rem}

\subsection{Strongly quasiconvex special subgroup and stable special subgroup characterizations in right-angled Coxeter groups}

In this subsection, we establish characterizations of strongly quasiconvex special subgroups and stable special subgroups in two dimensional right-angled Coxeter groups in terms of defining graphs. As a consequence, we give a criterion on simplical, triangle free graphs that guarantees that the associated right-angled Coxeter groups possess Morse boundaries which are not totally disconnected. 

We first develop some results in right-angled Coxeter groups that help establish the characterizations of strongly quasiconvex (stable) special subgroups in two dimensional right-angled Coxeter groups.

\begin{lem}
\label{lleemmaa1}
Let $\Gamma$ be a simplicial graph with vertex set $S$. Let $K$ be a special subgroup of $G_\Gamma$ generated by some subset $S_1$ of $S$ and $\Gamma_1$ the subgraph of $\Gamma$ induced by $S_1$. For each $g\in G_\Gamma$ $\gamma$ is a shortest path in $\Sigma^{(1)}_\Gamma$ that connects $g$ to some point in $K$ if and only if $\gamma$ is a geodesic in $\Sigma^{(1)}_\Gamma$ and each hyperplane crossed by some edge of $\gamma$ does not intersect the Cayley graph $\Sigma^{(1)}_{\Gamma_1}$ of $K$.
\end{lem}

\begin{proof}
We first assume that $\gamma$ is a shortest path in $\Sigma^{(1)}_\Gamma$ that connects $g$ to some point $h$ in $K$. Obviously, $\gamma$ must be a geodesic in $\Sigma^{(1)}_\Gamma$ connecting $h$ and $g$. Let $n$ be the length of $\gamma$. Then, the geodesic $\gamma$ along the direction from $h$ to $g$ is the concatenation of $n$ edges $e_1e_2\cdots e_n$ where each edge $e_i$ is labelled by some vertex $s_i$ in $S$. Therefore, $g=h(s_1s_2\cdots s_n)$. 

Assume for the contradiction that there is a hyperplane $H_i$ crossed by the edge $e_i$ intersects $\Sigma^{(1)}_{\Gamma_1}$. We choose the smallest number $i$ with that property. In particular, $s_i\in S_1 \subset K$. We claim that for each $j<i$ each hyperplane $H_j$ crossed by edge $e_j$ of $\gamma$ intersects $H_i$. Indeed, each hyperplane $H_j$ ($j<i$) does not intersect $\Sigma^{(1)}_{\Gamma_1}$ by the choice of $i$. Moreover, $\Sigma^{(1)}_{\Gamma_1}$ is connected then there is a path $\alpha$ in $\Sigma^{(1)}_{\Gamma_1}$ that connects $h$ and a point $u$ in $H_i\cap \Sigma^{(1)}_{\Gamma_1}$. Let $\beta$ is the path in $H_i$ that connects $u$ to the midpoint of $e_i$. Therefore, each hyperplane $H_j$ ($j<i$) must intersect the concatenation $\alpha\beta$. However, each hyperplane $H_j$ ($j<i$) does not intersect $\alpha$ because it does not intersect $\Sigma^{(1)}_{\Gamma_1}$. This implies that each hyperplane $H_j$ ($j<i$) intersects $H_i$. Therefore, $s_j$ commutes to $s_i$ for each $j<i$. Thus, \[g=h(s_1s_2\cdots s_{i-1}s_is_{i+1}\cdots s_n)=(hs_i)(s_1s_2\cdots s_{i-1}s_{i+1}\cdots s_n).\] 
The above equality shows that there is path of length $n-1$ that connects $g$ to the point $hs_i\in K$. This is a contradiction. Therefore, each hyperplane crossed by some edge of $\gamma$ does not intersect the Cayley graph $\Sigma^{(1)}_{\Gamma_1}$.

We now assume that $\gamma$ is a geodesic in $\Sigma^{(1)}_\Gamma$ and each hyperplane crossed by some edge of $\gamma$ does not intersect the Cayley graph $\Sigma^{(1)}_{\Gamma_1}$. Let $\gamma'$ be an arbitrary path in $\Sigma^{(1)}_{\Gamma}$ that connects $g$ to some point $h_1$ in $K$. Since $\Sigma^{(1)}_{\Gamma_1}$ is connected, there is a path $\alpha'$ in $\Sigma^{(1)}_{\Gamma_1}$ that connects $h$ and $h_1$. Therefore, hyperplanes crossed by edges of $\gamma$ intersect the concatenation $\alpha'\gamma'$. Also, these hyperplanes do not intersect $\alpha'$ since they do not intersect $\Sigma^{(1)}_{\Gamma_1}$. Therefore, all these hyperplanes must intersect $\gamma'$. This implies that $\ell(\gamma)\leq \ell(\gamma')$. Therefore, $\gamma$ is a shortest path in $\Sigma^{(1)}_\Gamma$ that connects $g$ to some point in $K$.
\end{proof}

\begin{lem}
\label{ld}
Let $\Gamma$ be a simplicial graph with vertex set $S$. Let $K$ be a special subgroup of $G_\Gamma$ generated by some subset $S_1$ of $S$. Let $\sigma$ be an induced 4--cycle with two pairs of non-adjacent vertices $(a_1,a_2)$ and $(b_1,b_2)$ such that $a_1$, $a_2$ both lie in $S_1$ but $b_1$ does not. Then $d_S\bigl((b_1b_2)^n, K\bigr)=2n$ and $d_S\bigl((b_1b_2)^n b_1, K\bigr)=2n+1$.
\end{lem}

\begin{proof}
We will only prove that $d_S\bigl((b_1b_2)^n, K\bigr)=2n$ and the argument for the remaining equality is identical. Let $\Gamma_1$ the subgraph of $\Gamma$ induced by $S_1$. Let $g=(b_1b_2)^n$ and choose the path $\gamma$ in $\Sigma^{(1)}_\Gamma$ connecting the identity element $e$ and $g$ that reads the word $(b_1b_2)^n$. Since two vertices $b_1$ and $b_2$ are not adjacent, the path $\gamma$ is a geodesic in $\Sigma^{(1)}_\Gamma$ by Remark \ref{rr}. Let $H_i$ be the hyperplane crossed by the $i^{th}$ edge of $\gamma$ in the direction from $e$ to $g$. Since the hyperplane $H_1$ is labelled by $b_1\notin S_1$, $H_1$ does not intersect $\Sigma^{(1)}_{\Gamma_1}$. Also the geodesic $\gamma$ is labelled by non-commuting elements $b_1$ and $b_2$ alternatively. This implies that $H_i$ does not intersect $H_j$ for $i\neq j$. In particular, $H_i$ does not intersect $H_1$ for $i\neq 1$. This implies that $H_i$ does not intersect $\Sigma^{(1)}_{\Gamma_1}$ for $i\neq 1$ either. By Lemma \ref{lleemmaa1}, the geodesic $\gamma$ is a shortest path that connects $(b_1b_2)^n$ and some point in $K$. Therefore, $d_S\bigl((b_1b_2)^n, K\bigr)=\ell(\gamma)=2n$. Using a similar argument we also obtain the equality $d_S\bigl((b_1b_2)^n b_1, K\bigr)=2n+1$.
\end{proof}

\begin{prop}
\label{pip1}
Let $\Gamma$ be a simplicial graph with vertex set $S$. Let $K$ be a special subgroup of $G_\Gamma$ generated by some subset $S_1$ of $S$. Assume that there is an induced 4--cycle $\sigma$ with two pairs of non-adjacent vertices $(a_1,a_2)$ and $(b_1,b_2)$ such that $a_1$, $a_2$ both lie in $S_1$ but $b_1$ does not. Then $K$ is not strongly quasiconvex in $G_\Gamma$.
\end{prop}

\begin{proof}
By Theorem \ref{th2}, this is sufficient to prove that the lower relative divergence of $G_\Gamma$ with respect to $K$ is not completely super linear. Let $\{\sigma^n_{\rho}\}$ be the lower relative divergence of $\Sigma^{(1)}_{\Gamma}$ with respect to $K$. We claim that for each $n\geq 2$ and $\rho\in(0,1]$ \[\sigma^n_{\rho}(r)\leq (4n+2)r \text{ for each $r>1$}.\]

Indeed, for each $r>1$ choose an integer $m$ in $[r,2r]$. If $m=2k$ for some integer $k$, we choose $x=(b_1b_2)^k$. Otherwise, $m=2k+1$ for some integer $k$ and we choose $x=(b_1b_2)^kb_1$. Therefore, $d_S(x,K)=m$ by Lemma \ref{ld}. We choose $y=(a_1a_2)^{mn}x$. Then \[d_S(y,K)=d_S\bigl((a_1a_2)^{mn}x,K\bigr)=d_S\bigl(x,(a_1a_2)^{-mn}K\bigr)=d_S(x,K)=m.\] Since $x$ commutes with both $a_1$ and $a_2$, $y=x(a_1a_2)^{mn}$. This implies that there is a geodesic $\gamma$ connecting $x$, $y$ and $\gamma$ traces the word $(a_1a_2)^{mn}$. Obviously, $d_S(x,y)=\ell(\gamma)=2mn$.

By the construction, each vertex of $\gamma$ is element of the form $x(a_1a_2)^i$ or $x\bigl((a_1a_2)^i a_1\bigr)$ in $G_\Gamma$. Again, $x$ commutes with both $a_1$ and $a_2$. Therefore, \[d_S\bigl(x(a_1a_2)^{i},K\bigr)=d_S\bigl((a_1a_2)^i x,K\bigr)=d_S\bigl(x,(a_1a_2)^{-i}K\bigr)=d_S(x,K)=m>r\] and $d_S\bigl(x\bigl((a_1a_2)^i a_1\bigr),K\bigr)=m>r$ similarly. Therefore, the path $\gamma$ lies outside the $r$--neighborhood of $K$.

Since $r\leq d_S(x,K)=m\leq 2r$, there is a path $\gamma_1$ in $\Sigma^{(1)}_{\Gamma}$ connecting $x$ and some point $u$ in $\partial N_r(K)$ such that $\gamma_1$ lies outside the $r$--neighborhood of $K$ and its length is bounded above by $r$. In particular, $d_S(x,u)\leq r$. Similarly, there is a path $\gamma_2$ in $\Sigma^{(1)}_{\Gamma}$ connecting $y$ and some point $v$ in $\partial N_r(K)$ such that $\gamma_2$ lies outside the $r$--neighborhood of $K$ and its length is bounded above by $r$. In particular, $d_S(x,u)\leq r$. Therefore,
\[d_S(u,v)\geq d_S(x,y)-d_S(x,u)-d_S(y,v)\geq 2mn-r-r\geq (2n-2)r\geq nr.\]

We observe that $\bar{\gamma}=\gamma_1\cup\gamma\cup\gamma_2$ is the path outside the $r$--neighborhood of $K$ that connects two points $u$, $v$ in $\partial N_r(K)$. Also, 
\[\ell(\bar{\gamma})=\ell(\gamma_1)+\ell(\gamma)+\ell(\gamma_2)\leq r+2mn+r\leq (4n+2)r.\]

Therefore, $\sigma^n_{\rho}(r)\leq (4n+2)r$ for each $r>1$. This implies that the lower relative divergence of $G_\Gamma$ with respect to $K$ is not completely super linear. Therefore, $K$ is not strongly quasiconvex in $G_\Gamma$ by Theorem \ref{th2}.

\end{proof}

\begin{prop}
\label{pip2}
Let $\Gamma$ be a simplicial, triangle free graph with vertex set $S$ and $K$ subgroup of $G_\Gamma$ generated by some subset $S_1$ of $S$. We assume that if $S_1$ contains two non-adjacent vertices of an induced 4--cycle $\sigma$, then $S_1$ contains all vertices of $\sigma$. If $K$ is an infinite subgroup of $G_\Gamma$, then the lower relative divergence of $G_\Gamma$ with respect to $K$ is at least quadratic
\end{prop}

\begin{proof}
Let $\{\sigma^n_{\rho}\}$ be the lower relative divergence of $\Sigma^{(1)}_{\Gamma}$ with respect to $K$. We claim that for each $n\geq 3$ and $\rho\in(0,1]$ \[\sigma^n_{\rho}(r)\geq (r-1)(\rho r-1) \text{ for each $r>0$}.\]
In fact, if $\sigma^n_{\rho}(r)=\infty$, then the above inequality is true obviously. Otherwise, let $\gamma$ be an arbitrary path outside $N_{\rho r}(K)$ that connects two points $x$ and $y$ in $\partial N_r(K)$ such that $d_S(x,y)\geq nr$. Let $\gamma'_1$ be a shortest geodesic in $\Sigma^{(1)}_{\Gamma}$ that connects $x$ and some point $k_1$ in $K$. Similarly, let $\gamma'_2$ be a shortest geodesic in $\Sigma^{(1)}_{\Gamma}$ that connects $y$ and some point $k_2$ in $K$. Therefore, $d_S(x,k_1)=d_S(y,k_2)=r$ obviously. This implies that \[d_S(k_1,k_2)\geq d_S(x,y)-d_S(x,k_1)-d_S(y,k_2)\geq (n-2)r\geq r.\] 

Let $\Gamma_1$ the subgraph of $\Gamma$ induced by $S_1$. We connect $k_1$ and $k_2$ by a geodesic $\alpha=e_1e_2\cdots e_m$ of length $m$ in the Cayley graph $\Sigma^{(1)}_{\Gamma_1}$ of $K$. Obviously $m\geq r$ and each edge $e_i$ of $\alpha$ is labelled by some vertex $s_i$ in $S_1$. Let $H_i$ be the hyperplane crossed by the edge $e_i$ of $\alpha$. Then each $H_i$ must intersect $\gamma'_1\cup\gamma\cup\gamma'_2$. However, $H_i$ does not intersect $\gamma'_1\cup\gamma'_2$ by Lemma \ref{lleemmaa1}. Therefore, each hyperplane $H_i$ must intersect some edge $f_i$ of $\gamma$. 

Let $h_0,h_1, h_2,\cdots,h_m$ be all vertices of $\alpha$ such that $h_{i-1}$ and $h_i$ are endpoints of $e_i$. For $1\leq i\leq m$ let $g_i$ be the endpoint of the edge $f_i$ of $\gamma$ that lies in the same component of $\Sigma_{\Gamma}-H_i$ containing $h_i$. Let $\alpha_i$ be a the geodesic in the support of $H_i$ that connects $h_i$ and $g_i$. By the convexity of the graph of $\Sigma^{(1)}_{\Gamma_1}$ in the graph $\Sigma^{(1)}_{\Gamma}$ each $\alpha_i$ is the concatenation of two geodesic $\beta_i$ and $\beta'_i$ such that $\beta_i$ lies entirely in $\Sigma^{(1)}_{\Gamma_1}$ and $\beta'_i$ intersects $\Sigma^{(1)}_{\Gamma_1}$ only at the common endpoint $h'_i\in K$ of $\beta_i$ and $\beta'_i$. 

We set $m_i=\ell(\beta'_i)$ for $1\leq i \leq m$. Then $m_i\geq \rho r$ since $g_i$ lies outside the $\rho r$--neighborhood of $K$. We assume that $\beta'_i=e^i_1e^i_2\cdots e^i_{m_i}$ then each edge $e^i_j$ of $\beta'_i$ is labelled by an element in $Lk(s_i)$ and the first edge $e^i_1$ is labelled by a vertex not in $S_1$. We claim that for each $1\leq i\leq m-1$ there is at least $(m_i-1)$ hyperplanes crossed by $\beta'_i$ intersect the subpath $\gamma_i$ of $\gamma$ connecting $g_i$, $g_{i+1}$. Since $\Gamma$ is triangle free, no pair of different vertices in $Lk(s_i)$ are adjacent. This implies that no pairs of hyperplanes crossed by two different edges of $\beta'_i$ intersects. Since the first edge $e^i_1$ of $\beta'_i$ is labelled by a vertex not in $S_1$, the hyperplane crossed by $e^i_1$ does not in intersect $\Sigma^{(1)}_{\Gamma_1}$. Therefore, no subsequent hyperplane crossed by $\beta'_i$ intersects $\Sigma^{(1)}_{\Gamma_1}$ either. We now consider the possibility a hyperplane crossed by $\beta'_i$ intersects $\beta'_{i+1}$. 

If the hyperplane $H_{e^i_1}$ crossed by the first edge $e^i_1$ of $\beta'_i$ does not in intersect $\beta'_{i+1}$, no subsequent hyperplane crossed by $\beta'_i$ intersects $\beta'_{i+1}$ either (otherwise, some hyperplane crossed by an edge of $\beta'_i$ intersects the hyperplane $H_{e^i_1}$ which is a contradiction). Therefore, all $m_i$ hyperplanes crossed by $\beta'_i$ intersect the subpath $\gamma_i$ of $\gamma$. If the hyperplane crossed by $e^i_1$ intersect $\beta'_{i+1}$, then edge $e^i_1$ is labelled by a vertex $a$ in $Lk(s_i)\cap LK(s_{i+1})$. Since $\Gamma$ is triangle free, two consecutive vertices $s_i$ and $s_{i+1}$ are never adjacent. We claim that the hyperplane crossed by the second edge $e^i_2$ of $\beta'_i$ does not intersect $\beta'_{i+1}$. Otherwise, the second edge $e^i_2$ of $\beta'_i$ is labelled by a vertex $b$ other than $a$ in $Lk(s_i)\cap LK(s_{i+1})$. Therefore, four points $a$, $b$, $s_i$, and $s_{i+1}$ are all vertices of an induced 4--cycle $\sigma$ in $\Gamma$. Moreover, two non-adjacent vertices $s_i$ and $s_{i+1}$ both lie in $S_1$ but $a$ does not. This is a contradiction. Therefore, the hyperplane crossed by the second edge $e^i_2$ of $\beta'_i$ does not intersect $\beta'_{i+1}$. This also implies that hyperplanes crossed edge $e^i_j$ ($2 \leq j \leq m$) of $\beta'_i$ does not intersect $\beta'_{i+1}$. Thus, all these hyperplanes must intersect the subpath $\gamma_i$ of $\gamma$. Therefore, there is exactly $(m_i-1)$ hyperplanes crossed by $\beta'_i$ intersect the subpath $\gamma_i$ of $\gamma$. 

For $1\leq i \leq m-1$ the length of the subpath $\gamma_i$ of $\gamma$ is at least $(m_i-1)$ because there is at least $(m_i-1)$ hyperplanes crossed by $\beta'_i$ intersect the subpath $\gamma_i$. This implies that \[\ell(\gamma)\geq \Sigma_{i=1}^{m-1} \ell(\gamma_i)\geq (m-1)(m_i-1)\geq (r-1)(\rho r-1).\]
Therefore, $\sigma^n_{\rho}(r)\geq (r-1)(\rho r-1)$ for each $r>0$. This implies that the lower relative divergence of $G_\Gamma$ with respect to $K$ is at least quadratic.
 
\end{proof}

The following theorem establishes characterizations of strongly quasiconvex subgroups in two dimensional right-angled Coxeter groups. This theorem is a direct result of Propositions \ref{pip1}, \ref{pip2} and Theorem \ref{th2}. This is also the main theorem of this section.

\begin{thm}
\label{ith}
Let $\Gamma$ be a simplicial, triangle free graph with vertex set $S$ and $K$ subgroup of $G_\Gamma$ generated by some subset $S_1$ of $S$. Then the following conditions are equivalent:
\begin{enumerate}
\item The subgroup $K$ is strongly quasiconvex in $G_\Gamma$.
\item If $S_1$ contains two non-adjacent vertices of an induced 4--cycle $\sigma$, then $S_1$ contains all vertices of $\sigma$.
\item Either $\abs{K}<\infty$ or the lower relative divergence of $G_\Gamma$ with respect to $K$ is at least quadratic. 
\end{enumerate}
\end{thm}

\begin{ques}
Is the lower relative divergence of a $\CAT(0)$ group with respect to an infinite strongly quasiconvex subgroup at least quadratic?
\end{ques}

The following corollary is a direct result of Theorem \ref{ith}. The corollary establishes a characterization of stable subgroups in two dimensional right-angled Coxeter groups.

\begin{cor}
\label{cococo1}
Let $\Gamma$ be a simplicial, triangle free graph with vertex set $S$ and $K$ subgroup of $G_\Gamma$ generated by some subset $S_1$ of $S$. Then the following conditions are equivalent:
\begin{enumerate}
\item The subgroup $K$ is stable in $G_\Gamma$.
\item The set $S_1$ does not contain a pair of non-adjacent vertices of an induced 4--cycle in $\Gamma$.
\end{enumerate}
\end{cor}

\begin{proof}
Assume that $S_1$ does not contain any pair of non-adjacent vertices of an induced 4-cycle of $\Gamma$. Then the subgroup $K$ is hyperbolic (see Corollary 12.6.3 in \cite{MR2360474}). Also the subgroup $K$ is strongly quasiconvex in $G_\Gamma$ by Theorem \ref{ith}. Therefore, $K$ is stable subgroup by Theorem \ref{th3}.

We now assume that $S_1$ contains a pair of non-adjacent vertices of an induced 4-cycle $\sigma$ of $\Gamma$. We will prove that $K$ is not stable in $G_\Gamma$. 
If $K$ is not strongly quasiconvex in $G_\Gamma$, then $K$ is not stable in $G_\Gamma$ by Theorem \ref{th3}. Otherwise, $S_1$ contains all vertices of the induced 4--cycle $\sigma$ by Theorem \ref{ith}. This implies that $K$ is not hyperbolic by Corollary 12.6.3 in \cite{MR2360474}. Thus, $K$ is not stable in $G_\Gamma$.
\end{proof}

\begin{ques}
Can we characterize all strongly quasiconvex (stable) subgroups of right-angled Coxeter groups using defining graphs?
\end{ques}

We guess that the recent work of Abbott-Behrstock-Durham \cite{ABD} can help us characterize all stable subgroups of right-angled Coxeter groups using defining graphs. However, characterizing all strongly quasiconvex subgroups of right-angled Coxeter groups using defining graphs is still a difficult question.

The following corollary give a criterion on two dimensional right-angled Coxeter groups to possess Morse boundaries which are not totally disconnected. We remark that the first example of a right-angled Coxeter group whose Morse boundary is not totally disconnected was constructed by Behrstock in \cite{B}. The following corollary generalizes his example in some sense.

\begin{cor}
If the simplicial, triangle free graph $\Gamma$ contains an induced loop $\sigma$ of length greater than $4$ such that the vertex set of $\sigma$ does not contain a pair of non-adjacent vertices of an induced 4--cycle in $\Gamma$, then the Morse boundary of right-angled Coxeter group $G_\Gamma$ is not totally disconnected.
\end{cor}

\begin{proof}
Since the vertex set $S_1$ of the loop $\sigma$ satisfies hypothesis of Corollary \ref{cococo1}, the subgroups $G_\sigma$ generated by $S_1$ is a stable subgroup in $G_\Gamma$. This implies that the Morse boundary of $G_\sigma$ is topologically embedded in the Morse boundary of $G_\Gamma$ by Theorem \ref{thth1}. Also, $G_\sigma$ is a 2--dimensional hyperbolic orbifold group. Therefore, the Morse boundary of $G_\sigma$ is also the Gromov boundary of $G_\sigma$ which is a topological circle. This implies that the Morse boundary of $G_\Gamma$ also contains a circle and therefore, it is not totally disconnected. 
\end{proof}

\begin{conj}
The Morse boundary of a right-angled Coxeter group $G_\Gamma$ is not totally disconnected if and only if the defining graph $\Gamma$ contains an induced loops $\sigma$ of length greater than $4$ such that the vertex set of $\sigma$ does not contain a pair of non-adjacent vertices of an induced 4--cycle in $\Gamma$.
\end{conj}

\subsection{Higher relative lower divergence in right-angled Coxeter groups} 

In this subsection, we construct right-angled Coxeter groups together with non-stable strongly quasiconvex subgroups whose lower relative divergence are arbitrary polynomials of degrees at least $2$. We remark that the author in \cite{Tran1} also constructed various examples of right-angled Coxeter groups whose lower relative divergence with respect to some subgroups are arbitrary polynomials of degrees at least $2$. However, all subgroups in those examples are stable and in this subsection we work on the cases of non-stable subgroups. 

We first establish a connection between lower relative divergence of certain pairs of right-angled Coxeter groups and divergence of some certain geodesics in their Davis complexes. This connection is a key ingredient for the main examples in this subsection.

\begin{lem}
\label{lelele1}
Let $\Gamma$ be a simplicial graph with the vertex $S$ and let $u$, $v$ be two non-adjacent vertices of $\Gamma$. Let $\Omega$ be a new graph by coning off two points $u$, $v$ of $\Gamma$ with a new vertex $t$ and let $\bar{S}=S\cup \{t\}$. Let $x$, $y$ be endpoints of a path $\gamma$ in $\Sigma^{(1)}_{\Omega}-N_1(G_\Gamma)$. If $g_1$, $g_2$ be two elements in $G_\Gamma$ such that $d_{\bar{S}}(x,G_\Gamma)=d_{\bar{S}}(x,g_1)$ and $d_{\bar{S}}(y,G_\Gamma)=d_{\bar{S}}(y,g_2)$, then $g_1^{-1}g_2$ lies in the subgroup $I$ generated by $u$ and $v$. 
\end{lem}

\begin{proof}
We can assume that $x$ and $y$ both lie in $G_\Omega$. Let $\gamma_1$ be a geodesic in $\Sigma^{(1)}_{\Omega}$ connecting $x$, $g_1$ and $\gamma_2$ be a geodesic in $\Sigma^{(1)}_{\Omega}$ connecting $y$, $g_2$. We can assume that $\gamma_1$ along the direction from $g_1$ to $x$ is the concatenation $\gamma_1=e_1e_2\cdots e_m$ of edges of $\Sigma^{(1)}_{\Omega}$. Similarly, $\gamma_2$ along the direction from $g_2$ to $y$ is the concatenation $\gamma_2=f_1f_2\cdots f_n$ of edges of $\Sigma^{(1)}_{\Omega}$. We claim that two edges $e_1$ and $f_1$ are crossed by the same hyperplane which is labelled by $t$. 

We first recall that $G_\Gamma$ is a subgroup generated by $S$ and $\bar{S}-S=\{t\}$. Also $\gamma_1$ is a shortest geodesic that connects $x$ and $G_\Gamma$. Then, the first edge $e_1$ of $\gamma_1$ is labelled by $t$ and we call $H$ the hyperplane crossed by $e_1$. By Remark \ref{rm11} the vertex set of the support of $H$ is $g_1G_{St(t)}$. Also, each element in $g_1G_{St(t)}$ has distance 0 or 1 from $G_\Gamma$. This implies that $H\cap\Sigma^{(1)}_{\Omega}$ lies in the 1--neighborhood of $G_\Gamma$. In particular, $H$ does not intersect $\gamma$. Also, $H$ does not intersect the Cayley graph $\Sigma^{(1)}_{\Gamma}$ of $G_\Gamma$. Thus, $H$ intersects the path $\gamma_2$. Moreover, $H$ must intersect the first edge $f_1$ of $\gamma_2$ because $H\cap\Sigma^{(1)}_{\Omega}$ lies in the 1--neighborhood of $G_\Gamma$. This fact implies that $g_1$, $g_2$ are both vertices of the support of $H$. In other word, $g_1$ and $g_2$ lie in the same left coset of $G_{St(t)}$ or $g_1^{-1}g_2$ lies in $G_{St(t)}$. Also, $g_1^{-1}g_2$ lies in $G_\Gamma$ and $G_\Gamma\cap G_{St(t)}=I$ which is the subgroup generated by $u$ and $v$. Thus, $g_1^{-1}g_2$ lies in the subgroup $I$.
\end{proof}

\begin{lem}
\label{lelele2}
Let $\Gamma$ be a simplicial graph with the vertex $S$ and let $u$, $v$ be two non-adjacent vertices of $\Gamma$. Let $\Omega$ be a new graph by coning off two points $u$, $v$ of $\Gamma$ with a new vertex $t$ and let $\bar{S}=S\cup \{t\}$. Let $I$ be the subgroup generated by $u$ and $v$. Then for each $x$ in $\Sigma^{(1)}_{\Gamma}$ $d_{\bar{S}}(tx,G_\Gamma)=d_S(x,I)+1$.
\end{lem}

\begin{proof}
We assume that $x$ lies in $G_\Gamma$ and let $n=d_S(x,I)$. Let $\gamma$ be a geodesic of length $n$ in $\Sigma^{(1)}_{\Gamma}$ that connects $x$ to some point $h$ in $I$. We can express $\gamma=e_1e_2\cdots e_n$ as a concatenation of edges in $\Sigma^{(1)}_{\Gamma}$, where each edge $e_i$ is labelled by some element $s_i\in S$. Therefore, $x=h(s_1s_2\cdots s_n)$. Let $\gamma_1=t\gamma$ then $\gamma_1$ is a geodesic that connects $th$ and $tx$. Moreover, we can express $\gamma_1=f_1f_2\cdots f_n$ where each edge $f_i$ is the translation of edge $e_i$ of $\gamma$ by $t$. In particular, each edge $f_i$ is also labelled by $s_i$. Since $th=ht$, there is an edge $e'_1$ labelled by $t$ with endpoints $h$ and $th$. We construct $\bar{\gamma}=e'_1\cup\gamma_1$ and we claim that $\bar{\gamma}$ is a shortest geodesic that connects $tx$ and $G_\Gamma$. 

We first prove that $\bar{\gamma}$ is a geodesic in $\Sigma^{(1)}_{\Omega}$. Since the inclusion $\Sigma^{(1)}_{\Gamma}\inclusion\Sigma^{(1)}_{\Omega}$ is an isometric embedding, $\gamma$ is also a geodesic in $\Sigma^{(1)}_{\Omega}$. Therefore, $\gamma_1=t\gamma$ is a geodesic in $\Sigma^{(1)}_{\Omega}$. Since $\gamma_1$ is labelled by elements in $S$ and the first edge $e'_1$ of $\bar{\gamma}$ is labelled by $t\notin S$, the hyperplane crossed by $e'_1$ does not intersect $\gamma_1$. Therefore, $\bar{\gamma}$ is a goedesic in $\Sigma^{(1)}_{\Omega}$.

We now prove that $\bar{\gamma}$ is a shortest geodesic that connects $tx$ and $G_\Gamma$. By Lemma \ref{lleemmaa1}, it is sufficient to prove that each hyperplane crossed by some edge of $\bar{\gamma}$ does not intersect the Cayley graph $\Sigma^{(1)}_{\Gamma}$ of $G_\Gamma$. Since $G_\Gamma$ is the subgroup generated by $S$ and the hyperplane $H$ crossed by the first edge $e'_1$ of $\bar{\gamma}$ is labelled by $t$, $H$ does not intersect $\Sigma^{(1)}_{\Gamma}$. Assume for the contradiction that hyperplane $H_i$ crossed by some edge $f_i$ of $\bar{\gamma}$ intersects $\Sigma^{(1)}_{\Gamma}$. We choose $i$ is the smallest number with that property. In particular, $H_i$ must intersect $H$ and therefore, $s_i$ commutes to $t$. This implies that $s_i=u$ or $s_i=v$. Therefore, $s_i$ is an element in $I$.

By the choice of $i$, each hyperplane $H_j$ crossed by edge $f_j$ ($j<i$) does not intersect $\Sigma^{(1)}_{\Gamma}$. Therefore, each $H_j$ must intersect $H_i$. This implies that $s_i$ commutes to all $s_j$ for $j<i$. Therefore, \[x=h(s_1s_2\cdots s_{i-1}s_is_{i+1}\cdots s_n)=(hs_i)(s_1s_2\cdots s_{i-1}s_{i+1}\cdots s_n).\]
This implies that there is a path in $\Sigma^{(1)}_{\Gamma}$ with length $n-1$ that connects $x$ to the element $hs_i\in I$ which is a contradiction. Thus, each hyperplane crossed by some edge of $\bar{\gamma}$ does not intersect the Cayley graph $\Sigma^{(1)}_{\Gamma}$ of $G_\Gamma$. Therefore, the path $\bar{\gamma}$ is a shortest geodesic that connects $tx$ and $G_\Gamma$. This implies that \[d_{\bar{S}}(tx,G_\Gamma)=\ell(\bar{\gamma})=\ell(\gamma_1)+1=\ell(\gamma)+1=d_S(x,I)+1.\] 
\end{proof}

\begin{prop}
\label{pipipi1}
Let $\Gamma$ be a simplicial graph with the vertex $S$ and let $u$, $v$ be two non-adjacent vertices of $\Gamma$. Let $\Omega$ be a new graph by coning off two points $u$, $v$ of $\Gamma$ with a new vertex $t$ and let $\bar{S}=S\cup \{t\}$. Let $\alpha$ be the bi-infinite geodesic containing the identity $e$ labelled by $u$ and $v$ alternatively. Then \[Div_{\alpha}^{\Sigma^{(1)}_{\Omega}}\preceq div(G_\Omega,G_\Gamma)\preceq Div_{\alpha}^{\Sigma^{(1)}_{\Gamma}},\] where $Div_{\alpha}^{\Sigma^{(1)}_{\Omega}}$ and $Div_{\alpha}^{\Sigma^{(1)}_{\Gamma}}$ are geodesic divergence of $\alpha$ in $\Sigma^{(1)}_{\Omega}$ and $\Sigma^{(1)}_{\Gamma}$ respectively. 
\end{prop}

\begin{proof}
Let $I$ be a subgroup generated by $u$ and $v$. Then the geodesic $\alpha$ is the Cayley graph of $I$ in ${\Sigma^{(1)}_{\Omega}}$ and ${\Sigma^{(1)}_{\Gamma}}$. Therefore, the lower relative divergence of ${\Sigma^{(1)}_{\Omega}}$ and ${\Sigma^{(1)}_{\Omega}}$ with respect to the path $\alpha$ are equivalent to the lower relative divergence of these spaces with respect to $I$. Also by Remark \ref{gddv}, the lower relative divergence of a geodesic space with respect to a periodic bi-infinite geodesic is equivalent to the divergence of the geodesic in the space. Therefore, this is sufficient to prove \[div(\Sigma^{(1)}_{\Omega},I)\preceq div(\Sigma^{(1)}_{\Omega},G_\Gamma)\preceq div(\Sigma^{(1)}_{\Gamma},I).\]

Let $\{\sigma^n_{\rho}\}$ be the lower relative divergence of $\Sigma^{(1)}_{\Omega}$ with respect to $G_\Gamma$, let $\{\bar{\sigma^n_{\rho}}\}$ be the lower relative divergence of $\Sigma^{(1)}_{\Omega}$ with respect to $I$, and let $\{\bar{\bar{\sigma^n_{\rho}}}\}$ be the lower relative divergence of $\Sigma^{(1)}_{\Gamma}$ with respect to $I$. We will prove that for each $n\geq 2$ and $\rho\in(0,1]$
\[\bar{\sigma^n_{\rho}}(r)\leq \sigma^n_{\rho}(r)\leq \bar{\bar{\sigma^{2n}_{\rho}}}(r)+2r\ \text{ for each $r>1$}.\] 

We first prove the left inequality. If $\sigma^n_{\rho}(r)=\infty$, the left inequality is true obviously. Otherwise, let $x$, $y$ be an arbitrary point in $\partial N_r(G_\Gamma)$ such that there is a path outside $N_r(G_\Gamma)$ in $\Sigma^{(1)}_{\Omega}$ that connects $x$ and $y$ and $d_{\bar{S}}(x,y)\geq nr$. Let $g_1$, $g_2$ be two elements in $G_\Gamma$ such that $d_{\bar{S}}(x,G_\Gamma)=d_{\bar{S}}(x,g_1)$ and $d_{\bar{S}}(y,G_\Gamma)=d_{\bar{S}}(y,g_2)$. Then $g_1, g_2$ both lie in the same left coset $gI$ of the subgroup $I$ by Lemma \ref{lelele1}. 

Since element $g$ in $G_\Omega$ acts isometrically on $\Sigma^{(1)}_{\Omega}$, $div(\Sigma^{(1)}_{\Omega},gI)=div(\Sigma^{(1)}_{\Omega},I)=\{\bar{\sigma^n_{\rho}}\}.$ This is obvious that \[d_{\bar{S}}(x,gI)=d_{\bar{S}}(x,G_\Gamma)=r \text{ and } d_{\bar{S}}(y,gI)=d_{\bar{S}}(y,G_\Gamma)=r.\] Also, any path that lies outside some $s$--neighborhood of $G_\Gamma$ also lies outside the $s$--neighborhood of $gI$ because $gI$ is a subset of $G_\Gamma$. Therefore, $\bar{\sigma^n_{\rho}}(r)\leq \sigma^n_{\rho}(r)$. This implies that $div(\Sigma^{(1)}_{\Omega},I)\preceq div(\Sigma^{(1)}_{\Omega},G_\Gamma)$.

We now prove the second inequality $\sigma^n_{\rho}(r)\leq \bar{\bar{\sigma^{2n}_{\rho}}}(r)+2r$ for each $r>1$. If $\bar{\bar{\sigma^{2n}_{\rho}}}(r)=\infty$, then the inequality is true obviously. Otherwise, let $x_1$, $y_1$ be arbitrary two points in $\partial N_r(I)\subset\Sigma^{(1)}_{\Gamma}$ such that there is a path outside $N_r(I)$ connecting $x_1$, $y_1$ and $d_S(x_1,y_1)\geq (2n)r$. Since the inclusion $\Sigma^{(1)}_{\Gamma}\inclusion\Sigma^{(1)}_{\Omega}$ is an isometric embedding, $d_{\bar{S}}(x_1,y_1)=d_S(x_1,y_1)\geq (2n)r$.

Let $\beta$ be an arbitrary path outside $N_{\rho r}(I)$ in $\Sigma^{(1)}_{\Gamma}$ connecting $x_1$ and $y_1$. Then the path $\beta_1=t\beta$ connects two points $tx_1$ and $ty_1$. By Lemma \ref{lelele2}, $\beta_1$ lies outside the $(\rho r+1)$--neighborhood of $G_\Gamma$ in $\Sigma^{(1)}_{\Omega}$ and \[d_{\bar{S}}(tx_1,G_\Gamma)=d_S(x_1,I)+1=r+1 \text{ and } d_{\bar{S}}(ty_1,G_\Gamma)=d_S(y_1,I)+1=r+1.\]
This implies that there is a geodesic $\beta_2$ with length 1 that lies outside the $r$--neighborhood of $G_\Gamma$ in $\Sigma^{(1)}_{\Omega}$ and $\beta_2$ connects $tx_1$ with some point $x_2 \in N_r(G_\Gamma)$. Similarly, there is a geodesic $\beta_3$ with length 1 that lies outside the $r$--neighborhood of $G_\Gamma$ in $\Sigma^{(1)}_{\Omega}$ and $\beta_3$ connects $ty_1$ with some point $y_2 \in N_r(G_\Gamma)$. 
Moreover, \begin{align*} d_{\bar{S}}(x_2,y_2)&\geq d_{\bar{S}}(tx_1,ty_1)-d_{\bar{S}}(tx_1,x_2)-d_{\bar{S}}(ty_1,y_2)\\&\geq d_{\bar{S}}(x_1,y_1)-d_{\bar{S}}(tx_1,x_2)-d_{\bar{S}}(ty_1,y_2)\\&\geq 2nr-1-1\geq nr.\end{align*} 
Also $\bar{\beta}=\beta_2\cup\beta_1\cup\beta_3$ is the path outside the $\rho r$--neighborhood of $G_\Gamma$ connecting $x_2$ and $y_2$. Therefore, 
\[\sigma^n_{\rho}(r)\leq \ell(\bar{\beta})\leq \ell(\beta_2)+\ell(\beta_1)+\ell(\beta_3)\leq 1+\ell(\beta)+1\leq \ell(\beta)+2r.\]
This implies that $\sigma^n_{\rho}(r)\leq \bar{\bar{\sigma^{2n}_{\rho}}}(r)+2r$. Therefore, $div(\Sigma^{(1)}_{\Omega},G_\Gamma)\preceq div(\Sigma^{(1)}_{\Gamma},I)$. Thus, \[Div_{\alpha}^{\Sigma^{(1)}_{\Omega}}\preceq div(G_\Omega,G_\Gamma)\preceq Div_{\alpha}^{\Sigma^{(1)}_{\Gamma}}.\]

\end{proof}

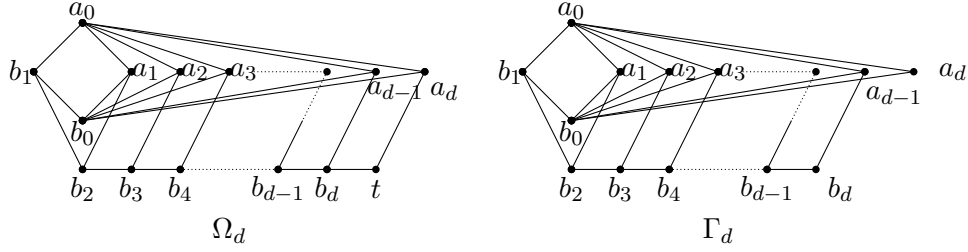
\begin{figure}
\begin{tikzpicture}[scale=0.65]

\draw (0,1) node[circle,fill,inner sep=1pt, color=black](1){} -- (1,2) node[circle,fill,inner sep=1pt, color=black](1){}-- (2,1) node[circle,fill,inner sep=1pt, color=black](1){}-- (1,0) node[circle,fill,inner sep=1pt, color=black](1){} -- (0,1) node[circle,fill,inner sep=1pt, color=black](1){}; 

\draw (1,2) node[circle,fill,inner sep=1pt, color=black](1){} -- (3,1) node[circle,fill,inner sep=1pt, color=black](1){}-- (1,0) node[circle,fill,inner sep=1pt, color=black](1){};

\draw (1,2) node[circle,fill,inner sep=1pt, color=black](1){} -- (4,1) node[circle,fill,inner sep=1pt, color=black](1){}-- (1,0) node[circle,fill,inner sep=1pt, color=black](1){};

\draw (1,2) node[circle,fill,inner sep=1pt, color=black](1){} -- (7,1) node[circle,fill,inner sep=1pt, color=black](1){}-- (1,0) node[circle,fill,inner sep=1pt, color=black](1){};

\draw (1,2) node[circle,fill,inner sep=1pt, color=black](1){} -- (8,1) node[circle,fill,inner sep=1pt, color=black](1){}-- (1,0) node[circle,fill,inner sep=1pt, color=black](1){};

\draw (3,1) node[circle,fill,inner sep=1pt, color=black](1){} -- (2,-1) node[circle,fill,inner sep=1pt, color=black](1){}-- (1,-1) node[circle,fill,inner sep=1pt, color=black](1){};

\draw (4,1) node[circle,fill,inner sep=1pt, color=black](1){} -- (3,-1) node[circle,fill,inner sep=1pt, color=black](1){}-- (2,-1) node[circle,fill,inner sep=1pt, color=black](1){};

\draw (7,1) node[circle,fill,inner sep=1pt, color=black](1){} -- (6,-1) node[circle,fill,inner sep=1pt, color=black](1){}-- (5,-1) node[circle,fill,inner sep=1pt, color=black](1){};

\draw (8,1) node[circle,fill,inner sep=1pt, color=black](1){} -- (7,-1) node[circle,fill,inner sep=1pt, color=black](1){}-- (6,-1) node[circle,fill,inner sep=1pt, color=black](1){};

\draw[densely dotted] (6,1) -- (5.5,0);

\draw (5.5,0) -- (5,-1);

\draw[densely dotted] (3,-1) node[circle,fill,inner sep=1pt, color=black](1){}-- (5,-1) node[circle,fill,inner sep=1pt, color=black](1){};

\draw[densely dotted] (4.5,1) -- (6,1) node[circle,fill,inner sep=1pt, color=black](1){};

\draw (0,1) node[circle,fill,inner sep=1pt, color=black](1){} -- (1,-1) node[circle,fill,inner sep=1pt, color=black](1){}-- (2,1) node[circle,fill,inner sep=1pt, color=black](1){};

\node at (1,-0.25) {$b_0$};

\node at (1,2.25) {$a_0$};

\node at (-0.25,1) {$b_1$};

\node at (2.3,1) {$a_1$};

\node at (3.3,1) {$a_2$};

\node at (4.3,1) {$a_3$};

\node at (7.4,0.6) {$a_{d-1}$};

\node at (8.4,0.6) {$a_d$};

\node at (1,-1.4) {$b_2$};

\node at (2,-1.4) {$b_3$};

\node at (3,-1.4) {$b_4$};

\node at (5,-1.4) {$b_{d-1}$};

\node at (6,-1.4) {$b_d$};

\node at (7,-1.4) {$t$};

\node at (4,-2.25) {$\Omega_d$};


\draw (10,1) node[circle,fill,inner sep=1pt, color=black](1){} -- (11,2) node[circle,fill,inner sep=1pt, color=black](1){}-- (12,1) node[circle,fill,inner sep=1pt, color=black](1){}-- (11,0) node[circle,fill,inner sep=1pt, color=black](1){} -- (10,1) node[circle,fill,inner sep=1pt, color=black](1){}; 

\draw (11,2) node[circle,fill,inner sep=1pt, color=black](1){} -- (13,1) node[circle,fill,inner sep=1pt, color=black](1){}-- (11,0) node[circle,fill,inner sep=1pt, color=black](1){};

\draw (11,2) node[circle,fill,inner sep=1pt, color=black](1){} -- (14,1) node[circle,fill,inner sep=1pt, color=black](1){}-- (11,0) node[circle,fill,inner sep=1pt, color=black](1){};

\draw (11,2) node[circle,fill,inner sep=1pt, color=black](1){} -- (18,1) node[circle,fill,inner sep=1pt, color=black](1){}-- (11,0) node[circle,fill,inner sep=1pt, color=black](1){};

\draw (11,2) node[circle,fill,inner sep=1pt, color=black](1){} -- (17,1) node[circle,fill,inner sep=1pt, color=black](1){}-- (11,0) node[circle,fill,inner sep=1pt, color=black](1){};

\draw (13,1) node[circle,fill,inner sep=1pt, color=black](1){} -- (12,-1) node[circle,fill,inner sep=1pt, color=black](1){}-- (11,-1) node[circle,fill,inner sep=1pt, color=black](1){};

\draw (14,1) node[circle,fill,inner sep=1pt, color=black](1){} -- (13,-1) node[circle,fill,inner sep=1pt, color=black](1){}-- (12,-1) node[circle,fill,inner sep=1pt, color=black](1){};

\draw (17,1) node[circle,fill,inner sep=1pt, color=black](1){} -- (16,-1) node[circle,fill,inner sep=1pt, color=black](1){}-- (15,-1) node[circle,fill,inner sep=1pt, color=black](1){};

\draw[densely dotted] (16,1) -- (15.5,0);

\draw (15.5,0) -- (15,-1);

\draw[densely dotted] (13,-1) node[circle,fill,inner sep=1pt, color=black](1){}-- (15,-1) node[circle,fill,inner sep=1pt, color=black](1){};

\draw[densely dotted] (14.5,1) -- (16,1) node[circle,fill,inner sep=1pt, color=black](1){};

\draw (10,1) node[circle,fill,inner sep=1pt, color=black](1){} -- (11,-1) node[circle,fill,inner sep=1pt, color=black](1){}-- (12,1) node[circle,fill,inner sep=1pt, color=black](1){};

\node at (11,-0.25) {$b_0$};

\node at (11,2.25) {$a_0$};

\node at (9.75,1) {$b_1$};

\node at (12.3,1) {$a_1$};

\node at (13.3,1) {$a_2$};

\node at (14.3,1) {$a_3$};

\node at (17.6,0.5) {$a_{d-1}$};

\node at (18.8,1) {$a_d$};

\node at (11,-1.4) {$b_2$};

\node at (12,-1.4) {$b_3$};

\node at (13,-1.4) {$b_4$};

\node at (15,-1.4) {$b_{d-1}$};

\node at (16.5,-1.4) {$b_{d}$};




\node at (14,-2.25) {$\Gamma_{d}$};

\end{tikzpicture}

\caption{Graph $\Omega_d$ is constructed from graph $\Gamma_d$ by coning off two non-adjacent vertices $a_d$, $b_d$ by a new vertex $t$.}
\label{asecond}
\end{figure}

We are now ready to construct the main examples for this subsection and we start with defining graphs. More precisely, for each integer $d\geq 2$ a graph $\Omega_d$ and its subgraph $\Gamma_d$ are constructed as in Figure \ref{asecond}. We remark that the graph $\Gamma_d$ was originally constructed by Dani-Thomas \cite{MR3314816} to study divergence of right-angled Coxeter groups and the graph $\Omega_d$ is a variation of $\Gamma_d$. However, we are going to use these pairs of graphs to construct corresponding pairs of right-angled Coxeter groups with desired lower relative divergence in this paper.

\begin{prop}[Proposition 3.19 in \cite{MR3451473}] 
\label{propo1}
For each $d\geq 2$, let $\Gamma_d$ be the graph as in Figure \ref{asecond}. Let $\alpha_d$ be a bi-infinite geodesic containing $e$ and labeled by $\cdots a_db_da_db_d\cdots$. Then the divergence of $\alpha_d$ in $\Sigma^{(1)}_{\Gamma_d}$ is equivalent to the polynomial of degree $d$.
\end{prop}

We remark that the above proposition was proved in \cite{MR3451473} but the proof was based mostly on the work of Dani-Thomas in \cite{MR3314816}. 
Similarly, we also obtain an analogous proposition as follows. 

\begin{prop}
\label{propo2}
For each $d\geq 2$, let $\Omega_d$ be the graph as in Figure \ref{asecond}. Let $\alpha_d$ be a bi-infinite geodesic containing $e$ and labeled by $\cdots a_db_da_db_d\cdots$. Then the divergence of $\alpha_d$ in $\Sigma^{(1)}_{\Omega_d}$ is equivalent to the polynomial of degree $d$.
\end{prop}

\begin{proof}
We observe that $\Gamma_d$ is an induced subgraph of $\Omega_d$. Therefore, the graph $\Sigma^{(1)}_{\Gamma_d}$ embeds isometrically in the graph $\Sigma^{(1)}_{\Omega_d}$. This implies that for each $r>0$ any path in $\Sigma^{(1)}_{\Gamma_d}$ that stays outside the ball centered at the identity element $e$ with radius $r$ in $\Sigma^{(1)}_{\Gamma_d}$ also stays outside the ball with the same center and the same radius in $\Sigma^{(1)}_{\Omega_d}$. Therefore, the divergence of $\alpha_d$ in $\Sigma^{(1)}_{\Omega_d}$ is dominated by the divergence of $\alpha_d$ in $\Sigma^{(1)}_{\Gamma_d}$. This implies that the divergence of $\alpha_d$ in $\Sigma^{(1)}_{\Omega_d}$ is dominated by $r^d$ by Proposition~\ref{propo1}.

We now prove that the divergence of $\alpha_d$ in $\Sigma^{(1)}_{\Omega_d}$ dominates $r^d$. We observe that $\Omega_d$ is an induced subgraph of $\Gamma_{d+2}$. Using the identical proof as in the first paragraph the divergence of $\alpha_d$ in $\Sigma^{(1)}_{\Omega_d}$ dominates the divergence of $\alpha_d$ in $\Sigma^{(1)}_{\Gamma_{d+2}}$. Therefore, it is sufficient to prove that the divergence of $\alpha_d$ in $\Sigma^{(1)}_{\Gamma_{d+2}}$ dominates $r^d$. In fact, for each $r>0$ let $\gamma$ be an arbitrary path connecting $\alpha_d(-r)$ and $\alpha_d(r)$ that lies outside the ball centered at $e$ with radius $r$. Let $\alpha$ be the ray obtained from $\alpha_d$ containing $\alpha_d(r)$ with the initial point $e$. Then $\alpha$ lies in the support of a hyperplane crossed by an edge labelled by $b_{d+1}$. Since the hyperplane $H$ crossed by the first edge of $\alpha$ intersects $\gamma$, there is a geodesic $\beta$ in the support of $H$ connecting $e$ and some point $x$ in $\gamma$. Let $\gamma_1$ be the subpath of $\gamma$ connecting $x$ and $\alpha_d(r)$. Then the length of $\gamma_1$ is at least $\frac{1}{2^{d(d+1)}}r^d$ by the proof of Proposition 5.3 in \cite{MR3314816}. Therefore, the length of $\gamma$ is also at least $\frac{1}{2^{d(d+1)}}r^d$. This implies that the divergence of $\alpha_d$ in $\Sigma^{(1)}_{\Gamma_{d+2}}$ dominates $r^d$. Therefore, the divergence of $\alpha_d$ in $\Sigma^{(1)}_{\Omega_d}$ dominates $r^d$.  
\end{proof}

We now state the main theorem of this subsection.

\begin{thm}
For each $d\geq 2$ we construct the graph $\Omega_d$ and the subgraph $\Gamma_d$ as in Figure \ref{asecond}. Then $G_{\Gamma_d}$ is a non-stable strongly quasiconvex subgroup of $G_{\Omega_d}$ and the lower relative divergence of $G_{\Omega_d}$ with respect to $G_{\Gamma_d}$ is equivalent to the polynomial of degree $d$. 
\end{thm}

\begin{proof}
By Propositions \ref{pipipi1}, \ref{propo1}, and \ref{propo2}, the lower relative divergence of $G_{\Omega_d}$ with respect to $G_{\Gamma_d}$ is equivalent to the polynomial of degree $d$. In particular, $G_{\Gamma_d}$ is a strongly quasiconvex subgroup of $G_{\Omega_d}$ by Theorem \ref{th2}. However, the subgroup $G_{\Gamma_d}$ is not hyperbolic because the graph $\Gamma_d$ contains an induced 4--cycle. This implies that $G_{\Gamma_d}$ is not a stable subgroup of $G_{\Omega_d}$. 
\end{proof}

\section{Strong quasiconvexity, stability, and lower relative divergence in right-angled Artin groups}
\label{sec8}

In this section, we prove that two notions of strong quasiconvexity and stability are equivalent in the right-angled Artin group $A_\Gamma$ (except for the case of finite index subgroups). We also characterize non-trivial strongly quasiconvex subgroups of infinite index in $A_\Gamma$ (i.e. non-trivial stable subgroups in $A_\Gamma$) by quadratic lower relative divergence. These results strengthen the work of Koberda-Mangahas-Taylor \cite{KMT} on characterizing purely loxodromic subgroups of right-angled Artin groups.

\subsection{Some backgrounds in right-angled Artin groups}

\begin{defn}
Given a finite simplicial graph $\Gamma$, the associated \emph{right-angled Artin group} $A_{\Gamma}$ has generating set $S$ the vertices of $\Gamma$, and relations $st = ts$ whenever $s$ and $t$ are adjacent vertices.

Let $S_1$ be a subset of $S$. The subgroup of $A_{\Gamma}$ generated by $S_1$ is a right-angled Artin group $A_{\Gamma_1}$, where $\Gamma_1$ is the induced subgraph of $\Gamma$ with vertex set $S_1$ (i.e.~$\Gamma_1$ is the union of all edges of $\Gamma$ with both endpoints in $S_1$). The subgroup $A_{\Gamma_1}$ is called a \emph{special subgroup} of $A_{\Gamma}$. 

A \emph{reduced word} for a group element $g$ in $A_{\Gamma}$ is a minimal length word in the free group $F(S)$ representing $g$.
\end{defn}

\begin{defn}
Let $\Gamma_1$ and $\Gamma_2$ be two graphs, the \emph{join} of $\Gamma_1$ and $\Gamma_2$ is a graph obtained by connecting every vertex of $\Gamma_1$ to every vertex of $\Gamma_2$ by an edge.

Let $J$ be an induced subgraph of $\Gamma$ which decomposes as a nontrivial join. We call $A_J$ a \emph{join subgroup} of $A_{\Gamma}$. A reduced word $w$ in $A_\Gamma$ is called a \emph{join word} if $w$ represents element in some join subgroup. If $\beta$ is a subword of $w$, we will say that $\beta$ is a \emph{join subword} of $w$ when $\beta$ is itself a join word.
\end{defn}

\begin{defn}
Let $\Gamma$ be a simplicial, finite, connected graph such that $\Gamma$ does not decompose as a nontrivial join. A group element $g$ in $A_{\Gamma}$ is \emph{loxodromic} if $g$ is not conjugate into a join subgroup. If every nontrivial group element in a subgroup $H$ of $A_{\Gamma}$ is loxodromic, then $H$ is \emph{purely loxodromic}.
\end{defn}

\begin{defn}
Let $\Gamma$ be a finite simplicial graph with the set $S$ of vertices. Let $T$ be a torus of dimension $\abs{S}$ with edges labeled by the elements of $S$. Let $X_{\Gamma}$ denote the subcomplex of $T$ consisting of all faces whose edge labels span a complete subgraph in $\Gamma$ (or equivalently, mutually commute in $A_{\Gamma}$). $X_{\Gamma}$ is called the \emph{Salvetti complex}.
\end{defn}

\begin{rem}
The fundamental group of $X_{\Gamma}$ is $A_{\Gamma}$. The universal cover $\tilde{X}_{\Gamma}$ of $X_{\Gamma}$ is a $\CAT(0)$ cube complex with a free, cocompact action of $A_{\Gamma}$. Obviously, the 1-skeleton $\tilde{X}^{(1)}_{\Gamma}$ of $\tilde{X}_{\Gamma}$ is the Cayley graph of $A_{\Gamma}$ with respect to the generating set $S$. Consequently, reduced words in $A_\Gamma$ correspond to geodesics in $\tilde{X}^{(1)}_{\Gamma}$, which we call \emph{combinatorial geodesics}. We refer to distance in $\tilde{X}^{(1)}_{\Gamma}$ as \emph{combinatorial distance}. 
If $\Gamma_1$ is an induced subgraph of a $\Gamma$, then there is a natural isometric embeddings of the 1-skeleton $\tilde{X}^{(1)}_{\Gamma_1}$ of $\tilde{X}_{\Gamma_1}$ into the 1-skeleton $\tilde{X}^{(1)}_{\Gamma}$ of $\tilde{X}_{\Gamma}$. 
\end{rem}

\begin{thm}[Theorem 1.1, Theorem 5.2, and Corollary 6.2 in \cite{KMT}]
\label{th}
Let $\Gamma$ be a simplicial, finite, connected graph such that $\Gamma$ does not decompose as a nontrivial join. Let $H$ be a finitely generated subgroup of $A_\Gamma$. Then the following are equivalent.
\begin{enumerate}
\item $H$ is a purely loxodromic subgroup.
\item $H$ is stable.
\item There exists a positive number $N = N(H)$ such that for any reduced word $w$ representing $h\in H$, and any join subword $w'$ of $w$, we have $\ell(w')\leq N$. 
\end{enumerate}
\end{thm}

The following corollary is a direct consequence of the above theorem. 

\begin{cor}
\label{c1}
Let $\Gamma$ be a simplicial, finite, connected graph with the vertex set $S$ such that $\Gamma$ does not decompose as a nontrivial join. Let $H$ be a non-trivial finitely generated purely loxodromic subgroup. Then there is a positive constant $M$ such that every geodesic in the Cayley graph $\tilde{X}^{(1)}_{\Gamma}$ connecting points in $H$ lies in the $M$--neighborhood of $H$.
\end{cor}

The properties of hyperplanes of $\tilde{X}_{\Gamma}$ in the following remark were observed by Behrstock-Charney in \cite{MR2874959}.

\begin{rem}
Each hyperplane in $\tilde{X}_{\Gamma}$ separates $\tilde{X}_{\Gamma}$ into two convex sets. It follows that the combinatorial distance between a pair of vertices equals the number of hyperplanes in $\tilde{X}_{\Gamma}$ separating those vertices.

For a generator $v$, let $e_v$ denote the edge from the basepoint 1 to the vertex $v$. Any edge in $\tilde{X}_{\Gamma}$ determines a unique hyperplane, namely the hyperplane containing the midpoint of that edge. Denote by $H_v$ the hyperplane containing the midpoint of $e_v$.

For a cube in $\tilde{X}_{\Gamma}$, all of the parallel edges are labeled by the same generator $v$. It follows that all of the edges crossing a hyperplane $H$ have the same label $v$, and we call this a hyperplane of \emph{type $v$}. Since $A_\Gamma$ acts transitively on edges labeled $v$, a hyperplane is of type $v$ if and only if it is a translate of the standard hyperplane $H_v$.

For a vertex $v$ of the graph $\Gamma$ let $lk(v)$ denote the subgraph of $\Gamma$ spanned by the vertices adjacent to $v$ and let $st(v)$ denote the subgraph spanned by $v$ and $lk(v)$. Obviously, $A_{St(v)}$ is a join subgroup and $A_{lk(v)}$ is a subgroup of $A_{St(v)}$. We call such subgroup $A_{St(v)}$ \emph{star subgroup}. 
\end{rem}

\begin{lem}[Lemma 3.1 in \cite{MR2874959}]
\label{c3}
Let $H_1 = g_1H_v$ and $H_2 = g_2H_w$. Then
\begin{enumerate}
\item $H_1$ intersects $H_2$ if and only if $v$, $w$ commute and $g_1^{-1}g_2 \in A_{lk(v)}A_{lk(w)}$.
\item There is a hyperplane $H_3$ intersecting both $H_1$ and $H_2$ if and only if there is $u$ in $st(v)\cap st(w)$ such that $g_1^{-1}g_2 \in A_{lk(v)}A_{lk(u)}A_{lk(w)}$.
\end{enumerate}
\end{lem}

\subsection{Lower relative divergence of a right-angled Artin group with respect to a purely loxodromic subgroup}

In this section, we compute the lower relative divergence of a right-angled Artin group with respect to a purely loxodromic subgroup (i.e. a stable subgroup). We first prove the quadratic upper bound for the lower relative divergence of a right-angled Artin group with respect to a loxodromic subgroup and we need the following lemmas.

\begin{lem}
\label{llll1}
Let $\Gamma$ be a simplicial, finite, connected graph with the vertex set $S$ such that $\Gamma$ does not decompose as a nontrivial join. Let $H$ be a non-trivial, finitely generated purely loxodromic subgroup of $A_{\Gamma}$. There is a positive number $K$ such that for each element $g$ in $A_\Gamma$ and each pair of commuting generators $(s_1,s_2)$ in $S$ \[d_S(gs_1^{i}s_2^{j},H)\geq \frac{\abs{i}+\abs{j}}{K}-\abs{g}_S-1.\] 
\end{lem}

\begin{proof}
By Theorem \ref{th}, there is a positive integer $N$ such that for any reduced word $w$ representing $h\in H$, and any join subword $w'$ of $w$, we have $\ell(w')\leq N$. Let $K=N+1$ and we will prove that \[d_S(gs_1^{i}s_2^{j},H)\geq \frac{\abs{i}+\abs{j}}{K}-\abs{g}_S-1.\] 

Let $m=d_S(gs_1^{i}s_2^{j},H)$. Then there is an element $g_1$ in $A_\Gamma$ with $\abs{g_1}_S=m$ and $h$ in $H$ such that $h=gs_1^{i}s_2^{j}g_1$. Since $s_1^{i}s_2^{j}$ is an element in some join subgroup of $A_\Gamma$ and $\abs{g_1}_S=m$, then $h$ can be represented by a reduced word $w$ that is a product of at most $(\abs{g}_S+1+m)$ join subwords. Also, the length of each join subword of $w$ is bounded above by $N$. Therefore, the length of $w$ is bounded above by $N\bigl(\abs{g}_S+m+1\bigr)$. Also, \[\ell(w)\geq \abs{s_1^{i}s_2^{j}}_S-\abs{g_1}_S-\abs{g}_S\geq \abs{i}+\abs{j}-m-\abs{g}_S.\] 

This implies that \[\abs{i}+\abs{j}-m-\abs{g}_S\leq N\bigl(\abs{g}_S+m+1\bigr).\]

Therefore, \[d_S(gs_1^{i}s_2^{j},H)=m\geq \frac{\abs{i}+\abs{j}}{N+1}-\abs{g}_S-\frac{N}{N+1}\geq \frac{\abs{i}+\abs{j}}{K}-\abs{g}_S-1.\]
\end{proof}

\begin{lem}
\label{l'2}
Let $\Gamma$ be a simplicial, finite, connected graph with the vertex set $S$ such that $\Gamma$ does not decompose as a nontrivial join. Let $H$ be a on-trivial, finitely generated purely loxodromic subgroup of $A_{\Gamma}$ and $h$ an arbitrary element in $H$. There is a number $L\geq 1$ such that for each positive integer $m\geq L^2$ and generator $s$ in $S$ there is a path $\alpha$ outside the $(m/L-L)$--neighborhood of $H$ connecting $s^m$ and $hs^m$ with the length bounded above by $Lm$. 
\end{lem}

\begin{proof}
Let $M=\diam{\Gamma}$, $K$ the positive integer as in Lemma \ref{llll1} and $k=\abs{h}_S$. Let $L=2(k+1)(M+1)+K+k+M+1$. Choose a reduced word \[w=s_1^{\epsilon_1}s_2^{\epsilon_2}\cdots s_k^{\epsilon_k} \text{, where $s_i\in S$ and $\epsilon_i \in \{-1,1\}$}\] that represents element $h$. For each $i\in \{1,2,\cdots, k\}$ let $t_i$ be a vertex in $St(s_i)$ and $w_i=s_1^{\epsilon_1}s_2^{\epsilon_2}\cdots s_i^{\epsilon_i}$. Then the length of each word $w_i$ is bounded above by $k$, $w_{i+1}=w_is_i^{\epsilon_i}$, and $w_k=w$ that represents element $h$. 

We now construct a path $\alpha_0$ outside the $(m/L-L)$--neighborhood of $H$ connecting $s^m$ and $w_1t_1^m$ with the length bounded above by $2(M+1)m$. Since $M=\diam{\Gamma}$, we can choose positive integer $n\leq M$ and $n+1$ generators $u_0, u_1, \cdots, u_n$ in $S$ such that the following conditions hold:
\begin{enumerate}
\item $u_0=s$ and $u_n=t_1$.
\item $u_j$ and $u_{j+1}$ commutes where $j\in \{0, 1, 2, \cdots, n-1\}$.
\end{enumerate}
For each $j\in \{0, 1, 2, \cdots, n-1\}$ let $\beta_j$ be a path connecting $u_j^m$ and $u_{j+1}^m$ of length $2m$ with vertices \[u_j^m, u_j^mu_{j+1}, u_j^mu_{j+1}^2, \cdots,u_j^mu_{j+1}^m, u_j^{m-1}u_{j+1}^m,u_j^{m-2}u_{j+1}^m, \cdots,u_{j+1}^m.\] 
By Lemma \ref{llll1}, the above vertices must lie outside the $(m/K-1)$--neighborhood of $H$. Therefore, these vertices also lies outside the $(m/L-L)$--neighborhood of $H$. Therefore, $\beta_j$ is a path outside the $(m/L-L)$--neighborhood of $H$ connecting $u_j^m$ and $u_{j+1}^m$. Since $w_1t_1^m=s_1^{\epsilon_1}t_1^m=t_1^ms_1^{\epsilon_1}$, then we can connect $t_1^m$ and $w_1t_1^m$ by an edge $\beta_n$ labelled by $s_1$. Let $\alpha_0=\beta_0\cup\beta_1\cup\cdots\cup \beta_n$. Then, it is obvious that the path $\alpha_0$ outside the $(m/L-L)$--neighborhood of $H$ connecting $s^m$ and $w_1t_1^m$ with the length bounded above by $2(M+1)m$. 

By similar constructions as above, for each $i\in \{1,2,\cdots, k-1\}$ there is a path $\alpha_i$ outside the $(m/L-L)$--neighborhood of $H$ connecting $w_it_i^m$ and $w_{i+1}t_{i+1}^m$ with the length bounded above by $2(M+1)m$. We can also construct a path $\alpha_k$ outside the $(m/L-L)$--neighborhood of $H$ connecting $ht_k^m$ and $hs^m$ with the length bounded above by $2Mm$. Let $\alpha=\alpha_0\cup\alpha_1\cup\cdots\cup \alpha_k$. Then, it is obvious that the path $\alpha$ outside the $(m/L-L)$--neighborhood of $H$ connecting $s^m$ and $hs^m$ with the length bounded above by $2(k+1)(M+1)m$. By the choice of $L$ we observe that the length of $\alpha$ is also bounded above by $Lm$.
\end{proof}

\begin{prop}
\label{p'1}
Let $\Gamma$ be a simplicial, finite, connected graph with the vertex set $S$ such that $\Gamma$ does not decompose as a nontrivial join. Let $H$ be a non-trivial, finitely generated purely loxodromic subgroup of $A_{\Gamma}$. Then the lower relative divergence of $A_{\Gamma}$ with respect to $H$ is at most quadratic. 
\end{prop}

\begin{proof}
Let $h$ be an arbitrary group element in $H$ and $L\geq 1$ a constant as in Lemma \ref{l'2}. Since each cyclic subgroup in a $\CAT(0)$ group is undistorted (see Corollary III.$\Gamma$.4.8 and Theorem III.$\Gamma$.4.10 in \cite{MR1744486}), there is a positive integer $L_1$ such that \[\abs{h^k}_S\geq \frac{\abs{k}}{L_1}-L_1 \text{ for each integer $k$}.\] Let $\{\sigma^n_{\rho}\}$ be the lower relative divergence of $\tilde{X}^{(1)}_{\Gamma}$ with respect to $H$. We will prove that function $\sigma^n_{\rho}(r)$ is bounded above by some quadratic function for each $n\geq 2$ and $\rho \in (0,1]$.

Choose a positive integer $m\in \bigl[L(L+r),2L(L+r)\bigr]$ and a generator $s$ in $S$. Then, there is a path $\alpha_0$ outside the $(m/L-L)$--neighborhood of $H$ connecting $s^m$ and $hs^m$ with the length bounded above by $Lm$. It is obvious that the path $\alpha_0$ also lies outside the $r$--neighborhood of $H$ by the choice of $m$. Choose a positive integer $k$ which lies between $L_1\bigl(nr+8L(L+r)+L_1\bigr)$ and $L_1\bigl(nr+8L(L+r)+L_1+1\bigr)$. Let $\alpha=\alpha_0\cup h\alpha_0\cup h^2\alpha_0\cup\cdots\cup h^{k-1}\alpha_0$. Then, $\alpha$ is a path outside the $r$--neighborhood of $H$ connecting $s^m$, $h^ks^m$ with the length bounded above by $kLm$. By the choice of $k$ and $m$, the length of $\alpha$ is bounded above by $2L_1L^2(L+r)\bigl(nr+8L(L+r)+L_1+1\bigr)$.

Since $r\leq d_S(s^m, H)\leq m$, then there is a path $\gamma_1$ outside $N_r(H)$ connecting $s^m$ and some point $x\in \partial N_r(H)$ such that the length of $\gamma_1$ is bounded above by $m$. By the choice of $m$, the length of $\gamma_1$ is also bounded above by $2L(L+r)$. Similarly, there is a path $\gamma_2$ outside $N_r(H)$ connecting $h^ks^m$ and some point $y\in \partial N_r(H)$ such that the length of $\gamma_2$ is bounded above by $2L(L+r)$. Let $\bar{\alpha}=\gamma_1\cup\alpha\cup\gamma_2$ then $\bar{\alpha}$ is a path outside $N_r(H)$ connecting $x$, $y$ and the length of $\bar{\alpha}$ is bounded above by $2L_1L^2(L+r)\bigl(nr+8L(L+r)+L_1+1\bigr)+4L(L+r)$. Therefore, for each ~$\rho \in (0,1]$
\[d_{\rho r}(x,y)\leq 2L_1L^2(L+r)\bigl(nr+8L(L+r)+L_1+1\bigr)+4L(L+r).\]

Also, \begin{align*}d_S(x,y)&\geq d_S(s^m, h^ks^m)-d_S(s^m, x)-d_S(h^ks^m,y)\\&\geq \bigl(\abs{h^k}_S-2m\bigr)-2L(L+r)-2L(L+r)\geq \frac{k}{L_1}-L_1-8L(L+r)\\&\geq \bigl(nr+8L(L+r)\bigr)-8L(L+r)\geq nr.\end{align*}

 Thus, for each ~$\rho \in (0,1]$
\[\sigma_{\rho}^n(r)\leq 2L_1L^2(L+r)\bigl(nr+8L(L+r)+L_1+1\bigr)+4L(L+r).\] 

This implies that the lower relative divergence of $A_{\Gamma}$ with respect to $H$ is at most quadratic. 

\end{proof}

We now establish the quadratic lower bound for the lower relative divergence of a right-angled Artin group with respect to a loxodromic subgroup.
\begin{prop}
\label{p'2}
Let $\Gamma$ be a simplicial, finite, connected graph with the vertex set $S$ such that $\Gamma$ does not decompose as a nontrivial join. Let $H$ be a non-trivial, finitely generated purely loxodromic subgroup of $A_{\Gamma}$. Then the lower relative divergence of $A_{\Gamma}$ with respect to $H$ is at least quadratic. 
\end{prop}

\begin{proof}
By Corollary \ref{c1}, there is a positive integer $M$ such that every geodesic in the Cayley graph $\tilde{X}^{(1)}_{\Gamma}$ connecting points in $H$ lies in $M$--neighborhood of $H$. By Theorem \ref{th}, there is another positive integer $N$ such that for any reduced word $w$ representing $h\in H$, and any join subword $w'$ of $w$, we have $\ell(w')\leq N$. Let $\{\sigma^n_{\rho}\}$ be the lower relative divergence of $\tilde{X}^{(1)}_{\Gamma}$ with respect to $H$. We will prove that for each $n\geq 9$ and $\rho \in (0,1]$
\[\sigma^n_{\rho}(r)\geq \bigl(\frac{r-1}{3N+1}\bigr)\bigl(\rho r-3N)-2r \text{ for each $r>\frac{2M+3N+2}{\rho}$.}\]

Let $u$ and $v$ be an arbitrary pair of points in $\partial N_r(H)$ such that $d_r(u, v)<\infty$ and $d_S(u,v)\geq nr$. Let $\gamma$ be an arbitrary path that lies outside the $\rho r$--neighborhood of $H$ connecting $u$ and $v$. We will prove that the length of $\gamma$ is bounded below by $\bigl(\frac{r-1}{3N+1}\bigr)\bigl(\rho r-3N)-2r$.

Let $\gamma_1$ be a geodesic of length $r$ in $\tilde{X}^{(1)}_{\Gamma}$ connecting $u$ and some point $x$ in $H$. Let $\gamma_2$ be another geodesic of length $r$ in $\tilde{X}^{(1)}_{\Gamma}$ connecting $v$ and some point $y$ in $H$. Let $\alpha$ be a geodesic in $\tilde{X}^{(1)}_{\Gamma}$ connecting $x$ and $y$. Then $\alpha$ lies in the $M$--neighborhood of $H$. Choose a positive integer $m$ such that $r<(3N+1)m+1<2r$. 

Since $d_S(x,y)\geq d_S(u,v)-2r\geq (n-2)r\geq 7r$, there is a subpath $\alpha_1$ with length $(3N+1)m+1$ of $\alpha$ such that $\alpha_1\cap\bigl(B(x,2r)\cup B(y,2r)\bigr)=\emptyset$. Also, the lengths of $\gamma_1$ and $\gamma_2$ are both $r$. This implies that $(\gamma_1\cup \gamma_2)\cap N_r(\alpha_1)=\emptyset$. Since $\gamma \cap N_{\rho r}(H)=\emptyset$ and $\alpha_1\subset N_M(H)$, then $\gamma \cap N_{\rho r-M}(\alpha_1)=\emptyset$. Also, $\rho r-M>\rho r/2$. Thus, $\gamma \cap N_{\rho r/2}(\alpha_1)=\emptyset$. Let $\bar{\gamma}=\gamma_1\cup\gamma\cup\gamma_2$. Then, $\bar{\gamma} \cap N_{\rho r/2}(\alpha_1)=\emptyset$.

We assume that $\alpha_1=e_0w_1e_1w_2\cdots e_m w_m$, where each $e_i$ is an edge labelled by some generator $a_i$ in $S$, each $w_i$ is a subpath of $\alpha_1$ of length exactly $3N$. For each $i\in \{0,1,2,\cdots,m\}$ let $H_i$ be the hyperplane intersecting $e_i$ and $v_i$ a point in $H_i\cap \bar{\gamma}$. For each $i\in \{1,2,\cdots,m\}$ let $\beta_i$ be the subpath of $\bar{\gamma}$ connecting $v_{i-1}$ and $v_i$.

If there is some hyperplane of type $b$ that intersects two hyperplanes $H_{i-1}$ and $H_i$ for some $i\in \{1,2,\cdots,m\}$, then $e_{i-1}w_i$ corresponds to a word $w$ that represents an element in $A_{lk(a_{i-1})}A_{lk(b)}A_{lk(a_i)}$ by Lemma \ref{c3}. Also, the length of $w$ is $3N+1$. Then $w$ 
is a product of 3 join words and so one of the join words 
has length greater than $N$. This contradicts the choice of $N$. Therefore, no hyperplane intersects both $H_{i-1}$ and $H_i$ for each $i\in \{1,2,\cdots,m\}$. Also, $\bar{\gamma}$ lies outside the $(\rho r/2)$--neighborhood of $\alpha_1$. The number of hyperplanes that intersect $H_{i-1}\cup H_i$ is bounded below by $2(\rho r/2)=\rho r$. Moreover, each of these hyperplanes either intersects $w_i$ or $\beta_i$. We note that the number of hyperplanes that intersect $w_i$ is exactly $\ell(w_i)=3N$ since $w_i$ is a geodesic. This implies that the number of hyperplanes that intersect $\beta_i$ is bounded below by $\rho r-3N$. Thus,
\[\ell(\beta_i)\geq \rho r-3N.\] 

Therefore, \[\ell(\gamma)=\ell(\bar{\gamma})-2r\geq \sum_{i=1}^{m} \ell(\beta_i)-2r\geq m(\rho r-3N)-2r\geq \bigl(\frac{r-1}{3N+1}\bigr)\bigl(\rho r-3N)-2r.\]

Thus, \[\sigma^n_{\rho}(r)\geq \bigl(\frac{r-1}{3N+1}\bigr)\bigl(\rho r-3N)-2r \text{ for each $r>\frac{2M+3N+2}{\rho}$.}\] or the lower relative divergence of $G$ with respect to $H$ is at least quadratic.
\end{proof}

The following theorem is obtained from Propositions \ref{p'1} and \ref{p'2}.

\begin{thm}
\label{th'1}
Let $\Gamma$ be a simplicial, finite, connected graph with the vertex set $S$ such that $\Gamma$ does not decompose as a nontrivial join. Let $H$ be a non-trivial, finitely generated purely loxodromic subgroup of $A_{\Gamma}$. Then the lower relative divergence of $A_{\Gamma}$ with respect to $H$ is exactly quadratic. 
\end{thm}

We now construct an example of a right-angled Artin group $A_\Gamma$ together with a non-virtually cyclic subgroup $H$ such that the lower relative divergence of the pair $(A_\Gamma,H)$ is exactly quadratic. 

\begin{figure}
\begin{tikzpicture}[scale=1]


\draw (-3,0) node[circle,fill,inner sep=2pt](a){} -- (-1,0) node[circle,fill,inner sep=2pt](a){} -- (1,0) node[circle,fill,inner sep=2pt](a){} -- (3,0) node[circle,fill,inner sep=2pt](a){}; 

\node at (-3,-0.35) {$a$};\node at (-1,-0.35) {$b$};\node at (1,-0.35) {$c$};\node at (3,-0.35) {$d$};

\end{tikzpicture}
\caption{}
\label{a1}
\end{figure}

\begin{cor}
Let $\Gamma$ be the graph as in Figure \ref{a1}. Let $A_\Gamma$ be the associated right-angled Artin group and $H$ the subgroup of $A_\Gamma$ generated by two elements $ada$ and $dad$. Then $H$ is a free subgroup of rank 2, and the lower relative divergence of $A_\Gamma$ with respect to $H$ is exactly quadratic.
\end{cor}

\begin{proof}
This is obvious that the subgroup $K$ generated by two elements $a$ and $d$ is a free group of rank 2. Since $H$ is a subgroup of $K$ generated by two elements $ada$ and $dad$, $H$ is also a free group of rank 2. It is obvious that any join subword of a reduced word in $A_\Gamma$ representing a non-trivial group element in $H$ is an element in the set $\{a, a^{-1},a^2, a^{-2}, d, d^{-1}, d^2, d^{-2}\}$. Therefore, $H$ is a loxodromic subgroup by Theorem \ref{th}. This implies that the lower relative divergence of $A_\Gamma$ with respect to $H$ is exactly quadratic by the main theorem.
\end{proof}

\subsection{Main theorem}

\begin{lem}
\label{sonice}
Let $K$ be a simplicial, finite graph ($K$ is not necessarily connected) and $\Gamma$ the coned-off graph of $K$ with cone point $v$. If $H$ is a strongly quasiconvex subgroup of the right-angled Artin group $A_\Gamma$, then $H$ is trivial or $H$ has finite index in $A_\Gamma$.
\end{lem}

\begin{proof}
We assume that $H$ is nontrivial and we will prove that $H$ has finite index in $A_\Gamma$. Since $A_\Gamma=A_K \times \langle v \rangle$, it is sufficient to prove that $H\cap A_K$ has finite index in $A_K$ and $v^n$ lies in $H$ for some $n\neq 0$. We first assume for the contradiction that $H \cap \langle v \rangle$ is a trivial group. Since $v$ commutes with all vertices of $\Gamma$, $v$ commutes with all elements in $H$. Therefore, $v^nHv^{-n}=H$ for all $n$. Also, $v^iH\neq v^j H$ for $i\neq j$. This is a contradiction by Theorem \ref{lacloi1}. Therefore, $v^n$ lies in $H$ for some $n\neq 0$.

We now assume again for the contradiction that $H\cap A_K$ has infinite index in $A_K$. Then there is an infinite sequence $(g_n)$ of distinct elements in $A_K$ such that $g_i(H\cap A_K)\neq g_j(H\cap A_K)$ for $i\neq j$. Therefore, $g^{-1}_ig_j$ is not a group element in $H\cap A_K$ for $i\neq j$. This implies that $g^{-1}_ig_j$ is not a group element in $H$ since $g_i$ and $g_j$ are already elements of $A_K$. Thus, we see that $g_i H\neq g_j H$ for $i\neq j$. Since $v^n$ lies in $H$ and $v^n$ commutes with all $g_i$, then the infinite cyclic group generated by $v^n$ is a subgroup of each $g_iHg_i^{-1}$. Therefore, $\bigcap g_iHg_i^{-1}$ is an infinite subgroup. This again contradicts to Theorem \ref{lacloi1}. Therefore, $H\cap A_K$ has finite index in $A_K$. This implies that $H$ has finite index in $A_\Gamma$.
\end{proof}

\begin{lem}
\label{verynice}
Let $\Gamma$ be a simplicial, finite, connected graph and $K$ be an induced star subgraph $st(v)$. For each $g_1$ and $g_2$ in $A_\Gamma$ there is a finite sequence of conjugates of star subgroups $g_1A_Kg_1^{-1}=Q_0,Q_1,\cdots, Q_m=g_2A_Kg_2^{-1}$ such that $Q_{i-1}\cap Q_i$ is infinite for each $i\in \{1,2,\cdots,m\}$. 
\end{lem}

\begin{proof}
Let $S$ be the vertex set of $\Gamma$ and $n=\abs{g_1^{-1}g_2}_S$. We will prove the above lemma by induction on $n$. If $n=0$, then $g_1=g_2$. Therefore, the conclusion is true obviously. We now assume that $n=1$. Then there is a vertex $u$ such that $g_2=g_1 u^{\epsilon}$, where $\epsilon=1$ or $\epsilon =-1$. Since $\Gamma$ is finite and connected, there is a finite sequence of vertices of $u=u_0,u_1, u_2,\cdots, u_{\ell}=v$ such that $u_{i-1}$ is adjacent to $u_i$ for each $i\in\{1,2,3,\cdots,\ell\}$. Let $P_i=g_1A_{St(u_i)}g_1^{-1}$ and $P'_i=g_2A_{St(u_i)}g_2^{-1}$ for each $i\in\{0,1,2,3,\cdots,\ell\}$. This is not hard to check $P_{i-1}\cap P_i$ and $P'_{i-1}\cap P'_i$ are infinite for each $i\in\{1,2,3,\cdots,\ell\}$. Moreover, $P_0=P'_0$ because $g_2=g_1 u^{\epsilon}$, where $\epsilon=1$ or $\epsilon =-1$. Therefore, the conclusion is true for $n=1$ obviously.

Assume the conclusion is true for all $n\leq k$ for some $k\geq 1$. We will prove that the conclusion is true for $n=k+1$. In fact, if $\abs{g_1^{-1}g_2}_S=k+1$, then there is $g_3$ in $A_\Gamma$ such that $\abs{g_1^{-1}g_3}_S=k$ and $\abs{g_3^{-1}g_2}_S=1$. By the inductive hypothesis, there is a finite sequence of conjugates of star subgroups $g_1A_Kg_1^{-1}=L_0,L_1,\cdots, L_{m_1}=g_3A_Kg_3^{-1}$ such that $L_{i-1}\cap L_i$ is infinite for each $i\in \{1,2,\cdots,m_1\}$. Similarly, there is a finite sequence of conjugates of star subgroups $g_3A_Kg_3^{-1}=L'_0,L'_1,\cdots, L'_{m_2}=g_2A_Kg_2^{-1}$ such that $L'_{i-1}\cap L'_i$ is infinite for each $i\in \{1,2,\cdots,m_2\}$. Therefore, there is a finite sequence of conjugates of star subgroups $g_1A_Kg_1^{-1}=Q_0,Q_1,\cdots, Q_m=g_2A_Kg_2^{-1}$ such that $Q_{i-1}\cap Q_i$ is infinite for each $i\in \{1,2,\cdots,m\}$. This implies that the conclusion is true for $n=k+1$.
\end{proof}

\begin{prop}
\label{kaka}
Let $\Gamma$ be a simplicial, finite, connected graph such that $\Gamma$ does not decompose as a nontrivial join. Let $H$ be a non-trivial, infinite index subgroup of the right-angled Artin group $A_\Gamma$. If $H$ is a strongly quasiconvex subgroup, then $H$ is free.
\end{prop}

\begin{proof}
We are going to prove $H$ is \emph{star-free} (i.e. for each vertex $v$ of $\Gamma$ and $g\in A_\Gamma$ the subgroup $gHg^{-1}\cap A_{st(v)}$ is trivial) and then $H$ is a free group by Theorem 1.2 in \cite{KMT}. We first observe that $H$ is finitely generated and each conjugate of $H$ is also a strongly quasiconvex subgroup. We now prove that for each vertex $v$ of $\Gamma$ and $g\in A_\Gamma$ the subgroup $gHg^{-1}\cap A_{st(v)}$ is trivial. We assume for the contradiction that $g_0Hg_0^{-1}\cap A_{st(v)}$ is not trivial for some vertex $v$ and some $g_0\in A_\Gamma$. We claim that $gHg^{-1}\cap A_{st(v)}$ has finite index in $A_{st(v)}$ for all $g\in A_\Gamma$

In fact, since $g_0Hg_0^{-1}$ is a strongly quasiconvex subgroup and $A_{st(v)}$ is an undistorted subgroup, then $g_0Hg_0^{-1}\cap A_{st(v)}$ is a strongly quasiconvex subgroup of $A_{st(v)}$ by Proposition \ref{pp3}. Therefore, $g_0Hg_0^{-1}\cap A_{st(v)}$ has finite index in $A_{st(v)}$ by Lemma \ref{sonice}. 

We now prove that $gHg^{-1}\cap A_{st(v)}$ has finite index in $A_{st(v)}$ for all $g\in A_\Gamma$. By Lemma \ref{verynice}, there is a finite sequence of conjugates of star subgroups $g_0^{-1}A_{st(v)}g_0=Q_0,Q_1,\cdots, Q_m=g^{-1}A_{st(v)}g$ such that $Q_{i-1}\cap Q_i$ is infinite for each $i\in \{1,2,\cdots,m\}$. Since $g_0Hg_0^{-1}\cap A_{st(v)}$ has finite index in $A_{st(v)}$, $H \cap g_0^{-1}A_{st(v)}g_0$ has finite index in $Q_0=g_0^{-1}A_{st(v)}g_0$. Also, subgroup $Q_0\cap Q_1$ is infinite. Then, $H \cap Q_1$ is not trivial. Using a similar argument as above, we obtain $H \cap Q_1$ has finite index in $Q_1$. Repeating this process, we have $H\cap g^{-1}A_{st(v)}g$ has finite index in $g^{-1}A_{st(v)}g$. In other word, $gHg^{-1}\cap A_{st(v)}$ has finite index in $A_{st(v)}$. 

By Theorem \ref{lacloi1}, there is a number $n$ such that the intersection of any $(n+1)$ essentially distinct conjugates of $H$ is finite. Since $H$ has infinite index in $A_\Gamma$, there is $n+1$ distinct element $g_1, g_2,\cdots g_{n+1}$ such that $g_iH\neq g_j H$ for each $i\neq j$. Also, $g_i H g_i^{-1} \cap A_{st(v)}$ has finite index in $A_{st(v)}$ for each $i$. Then $(\cap g_i H g_i^{-1}) \bigcap A_{st(v)}$ also has finite index in $A_{st(v)}$. In particular, $\bigcap g_i H g_i^{-1}$ is infinite which is a contradiction. Therefore, for each vertex $v$ of $\Gamma$ and $g\in A_\Gamma$ the subgroup $gHg^{-1}\cap A_{st(v)}$ is trivial. This implies that $H$ is a free group by Theorem 1.2 in \cite{KMT}.
\end{proof}

We are now ready for the main theorem in this section.

\begin{thm}
Let $\Gamma$ be a simplicial, finite, connected graph such that $\Gamma$ does not decompose as a nontrivial join. Let $H$ be a non-trivial, infinite index subgroup of the right-angled Artin group $A_\Gamma$. Then the following are equivalent:
\begin{enumerate}
\item $H$ is strongly quasiconvex.
\item $H$ is stable.
\item The lower relative divergence of $A_\Gamma$ with respect to $H$ is quadratic.
\item The lower relative divergence of $A_\Gamma$ with respect to $H$ is completely super linear.
\end{enumerate}
\end{thm}

\begin{proof}
The implication ``$(1)\implies (2)$'' is obtained from Theorem \ref{th3} and Proposition \ref{kaka}. The implication ``$(2)\implies (3)$'' is deduced from Theorem \ref{th'1}. The implication ``$(3)\implies (4)$'' is trivial and the implication ``$(4)\implies (1)$'' follows Theorem \ref{th3}.
\end{proof}

\bibliographystyle{alpha}
\bibliography{Tran}
\end{document}